\documentclass[a4paper,reqno]{amsart}
\usepackage{amssymb, amsmath, amscd}
\usepackage[final]{graphicx}
\usepackage{wrapfig}
\def\PBG{(0,\infty)\times[0,\infty)}
\def\mapright#1{\smash{\mathop{\longrightarrow}\limits\sp{#1}}}
\newtheorem{theorem}{Theorem}[section]
\newtheorem{lemma}[theorem]{Lemma}
\newtheorem{remark}[theorem]{Remark}
\newtheorem{definition}[theorem]{Definition}

\makeatletter
 \@addtoreset{equation}{section}
\makeatother\begin{document}
\title[Moduli spaces]{Homogeneous affine surfaces: Moduli spaces}
\author{M. Brozos-V\'{a}zquez,\, E. Garc\'{i}a-R\'{i}o,  and P. Gilkey}
\address{MBV: Departamento de Matem\'aticas, Escola Polit\'ecnica Superior, Universidade da Coru\~na, Ferrol\; 15402, Spain}
\email{ miguel.brozos.vazquez@udc.gal}
\address{EGR: Faculty of Mathematics,
University of San\-tia\-go de Compostela,
Santiago de Compostela\; 15782, Spain}
\email{eduardo.garcia.rio@usc.es}
\address{PG: Mathematics Department, \; University of Oregon, \;\;
 Eugene \; OR 97403, \; USA}
\email{gilkey@uoregon.edu}
\begin{abstract}
We analyze the moduli space of non-flat homogeneous affine connections on surfaces. 
For Type~$\mathcal{A}$ surfaces, we write down complete sets of invariants that determine the
local isomorphism type depending on the rank of the Ricci tensor and examine the structure of the associated moduli space. 
For Type~$\mathcal{B}$ surfaces which are not Type~$\mathcal{A}$
we show the corresponding moduli space is a simply connected real analytic 4-dimensional manifold with
second Betti number equal to $1$.


\end{abstract}
\maketitle
\section{Introduction}

Moduli spaces, representing the space of solutions of geometric classifications,
are central in understanding the corresponding set of objects.
If the property under consideration is invariant by certain transformations,
then the associated moduli space should reflect the action of the
corresponding group of transformations. Moduli spaces have many
applications both in mathematics and physics and are closely linked
to the construction of invariants. 

A manifold ${M}$  equipped with a given structure (affine connection,
Riemannian metric, K\"ahler structure, etc.) is \textit{locally homogeneous}
if, given any pair of points of $M$, there is a germ of a diffeomorphism
(isomorphism, isometry, complex isometry, etc.) taking one point
into another and preserving the given structure. 

In this paper we shall study the moduli spaces of homogeneous affine
surfaces. We shall identify two such surfaces if they are affine equivalent
for some (orientation preserving) affine transformation.
Hence moduli spaces of homogeneous affine surfaces may be interpreted as 
providing a universal space of parameters to the problem of finding all
homogeneous affine connections that may carry some natural geometric
and topological structures.
The classification of non-flat locally homogeneous affine surfaces was obtained
by Opozda \cite{Op04} (see also \cite{AMK08, Du, G-SG, KVOp2} for related
work). They fall into three non-disjoint families: {\it Type}~$\mathcal{A}$,
{\it Type}~$\mathcal{B}$ and {\it Type}~$\mathcal{C}$ that we shall discuss
in more detail in Theorem~\ref{T1.1}. 
The Type~$\mathcal{C}$ surfaces correspond to
locally  homogeneous affine surfaces where the connection in question is
 the Levi-Civita connection of a metric with constant Gauss curvature. There
 are a finite number of such examples depending on the sign of the Gauss
 curvature and the signature of the underlying manifold. Thus
 we shall
restrict our attention to the study of the moduli spaces of Type~$\mathcal{A}$ and Type~$\mathcal{B}$ surfaces.

A slightly weaker definition of homegeneity was considered in \cite{DG}: a manifold is called \textit{quasihomogeneous} if it is locally homogeneous on a nontrivial open set, but not on the whole surface. There, a classification of torsion-free real-analytic quasihomogeneous affine connections on compact orientable surfaces was given. 

We adopt the following notational conventions.
An affine surface is a pair $\mathcal{M}=(M,\nabla)$ where $\nabla$
is a torsion free connection on the tangent bundle of a $2$-dimensional
manifold $M$. If $x=(x^1,x^2)$ is a system of local
coordinates on $M$, we adopt the {\it Einstein convention} and sum
over repeated indices to define the {\it Christoffel symbols} $\Gamma=\Gamma_{ij}{}^k$ by expanding:
$$\nabla_{\partial_{x^i}}\partial_{x^j}=\Gamma_{ij}{}^k\partial_{x^k}\,;$$
the condition that $\nabla$ is torsion free is equivalent to
the symmetry $\Gamma_{ij}{}^k=\Gamma_{ji}{}^k$ of the Christoffel symbols.
Let $R(X,Y):=\nabla_X\nabla_Y-\nabla_Y\nabla_X-\nabla_{[X,Y]}$
be the curvature operator and let $\rho(X,Y):=\operatorname{Tr}\{Z\rightarrow R(Z,X)Y\}$ be the Ricci tensor.

Affine connections on surfaces have been used to construct new examples of pseudo-Riemannian metrics exhibiting properties without Riemannian counterpart \cite{CGV10, CGGV09, De, KoSe}.
Flat connections play a distinguished role in many problems and the corresponding moduli spaces, together with their geometric structure, have been broadly investigated in the literature.

Recall that $\mathcal{M}$ is \emph{locally homogeneous} if given any 
pair of points of $M$, there is a germ of a diffeomorphism taking one point into another and preserving
$\nabla$. Suppose $\nabla$ is a connection on 
$M=\mathbb{R}^2$ so that the Christoffel symbols are
constant. The translation group $(x^1,x^2)\rightarrow(x^1+b^1,x^2+b^2)$ acts transitively on
$\mathbb{R}^2$ and preserves $\nabla$, so this geometry is homogeneous. Similarly, if
$\nabla$ is a connection on $M=\mathbb{R}^+\times\mathbb{R}$ so that the Christoffel symbols
take the form $\Gamma_{ij}{}^k=(x^1)^{-1}C_{ij}{}^k$ for $C_{ij}{}^k$ constant, then the ``$ax+b$"
group of transformations $(x^1,x^2)\rightarrow(ax^1,ax^2+b)$ for $a>0$ and $b\in\mathbb{R}$
acts transitively on $M$ and preserves $\nabla$
so this geometry is homogeneous as well. Finally, the Levi-Civita connection of a simply
connected complete Riemann surface of constant non-zero sectional curvature is homogeneous.
The following classification result of Opozda \cite{Op04} shows that these are the
only possible geometric models for a locally homogeneous affine surface:

\begin{theorem}\label{T1.1}
Let $\mathcal{M}=(M,\nabla)$ be a locally homogeneous affine surface which is not flat. Then at least one of the following
three possibilities holds which describe the local geometry:
\begin{itemize}
\item[$\mathcal{A}$)] There exists a coordinate atlas so the Christoffel symbols
$\Gamma_{ij}{}^k$ are constant.
\item[$\mathcal{B}$)] There exists a coordinate atlas so the Christoffel symbols have the form
$\Gamma_{ij}{}^k=(x^1)^{-1}C_{ij}{}^k$ for $C_{ij}{}^k$ constant and $x^1>0$.
\item[$\mathcal{C}$)] $\nabla$ is the Levi-Civita connection of a metric of constant sectional curvature.
\end{itemize}\end{theorem}

We assume $\mathcal{M}$ is not flat which for a surface means that the Ricci tensor $\rho$ does not vanish.
One says that $\mathcal{M}$ is {\it Type~$\mathcal{A}$}, {\it Type~$\mathcal{B}$} or
{\it Type~$\mathcal{C}$} depending on which possibility holds in Theorem~\ref{T1.1}. These are not exclusive. 
Although there are no surfaces which are both 
Type~$\mathcal{A}$ and Type~$\mathcal{C}$, there are surfaces which are both Type~$\mathcal{A}$
and Type~$\mathcal{B}$ and there are surfaces which are both Type~$\mathcal{B}$ and Type~$\mathcal{C}$.
The Ricci tensor is symmetric for Type~$\mathcal{A}$ and Type~$\mathcal{C}$ surfaces
but need not be symmetric for Type~$\mathcal{B}$ surfaces in general.
Let $\mathfrak{Z}_{\mathcal{A}}$ (resp. $\mathfrak{Z}_{\mathcal{B}}$) be the moduli space
of  isomorphism classes of germs of affine surfaces of Type~$\mathcal{A}$ (resp. Type~$\mathcal{B}$).

In previous work,  Kowalski and Vlasek \cite{KV03} showed
that the moduli space of Type $\mathcal{A}$ connections is at most 
two-dimensional and that the moduli space of Type $\mathcal{B}$ connections
is at most four dimensional. Our purpose is to understand these
moduli spaces in more
detail. We shall investigate their underlying structure and construct some new
invariants which are not of Weyl type. For Type~$\mathcal{A}$ surfaces, 
the invariants we shall construct completely characterize the homogeneous
affine structure up to affine equivalence. We now summarize our main results.

\subsection{Type~$\mathcal{A}$ surfaces}

Let $\mathcal{Z}_{\mathcal{A}}$ be the set of Christoffel symbols
$\Gamma\in\mathbb{R}^6$ defining a Type~$\mathcal{A}$ structure.
Thus, let
$$
\Gamma:=(\Gamma_{11}{}^1,\Gamma_{11}{}^2,
\Gamma_{12}{}^1=\Gamma_{21}{}^1,\Gamma_{12}{}^2=\Gamma_{21}{}^2,
\Gamma_{22}{}^1,\Gamma_{22}{}^2)\in\mathbb{R}^6
$$
determine a translation invariant
homogeneous affine structure on $\mathbb{R}^2$. The Ricci tensor is given by
\begin{equation}\label{E1.a}
\begin{array}{l}
\rho_{11}=\Gamma_{12}{}^2 (\Gamma_{11}{}^1-\Gamma_{12}{}^2)
+\Gamma_{11}{}^2 (\Gamma_{22}{}^2-\Gamma_{12}{}^1),\\[0.03in]
\rho_{12}=\rho_{21}=\Gamma_{12}{}^1 \Gamma_{12}{}^2
-\Gamma_{11}{}^2 \Gamma_{22}{}^1,\\[0.03in]
\rho_{22}=\Gamma_{22}{}^1 (\Gamma_{11}{}^1-\Gamma_{12}{}^2)
+\Gamma_{12}{}^1 (\Gamma_{22}{}^2-\Gamma_{12}{}^1)\,.
\end{array}\end{equation}
Thus, in particular, $\rho$ is symmetric if $\mathcal{M}$ is of Type~$\mathcal{A}$.
If $\rho=0$, then $\mathcal{M}$
 is flat. As we shall assume $\mathcal{M}$ is not flat, either $\operatorname{Rank}\{\rho\}=1$ or
 $\operatorname{Rank}\{\rho\}=2$.

\subsubsection{Type~$\mathcal{A}$ surfaces with $\operatorname{Rank}\{\rho\}=1$.}
Let $\mathcal{M}=(M,\nabla)$ be an affine surface of Type~$\mathcal{A}$ with
$\operatorname{Rank}\{\rho\}=1$. We showed previously
\cite{BGGP16} that $\mathcal{M}$ is both of Type~$\mathcal{A}$ and of Type~$\mathcal{B}$ implies
$\operatorname{Rank}\{\rho\}=1$.
Let $P$ be an arbitrary point of $M$;
 the particular point is irrelevant since $\mathcal{M}$ is locally homogeneous. Choose
 $X\in T_PM$ so that $\rho(X,X)\ne0$. Let $\nabla\rho$ be the covariant derivative of the
 Ricci tensor. Define
 \begin{eqnarray*}
&& \alpha_X(\mathcal{M}):=\nabla\rho(X,X;X)^2\cdot\rho(X,X)^{-3}\text{ and }
 \epsilon_X(\mathcal{M}):=\operatorname{sign}(\rho(X,X))=\pm1\,.
 \end{eqnarray*}
The following result \cite{BGGP16} shows that $\alpha_X$ and
$\epsilon_X$
give a complete system of invariants for such surfaces:
 \begin{theorem}\label{T1.2}
Let $\mathcal{M}$ be an affine
 surface of Type~$\mathcal{A}$ with $\operatorname{Rank}\{\rho\}=1$. Then
 $\alpha_X$ and $\epsilon_X$ are independent of $X$ and define affine invariants 
 $\alpha(\mathcal{M})$ and $\epsilon(\mathcal{M})$.
\begin{enumerate}
\item $\mathcal{M}$ is symmetric (i.e. $\nabla\rho=0$) if and only if $\alpha(\mathcal{M})=0$.
 \item If $\alpha(\mathcal{M})\ne0$, then $\mathcal{M}$ is also of Type~$\mathcal{B}$ if and only if $\alpha(M)\notin(0,16)$.
 \item If $\alpha(\mathcal{M})=0$, then $\mathcal{M}$ is also of  Type~$\mathcal{B}$ if and only if $\epsilon<0$.
 \item If $\tilde{\mathcal{M}}$ is another Type~$\mathcal{A}$ surface
 with $\operatorname{Rank}(\tilde\rho)=1$,
 with $\alpha(\tilde{\mathcal{M}})=\alpha(\mathcal{M})$, and with
 $\epsilon(\tilde{\mathcal{M}})=\epsilon(\mathcal{M})$, then $\tilde{\mathcal{M}}$ is locally isomorphic
 to $\mathcal{M}$.
 \end{enumerate}
 \end{theorem}

\subsubsection{Type~$\mathcal{A}$ surfaces with $\operatorname{Rank}\{\rho\}=2$.}
If $\mathcal{M}$ is a surface of Type~$\mathcal{A}$ with $\operatorname{Rank}\{\rho\}=2$,
then $\mathcal{M}$ is not of Type~$\mathcal{B}$.
The structure group is the affine group
\begin{eqnarray*}
&&\mathfrak{F}:=\{T\in\operatorname{Diff}(\mathbb{R}^2):T(x^1,x^2)=(a_1^1x^1+a_2^1x^2+b^1,a_1^2x^1+a_2^2x^2+b^2)\\
&&\qquad\qquad\text{ where }
(a_i^j)\in\operatorname{GL}(2,\mathbb{R})\}\,.
\end{eqnarray*}
If $T\in\mathfrak{F}$ then $T_\ast \partial_{x^i}=\tilde a_j^i \partial_{x^j}$, where $\tilde a^i_j$ are the components of the inverse matrix so
$\sum_j a_i^j\tilde a_j^k=\delta_i^k$
is the Kronecker symbol. If $\Gamma\in\mathbb{R}^6$, let $T^*\Gamma$ be the associated Christoffel symbols
in the new coordinate system; these are constant and given in the form:
$$
(T^*\Gamma)_{ij}{}^k=
\sum_{\alpha,\beta,\gamma} \tilde a_i^\alpha \tilde a_j^\beta  a^k_\gamma\Gamma_{\alpha\beta}{}^\gamma
$$
The translations play no role and we obtain a representation of
$\operatorname{GL}(2,\mathbb{R})$ on $\mathbb{R}^6$.
Fix a Type~$\mathcal{A}$ coordinate atlas $\{(\mathcal{O}_\alpha,\phi_\alpha)\}$ on $M$. Here
$\{\mathcal{O}_\alpha\}$ forms an open cover of $M$ so the Christoffel symbols are constant
and the maps $\phi_\alpha$ are diffeomorphisms from $\mathcal{O}_\alpha$
to an open subset of $\mathbb{R}^2$. Sum over repeated indices to define:
\begin{equation}\label{E1.b}
\rho_{ij}^3:=\Gamma_{ik}{}^l\Gamma_{jl}{}^k,\quad
\psi_3:=\operatorname{Tr}_\rho\{\rho^3\}=\rho^{ij}\rho_{ij}^3,\quad
\Psi_3:=\det(\rho^3)/\det(\rho).
\end{equation}
The transition functions of a Type~$\mathcal{A}$ coordinate atlas on $\mathcal{M}$ belong to the affine group $\mathfrak{F}$, and
$\psi_3$ and $\Psi_3$ are independent of the local Type~$\mathcal{A}$
coordinates chosen and thus are affine invariants of $\mathcal{M}$ (see Lemma~\ref{L1.3}).
We can use the local homogeneity to choose the coordinate atlas so that 
${}^\alpha\Gamma=\Gamma$ is independent of the particular chart 
$\mathcal{O}_\alpha$. Let $\rho=\rho_\Gamma$; this is a fixed non-degenerate
bilinear form. The transition functions then belong to the orthogonal affine group
$\mathfrak{F}^\rho:=\{T\in\mathfrak{F}:T^*\rho=\rho\}$, i.e. 
$(a_{ij})$ is in the orthogonal group determined by $\rho$.

\begin{definition}\rm
Let $\mathcal{Z}_+$ (resp. $\mathcal{Z}_0$ or $\mathcal{Z}_-$)
 be the set of Christoffel symbols
$\Gamma\in\mathbb{R}^6$ defining a Type~$\mathcal{A}$ structure such that the Ricci tensor 
is positive definite (resp. indefinite or negative definite). Let $\mathfrak{Z}_\varepsilon$  for
$\varepsilon=+,0,-$ be the associated moduli space. 
\end{definition}
It will follow from Lemma~\ref{L1.3} below that
$\mathfrak{Z}_\varepsilon=\mathcal{Z}_\varepsilon/\operatorname{GL}(2,\mathbb{R})$. Since $\mathcal{Z}_\varepsilon$ is an open subset of $\mathbb{R}^6$ and since
$\operatorname{GL}(2,\mathbb{R})$ is a 4-dimensional Lie group, one expects
that $\mathfrak{Z}_\varepsilon$ will be 2-dimensional. Kowalski and Vlasek \cite{KV03}
have shown that this is the case.

Let $\Theta_\varepsilon:=(\psi_3,\Psi_3)$ on $\mathcal{Z}_\varepsilon$; this real analytic map
extends to a map from the moduli space $\mathfrak{Z}_\varepsilon$ to $\mathbb{R}^2$.
The following result makes this very explicit and gives
a complete set of real analytic invariants:

 \begin{theorem}\label{T1.5} $\Theta_\varepsilon$ is a 1-1 map from 
 $\mathfrak{Z}_\varepsilon$ to a closed simply connected subset of $\mathbb{R}^2$.
\end{theorem}

We discuss the range of $\Theta_\varepsilon$ as follows. Consider the two curves
\begin{equation}\label{E1.c}
\begin{array}{l}
\textstyle\sigma_+(t):=(4 t^2+\frac{1}{t^2}+2,4 t^4+4 t^2+2),
\\[0.05in]
\textstyle\sigma_-(t):=(-4 t^2-\frac{1}{t^2}+2,4 t^4-4 t^2+2).
\end{array}\end{equation}
The curve $\sigma_+$ is smooth; the curve $\sigma_-$ has a cusp at $(-2,1)$ when $t=\frac1{\sqrt2}$.
These two curves divide the plane into 3 open regions $\mathfrak{O}_-$, $\mathfrak{O}_0$, and $\mathfrak{O}_+$
where $\mathfrak{O}_-$ lies in the second quadrant and is bounded on the right
by $\sigma_-$, $\mathfrak{O}_+$ lies in the first quadrant, and is bounded on the left by $\sigma_+$
and $\mathfrak{O}_0$ lies in between and is bounded on the left by $\sigma_-$ and on the right by $\sigma_+$.
Let $\mathfrak{C}_\varepsilon$ be the closure of $\mathfrak{O}_\varepsilon$; $\mathfrak{C}_-=\mathfrak{O}_-\cup\operatorname{range}(\sigma_-)$, $\mathfrak{C}_+=\mathfrak{O}_+\cup\operatorname{range}(\sigma_+)$, and
$\mathfrak{C}_0=\mathfrak{O}_0\cup\operatorname{range}(\sigma_-)\cup\operatorname{range}(\sigma_+)$.
\begin{theorem}\label{T1.6}
Adopt the notation established above. Then 
$\Theta_\varepsilon(\mathcal{Z}_\varepsilon)=\mathfrak{C}_\varepsilon$.
\end{theorem}

In Fig. 1.1 below, we show the two curves $\sigma_\pm$ which bound 
the images of the moduli spaces $\Theta_-(\mathfrak{Z}_-)$, 
$\Theta_0(\mathfrak{Z}_0)$, and $\Theta_+(\mathfrak{Z}_+)$
\smallbreak\hglue 3cm\vbox{\includegraphics[height=3.5cm,keepaspectratio=true]{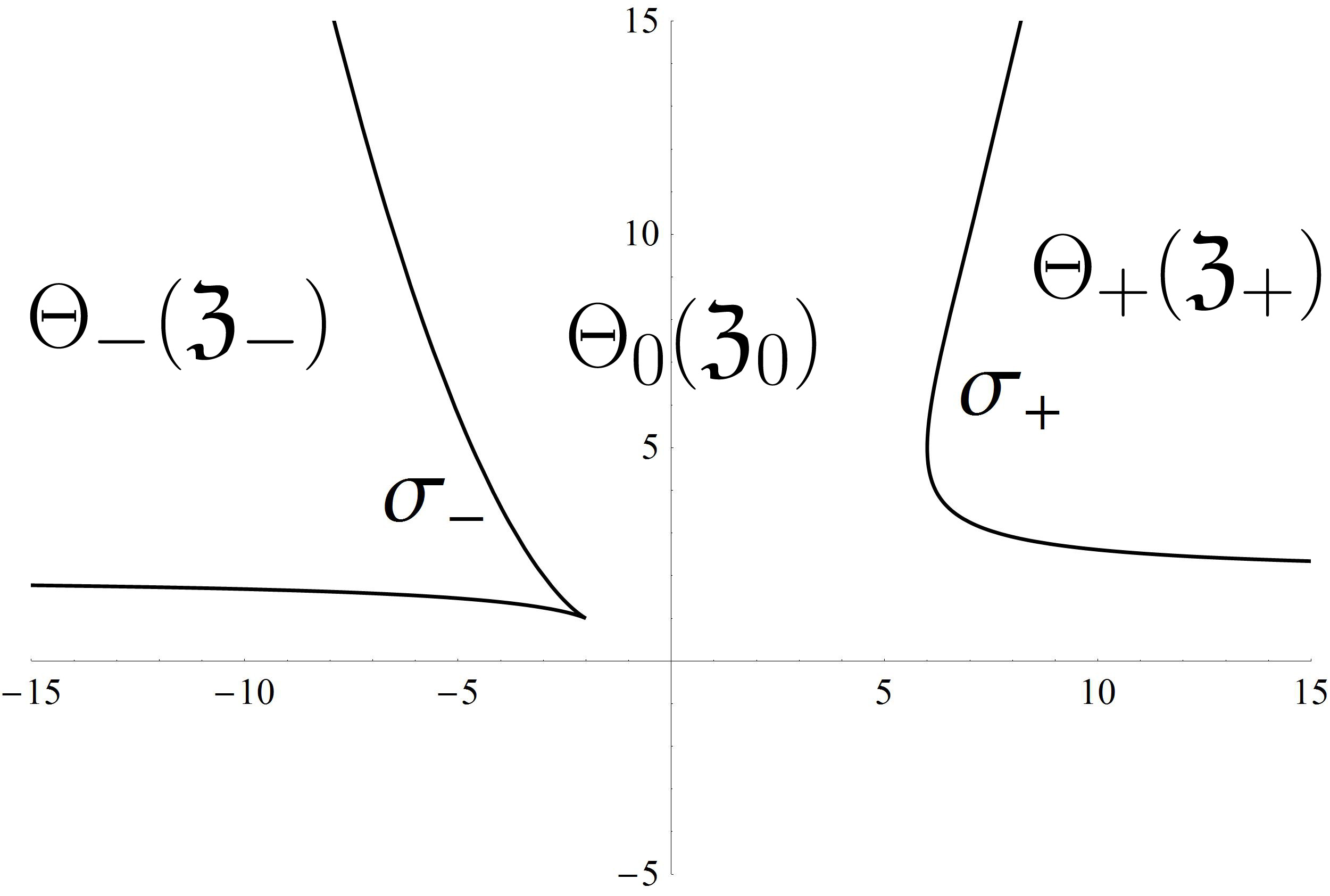}
 \vglue -2cm\hglue 6 cm$\Leftarrow$ Fig.~1.1\vglue 1cm}
\medbreak\noindent
Note that although
$\Theta_\varepsilon$ is 1-1 on $\mathfrak{Z}_\varepsilon$, we have
that $\Theta_+(\mathfrak{Z}_+)$ intersects $\Theta_0(\mathfrak{Z}_0)$
along their common boundary $\sigma_+$ and
that $\Theta_-(\mathfrak{Z}_-)$ intersects $\Theta_0(\mathfrak{Z}_0)$
along their common boundary $\sigma_-$. This does not mean that $\mathfrak{Z}_+$ (resp. $\mathfrak{Z}_-$) intersects
$\mathfrak{Z}_0$ nor does it mean that $\Theta_-$, or $\Theta_0$, or
$\Theta_+$ are not 1-1
on their respective domains.

We have worked in the category of unoriented surfaces to set
$\mathfrak{Z}_\varepsilon=\mathcal{Z}_\varepsilon/\operatorname{GL}(2,\mathbb{R})$. 
If we work in the category
of oriented surfaces, we may obtain a corresponding moduli space 
$\mathfrak{Z}_\varepsilon^+=\mathcal{Z}_\varepsilon/\operatorname{GL}^+(2,\mathbb{R})$ of oriented
affine surfaces of Type~$\mathcal{A}$ where 
$\operatorname{GL}^+(2,\mathbb{R})$ denotes the group of matrices with
positive determinant. If $\Gamma\in\mathcal{Z}_\varepsilon$, let
\begin{eqnarray*}
&&\mathcal{G}^+(\Gamma):=\{T\in\operatorname{GL}^+(2,\mathbb{R}):T^*\Gamma=\Gamma\},\\
&&\mathcal{G}(\Gamma):=\{T\in\operatorname{GL}(2,\mathbb{R}):T^*\Gamma=\Gamma\}
\end{eqnarray*}
be the isotropy subgroups of the actions. 
Let $\text{dvol}:=|\det(\rho_{ij})|^{1/2}dx^1\wedge dx^2$ be the oriented volume form. Extend $\rho$
to an inner product on tensors of all type and sum over repeated indices to define an invariant $\chi$
which is sensitive to the orientation:
\begin{equation}\label{E1.d}
\chi(\Gamma):=\rho(\Gamma_{ab}{}^b\Gamma_{ij}{}^k\rho^3_{kl}\rho^{ij}dx^a\wedge dx^l,\text{dvol})\,.
\end{equation}
Since this is given in terms of contraction of indices, it is an affine invariant
by Lemma \ref{L1.3}.

The cusp point $\Gamma_{\text{csp}}\in\mathfrak{Z}_-$ which satisfies
$\Theta_-(\Gamma_{\text{csp}})=(-2,1)$ is singular and corresponds to the structure $\sigma_-(\frac1{\sqrt2})$;
$$[\Gamma_{\text{csp}}]=[
\{\Gamma_{11}{}^1=-\textstyle\frac1{\sqrt 2},\,\,\Gamma_{11}{}^2=0,\,\,\Gamma_{12}{}^1=0,\,\,\Gamma_{12}{}^2=\textstyle\frac1{\sqrt 2},\,\,\Gamma_{22}{}^1=\textstyle\frac1{\sqrt 2},\,\,\Gamma_{22}{}^2=0\}]\,.
$$

Theorem~\ref{T1.7} shows the sequence 
$\operatorname{GL}^+(2,\mathbb{R})\rightarrow\mathcal{Z}_\varepsilon\rightarrow\mathfrak{Z}_\varepsilon^+$
is a principle bundle except at the cusp point of $\mathfrak{C}_-$.
Furthermore, a structure $\Gamma$ with a given orientation
is isomorphic to $\Gamma$ with the opposite orientation if and only if $\Gamma$ belongs to the boundary of
one of the regions. This shows that $\mathfrak{Z}_\varepsilon^+$ is the double of $\mathfrak{Z}_\varepsilon$
where we glue two copies of $\mathfrak{Z}_\varepsilon$ together along the boundary as described above
in Fig.~1.1.

\begin{theorem}\label{T1.7}
\ \begin{enumerate}
\item Let $\Gamma\in\mathcal{Z}_+$. Then $|\mathcal{G}^+(\Gamma)|=1$. 
\begin{enumerate}
\item If $\Theta_+\Gamma\in\operatorname{bd}(\mathfrak{C}_+)$, then $|\mathcal{G}(\Gamma)|=2$.
\item If $\Theta_+\Gamma\notin\operatorname{bd}(\mathfrak{C}_+)$, then $|\mathcal{G}(\Gamma)|=1$.
\end{enumerate}
\item Let $\Gamma\in\mathcal{Z}_-$.
\begin{enumerate}
\item If $\Theta_-(\Gamma)=(-2,1)$, then
$|\mathcal{G}^+(\Gamma)|=3$ and $|\mathcal{G}(\Gamma)|=6$.
\item If $\Theta_-(\Gamma)\in\operatorname{bd}(\mathfrak{C}_+)-\{(-2,1)\}$, then 
$|\mathcal{G}^+(\Gamma)|=1$ and $|\mathcal{G}(\Gamma)|=2$. 
\item If $\Theta_-(\Gamma)\notin\operatorname{bd}(\mathfrak{C}_+)$, then
$|\mathcal{G}^+(\Gamma)|=1$ and $|\mathcal{G}(\Gamma)|=1$. 
\end{enumerate}
\item Let $\Gamma\in\mathcal{Z}_0$. Then $|\mathcal{G}^+(\Gamma)|=1$. 
\begin{enumerate}
\item If $\Theta_0\Gamma\in\operatorname{bd}(\mathfrak{C}_0)$, then $|\mathcal{G}(\Gamma)|=2$.
\item If $\Theta_0\Gamma\notin\operatorname{bd}(\mathfrak{C}_0)$, then $|\mathcal{G}(\Gamma)|=1$.
\end{enumerate}
\item The map $\tilde\Theta_\varepsilon:=(\psi_3,\Psi_3,\chi)$ is a 
$1-1$ embedding of $\mathfrak{Z}_\varepsilon$
into $\mathbb{R}^3$ for $\varepsilon=-,0,+$. The embedding $\Theta_\varepsilon$
is the projection of $\tilde\Theta_\varepsilon$ on the horizontal plane. 
It is a ``fold" map that
is $1-1$ except along the boundary of $\Theta_\varepsilon(\mathfrak{Z}_\varepsilon)$
which is locus where $\chi=0$.
\end{enumerate}
\end{theorem}

\subsection{Type~$\mathcal{B}$ surfaces}
 
We now examine the moduli space of Type~$\mathcal{B}$ surfaces
$\mathfrak{Z}_{\mathcal{B}}$.
Let $\mathcal{M}_{C}:=(\mathbb{R}^+\times\mathbb{R},\Gamma=(x^1)^{-1}C) $ 
be the surface of 
Type~$\mathcal{B}$ determined by $C\in\mathbb{R}^6$. 
Let $\kappa(\mathcal{M}_C)$ be the dimension of the space of affine Killing vector fields on 
$\mathcal{M}_C$. We established the following result previously \cite{BGGP16}:
\begin{theorem}
If $\mathcal{M}_C$ is not flat, then $2\le\kappa\le4$. Furthermore, $\mathcal{M}$ is also 
Type~$\mathcal{A}$ if and only if $\kappa=4$. 
\end{theorem}
Since the case that $\mathcal{M}_C$ is also of Type~$\mathcal{A}$ has already been dealt with, 
we shall assume henceforth
that $2\le\kappa\le 3$. Let 
$$
\mathcal{Z}_{23\mathcal{B}}:=
\{C\in\mathbb{R}^6:2\le\kappa(\mathcal{M}_C)\le3\}\subset\mathbb{R}^6
$$
and let $\mathfrak{Z}_{23\mathcal{B}}$ be the associated moduli space. 

Let $C\in\mathcal{Z}_{23\mathcal{B}}$ define a homogeneous structure on $\mathbb{R}^+\times\mathbb{R}$
with $2\le\kappa(\mathcal{M}_C)\le3$.
Use $\pm dx^1\wedge dx^2$ to define the corresponding oriented surface $\mathcal{M}_C^\pm$. Let 
$\mathfrak{Z}_{23\mathcal{B}}^+$ be the associated moduli space of oriented surfaces with 
$2\le\kappa(\mathcal{M}_C^\pm)\le 3$. Let
$$
\begin{array}{l}
\mathfrak{G}:=\{T:
T(x^1,x^2)=(ax^1,bx^1+cx^2+d)\text{ for }a>0\text{ and }c\ne0\},\\[0.05in]
\mathfrak{I}:=\{T_{b,c}\in\mathfrak{G}:T_{b,c}(x^1,x^2)=(x^1,bx^1+cx^2)\text{ for }c\ne0\},\\[0.05in]
\mathfrak{I}^+:=\{T_{b,c}\in\mathfrak{I}\text{ for }c>0\}\,.
\end{array}$$
Our previous discussion in \cite{BGGP16} permits us to identify 
$$
\mathfrak{Z}_{23\mathcal{B}}=\mathcal{Z}_{23\mathcal{B}}/\mathfrak{I}\text{ and }
\mathfrak{Z}_{23\mathcal{B}}^+=\mathcal{Z}_{23\mathcal{B}}/\mathfrak{I}^+\,.
$$
This is a non-trivial assertion if $\kappa(\mathcal{M})=3$ as there are non-linear affine transformations.
However, they play no role in defining the affine isomorphism type.
We define several invariant tensors on $\mathfrak{Z}_{23\mathcal{B}}$ that allow us to introduce coordinates and to obtain the structure of the moduli space $\mathfrak{Z}_{23\mathcal{B}}$:   
\begin{eqnarray*}
&&\rho_1:=\textstyle\frac1{x^1}
\{\Gamma_{12}{}^2dx^1\otimes dx^1+\Gamma_{22}{}^2dx^1\otimes dx^2
-\Gamma_{12}{}^1dx^2\otimes dx^1-\Gamma_{22}{}^1dx^2\otimes dx^2\},\nonumber\\
&&\rho_2:=\Gamma_{ij}{}^k\Gamma_{kl}{}^ldx^i\otimes dx^j,\quad
\rho_3:=\Gamma_{ik}{}^l\Gamma_{jl}{}^kdx^i\otimes dx^j,\quad
\rho_0:=\Gamma_{ij}{}^jdx^i,\\
&&\rho_4:=\Gamma_{ij}{}^kC_{ak}{}^jC_{bc}{}^idx^a\otimes dx^b\otimes dx^c\,.\nonumber
\end{eqnarray*}
We shall show in Lemma \ref{L5.1} that these tensors are invariantly defined.

In Lemma~\ref{L5.2}, we will examine the
isotropy subgroup of the natural action of $\mathfrak{I}^+$ on $\mathcal{K}_{23B}$, we will
show that $\mathcal{K}_{23\mathcal{B}}$ is a 4-dimensional real analytic manifold, and we will
prove the natural projection $\pi^+$ from $\mathcal{K}_{23\mathcal{B}}$ to $\mathfrak{K}_{23\mathcal{B}}^+$ is a real-analytic $\mathfrak{J}^+$ fiber bundle.
In Lemma~\ref{L5.3}, we show that $\mathfrak{Z}_{23\mathcal{B}}^+$ is simply connected and
that the second Betti number is 2. Our main result concerning Type~$\mathcal{B}$ structures is summarized in the following:

\begin{theorem}\label{T1.10}
The moduli space $\mathfrak{Z}_{23\mathcal{B}}$ is a $4$-dimensional real analytic manifold,
the natural projection from $\mathcal{Z}_{23\mathcal{B}}$ to $\mathfrak{Z}_{23\mathcal{B}}$
is real analytic,
$\mathfrak{Z}_{23\mathcal{B}}$ is simply connected,
and the second Betti number of $\mathfrak{Z}_{23\mathcal{B}}$  is 1.
\end{theorem}

In Theorem~\ref{T5.4} we complete the proof of Theorem~\ref{T1.10}
by giving $\mathfrak{Z}_{23\mathcal{B}}$ a real analytic structure, by
showing that the projection from $\mathfrak{Z}_{23\mathcal{B}}^+$ to 
$\mathfrak{Z}_{23\mathcal{B}}$
is a $\mathbb{Z}_2$ branched cover where the ramification set is a real analytic sub-manifold of
co-dimension 2, and by demonstrating that $\mathfrak{Z}_{23\mathcal{B}}^+$ is simply connected and
has second Betti number equal to 1.

\subsection{Outline of the paper} In Section~\ref{S2}, we prove
Theorem~\ref{T1.5} in the positive and negative definite settings. In Section~\ref{S3},
we prove Theorem~\ref{T1.5} in the indefinite setting. In Section~\ref{S4} we prove
 Theorem~\ref{T1.6} and Theorem~\ref{T1.7} and in
Section~\ref{S5} we prove Theorem~\ref{T1.10}.

\section{The moduli spaces $\mathfrak{Z}_\pm$}\label{S2}

Recall from \cite{BGGP16} that if $\mathcal{M}$ is a surface of Type~$\mathcal{A}$ with $\operatorname{Rank}\{\rho\}=2$,
then $\mathcal{M}$ is not of Type~$\mathcal{B}$. Moreover the structure group is the affine group
\begin{eqnarray*}
&&\mathfrak{F}:=\{T\in\operatorname{Diff}(\mathbb{R}^2):T(x^1,x^2)=(a_1^1x^1+a_2^1x^2+b^1,a_1^2x^1+a_2^2x^2+b^2)\\
&&\qquad\qquad\text{ where }
(a_i^j)\in\operatorname{GL}(2,\mathbb{R})\}\,.
\end{eqnarray*}
Further observe that since any Type~$\mathcal{A}$ surface is projectively flat (see \cite{CGV10}), 
if $\operatorname{Rank}\{\rho\}=2$, then the Ricci tensor defines a metric so that $(\rho,\nabla)$ is a Codazzi pair, also called a statistical structure (see \cite{Op15} and references therein). 

\begin{lemma}\label{L1.3}
Let $\mathcal{M}=(M,\nabla)$ be a Type~$\mathcal{A}$ affine surface with $\operatorname{Rank}(\rho)=2$.
Then transition functions of a Type~$\mathcal{A}$ coordinate atlas on $\mathcal{M}$ belong to the affine group $\mathfrak{F}$, and
$\psi_3$ and $\Psi_3$ defined in Equation \eqref{E1.b} are independent of the local Type~$\mathcal{A}$
coordinates chosen and thus are affine invariants of $\mathcal{M}$.
\end{lemma}

\begin{proof}
Cover $M$ by Type~$\mathcal{A}$ coordinate charts $(\mathcal{O}_\alpha,\phi_\alpha)$
so ${}^\alpha\Gamma$ is constant. The transition functions $\phi_{\alpha\beta}$ 
then are local diffeomorphisms
of $\mathbb{R}^2$ so that $\phi_{\alpha\beta}^*\{{}^\beta\rho\}={}^\alpha\rho$.
The Ricci tensors ${}^\alpha\rho$
and ${}^\beta\rho$ define flat pseudo-Riemannian metrics. 
This implies $d\phi_{\alpha\beta}$ is constant and, consequently, $\phi_{\alpha\beta}$ is
an affine linear transformation. Since contracting upper and lower indices is
invariant under the action of $\operatorname{GL}(2,\mathbb{R})$, $\psi_3$ and $\Psi_3$ are invariants.
\end{proof}

We begin by parametrizing the moduli spaces $\mathfrak{Z}_\pm$.
 \begin{definition}\label{D2.1}\rm
Let $\mathcal{S}:=\PBG$. If $(x,y)\in\mathbb{R}^2$ with $x\ne0$, define: 
\smallbreak
$\Gamma_\pm(x,y):=\{
\Gamma_{11}{}^1=\textstyle x\pm\frac1x,
\Gamma_{11}{}^2=0,\Gamma_{12}{}^1=0,\Gamma_{12}{}^2=x,
\Gamma_{22}{}^1=x,\Gamma_{22}{}^2=y\}.
$\smallbreak\noindent
By Equation~(\ref{E1.a}),
$\rho(\Gamma_\pm(x,y))=\pm\operatorname{Id}$ so
$\Gamma_\pm$ takes values in $\mathcal{Z}_\pm$.
\end{definition}

If $\Gamma\in\mathcal{Z}_\pm$, let $[\Gamma]$ denote the corresponding element of
$\mathfrak{Z}_\pm$.

\begin{lemma}\label{L2.2}
If $[\Gamma]\in\mathfrak{Z}_\pm$, 
then there exists $(x,y)\in \mathcal{S}$ so $[\Gamma]=[\Gamma_\pm(x,y)]$.\end{lemma}

\begin{proof} Choose $[\Gamma]\in\mathfrak{Z}_\pm$. 
We will show that there exists $T\in\operatorname{GL}(2,\mathbb{R})$ and a point $(x,y)\in\mathcal{S}$ so that 
$T^*\Gamma=\Gamma_\pm(x,y)$, i.e. that $\Gamma=\Gamma_\pm(x,y)$ after an appropriate change of coordinates.
We begin by making a linear change of coordinates to assume $\rho_\Gamma=\pm\operatorname{Id}$.
Consider the rotation
\begin{equation}\label{E2.a}
\begin{array}{l}
T_\theta(\partial_{x^1}):=\cos(\theta)\partial_{x^1}+\sin(\theta)\partial_{x^2},\\
T_\theta(\partial_{x^2}):=-\sin(\theta)\partial_{x^1}+\cos(\theta)\partial_{x^2}\,.
\end{array}\end{equation}
Let $(T_\theta^*\Gamma)$ be the expression of $\Gamma$ in this new coordinate system. Then
\begin{eqnarray*}
(T_\theta^*\Gamma)_{12}{}^1&=&\phantom{+}\cos^3(\theta)\Gamma_{12}{}^1
+\cos^2(\theta)\sin(\theta)(\Gamma_{22}{}^1+\Gamma_{12}{}^2-\Gamma_{11}{}^1)\\
&&+\cos(\theta)\sin^2(\theta)\{\Gamma_{22}{}^2
-\Gamma_{12}{}^1-\Gamma_{11}{}^2\}-\sin^3(\theta)\Gamma_{12}{}^2.
\end{eqnarray*}
Since $(T_\pi^*\Gamma)_{12}{}^1=-(T_0^*\Gamma)_{12}{}^1$,
the Intermediate Value Theorem shows that there exists $\theta$ so $(T_\theta^*\Gamma)_{12}{}^1=0$. 
There may, of course, be other values where this happens and this will play a role
in our subsequent development. 
It is exactly this step which fails in the indefinite setting as the structure group is $O(1,1)$
and not $O(2)$.

Normalize the coordinate system so that $\rho=\pm\operatorname{Id}$ and $\Gamma_{12}{}^1=0$. 
Because $\rho_{12}=-\Gamma_{11}{}^2\Gamma_{22}{}^1=0$ and
$\rho_{22}=(\Gamma_{11}{}^1-\Gamma_{12}{}^2)\Gamma_{22}{}^1\ne0$,
we conclude $\Gamma_{22}{}^1\ne0$ and $\Gamma_{11}{}^2=0$
so $\rho_{11}=(\Gamma_{11}{}^1-\Gamma_{12}{}^2)\Gamma_{12}{}^2$ and
$\rho_{22}=(\Gamma_{11}{}^1-\Gamma_{12}{}^2)\Gamma_{22}{}^1$.
Set $x:=\Gamma_{12}{}^2\ne0$ and $y:=\Gamma_{22}{}^2$; $x\ne0$. By changing the sign of $x^1$
and/or $x^2$, we may assume $x>0$ and $y\ge0$. We have
$\Gamma_{22}{}^1=x$.
We solve the equation $(\Gamma_{11}{}^1-x)x=\pm1$ to determine
$\Gamma_{11}{}^1$ and obtain
the normalizations defining $\Gamma_\pm$ to complete the proof.
\end{proof}

We establish some notation to use subsequently.
Set
$$
p_\pm(x,y):=\psi_3(\Gamma_\pm(x,y))\text{ and }P_\pm(x,y):=\Psi_3(\Gamma_\pm(x,y))
$$
to regard the invariants of Lemma~\ref{L1.3} as functions on $\mathcal{S}$. A direct
computation then shows:
\begin{equation}\label{E2.b}
\begin{array}{l}
p_\pm(x,y):=\pm4 x^2\pm\frac{1}{x^2}\pm y^2+2,\\
P_\pm(x,y):=4 x^4+x^2 \left(y^2\pm4\right)+\frac{y^2}{x^2}+2(1\pm y^2)\,.
\end{array}\end{equation}
Let $\Theta_\pm:=(p_\pm,P_\pm):\mathcal{S}\rightarrow\mathbb{R}^2$ and let
$J_\pm:=\det\{\Theta_\pm^\prime\}$ be the Jacobian
determinant. A direct computation shows:
\begin{equation}\label{E2.c}
\begin{array}{l}
J_-(x,y):=\frac{4 \left(x^2+1\right) y \left(4 x^6+x^4 y^2-x^2 \left(y^2+3\right)+1\right)}{x^5},\\
J_+(x,y):=-\frac{4 \left(x^2-1\right) y \left(4 x^6+x^4 y^2+x^2 \left(y^2-3\right)-1\right)}{x^5}.
\end{array}
\end{equation}
Thus $\Theta_\pm$ is a local diffeomorphism from $\mathcal{S}$ to $\mathbb{R}^2$ except
where $J_\pm(x,y)=0$. Set 
\begin{equation}\label{E2.d}
\begin{array}{ll}
Y_-(x):=\sqrt{\frac{1 + 4 x^6 - 3 x^2}{x^2 - x^4}}&\text{ for }0<x<1,\\
Y_+(x):=\sqrt{\frac{-4 x^6+3 x^2+1}{x^4+x^2}}&\text{ for }0<x\le1.
\end{array}\end{equation}
Let $L_\pm:=\{(x,y)\in\mathcal{S}:J_\pm(x,y)=0\}$ be the {\it Jacobi locus};
\begin{equation}\label{E2.e}
\begin{array}{llll}
L_-:=\{(x,y)\in\mathcal{S}:y=Y_-(x)\text{ for }0<x<1\text{ or } y=0\},\\
L_+:=\{(x,y)\in\mathcal{S}:y=Y_+(x)\text{ for }0<x\le 1\text{ or } y=0\text{ or }x=1\}.
\end{array}\end{equation}
\smallbreak\centerline{
\includegraphics[width=3.5cm,keepaspectratio=true]{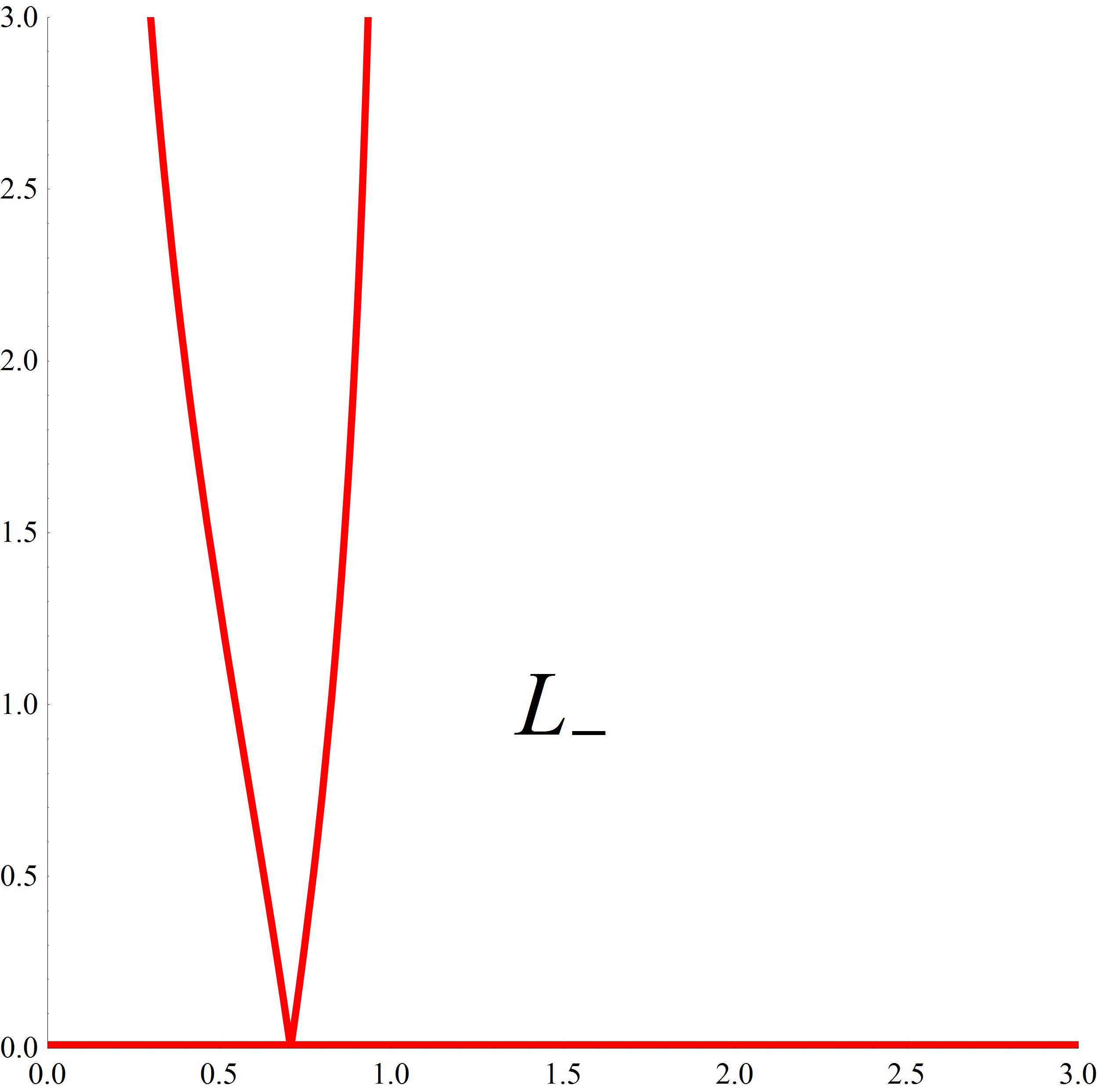}\quad
\includegraphics[width=3.5cm,keepaspectratio=true]{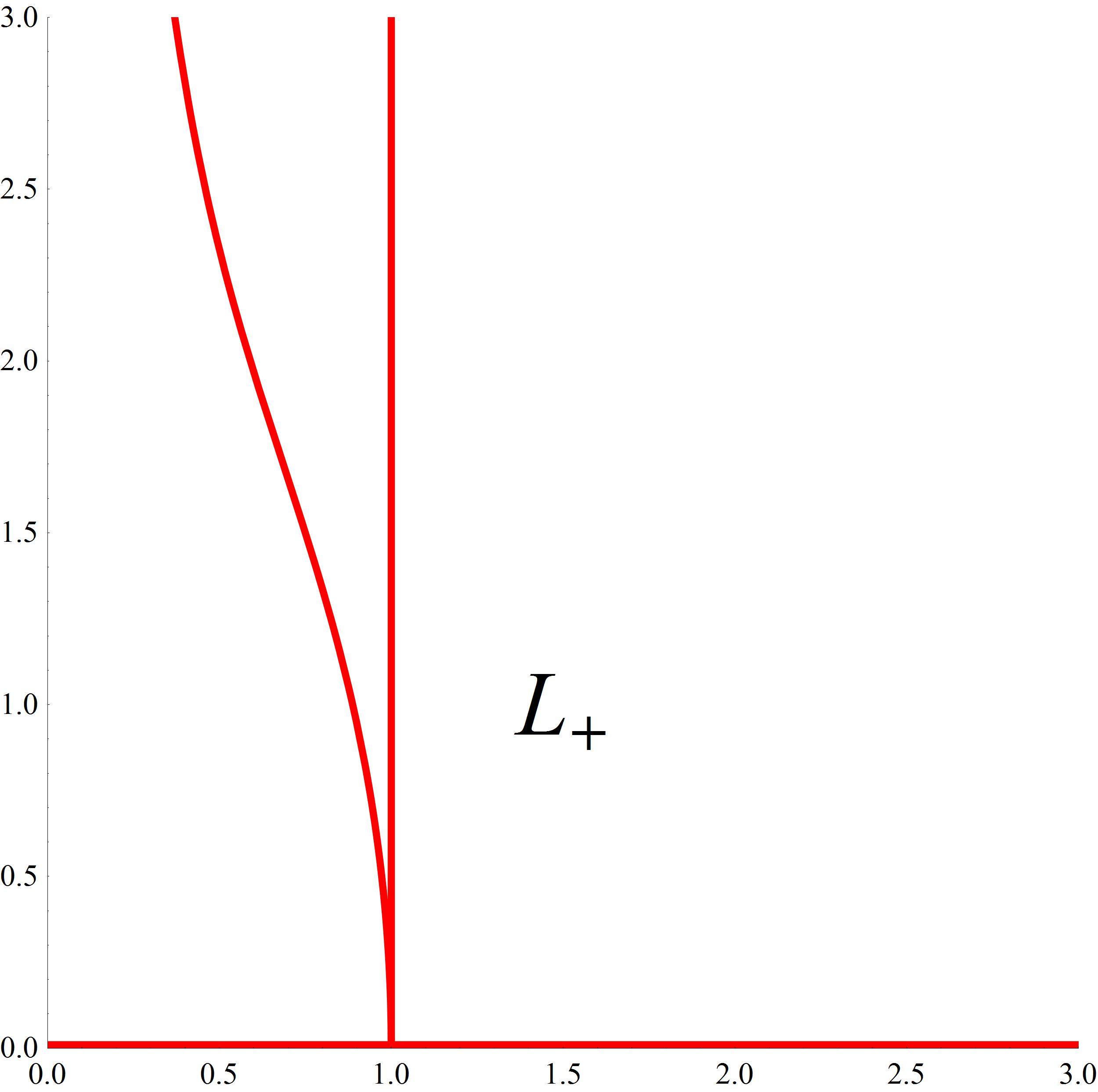}}
\smallbreak\hfill Fig.~2.1 $\Uparrow$\hfill Fig.~2.2 $\Uparrow$\hfill\vphantom{.}
\smallbreak\noindent Fig.~2.1 gives the Jacobi locus $L_-$ and 
Fig.~2.2 gives the Jacobi locus $L_+$.
The $x$ axis is always on the Jacobi locus. The half ray $x=1$ is in $L_+$ but not in $L_-$. The graph $y=Y_+(x)$ touches the $x$-axis when $x=1$.
The graph $y=Y_-(x)$ touches the $x$-axis when $x=\frac1{\sqrt2}$.

\subsection{The moduli space $\mathfrak{Z}_+$}\label{S2.1}
Adopt the notation of Definition~\ref{D2.1}.
We may set $T_1(x^1,x^2):=(-x^1,x^2)$. Then $T_1^*\Gamma_+(x,y)=\Gamma_+(-x,y)$ so
$[\Gamma_+(x,y)]=[\Gamma_+(-x,y)]$. Similarly we may use $T_2(x^1,x^2)=(x^1,-x^2)$ to see
 $[\Gamma_+(x,y)]=[\Gamma_+(x,-y)]$. Consequently, 
 we can change the sign of $x$ and/or of $y$ as desired.
By Lemma~\ref{L1.3}, $[\Gamma_+(x,y)]=[\Gamma_+(\bar x,\bar y)]$
if and only if
$T^*\Gamma_+(x,y)=\Gamma_+(\bar x,\bar y)$ for some $T\in\operatorname{GL}(2,\mathbb{R})$.
Since the Ricci tensor is normalized to $\pm\operatorname{diag}(1,1)$, this implies $T\in O(2)$;
after replacing $x^1$ by $-x^1$ (or $x^2$ by $-x^2$)
if necessary, we may assume $T\in SO(2)$. This means that there
exists a rotation $T_\theta$ (see Equation~(\ref{E2.a})) so that 
$$T_\theta^*\{\Gamma_+(x,y)\}_{12}{}^1=0,\quad
    T_\theta^*\{\Gamma_+(x,y)\}_{12}{}^2=\pm x,\quad
    T_\theta^*\{\Gamma_+(x,y)\}_{22}{}^2=\pm y\,.
$$
We extend Equation~(\ref{E2.a}) to the conformal group. For $(u,v)\ne(0,0)$, define:
\begin{equation}\label{E2.f}
T_{u,v}(\partial_{x^1})=u\partial_{x^1}+v\partial_{x^2},\quad
 T_{u,v}(\partial_{x^2})=-v\partial_{x^1}+u\partial_{x^2}\,.
 \end{equation}
We recover $T_\theta$ by setting $\cos(\theta)=\frac {u}{\sqrt{u^2+v^2}}$ and
$\sin(\theta)=\frac{v}{\sqrt{u^2+v^2}}$.
We examine the equation $\{T^*\Gamma_+(x,y)\}_{12}{}^1=0$. This yields the
homogeneous cubic equation 
$$
v\{u^2(x^2-1)+u v x y-v^2 x^2\}=0\,.
$$
If $v=0$, then $T_{1,0}$ is the identity map.  Thus we may assume
$v\ne0$. Since we may replace $(u,v)$ by
$(-u,-v)$, we may assume $v>0$. We normalize
to assume $v=1$ and obtain the quadratic equation:
\begin{equation}\label{E2.g}
u^2(x^2-1)+uxy-x^2=0\,.
\end{equation}
Equation~(\ref{E2.g}) has at most 2 solutions $u_\pm(x,y)$; we shall
set $u_\pm=\text{DNE}$ if no solution
exists. Assume $x>0$ and set:
\begin{equation}\label{E2.i}
u_\pm(x,y):=\left\{\begin{array}{ll}
\text{DNE}&\text{if }x=1\text{ and }y=0\\
\text{DNE}&\text{if }x\ne1\text{ and }y^2+4(x^2-1)<0\\
\frac{-y}{2(x^2-1)}x&\text{if }y^2+4(x^2-1)=0\\
\frac xy&\text{if }x=1\text{ and }y>0\\
\frac{-y\pm\sqrt{y^2+4(x^2-1)}}{2(x^2-1)}x&\text{if }x\ne1\text{ and }y^2+4(x^2-1)>0\\\end{array}\right\}.\end{equation}
Assuming $u_\pm$ is well defined, the corresponding rotation angles $\theta_\pm$
and solutions $(x_\pm(x,y),y_\pm(x,y))$ to the equation $[\Gamma(x,y)]=[\Gamma(x_\pm,y_\pm)]$
are given by:
\begin{eqnarray*}
&&\textstyle\cos_\pm(x,y):=\frac{u_\pm(x,y)}{\sqrt{1+ u_\pm(x,y)^2}},\quad
\sin_\pm(x,y):=\frac1{\sqrt{1+u_\pm(x,y)^2}},\\
&&x_\pm(x,y):= T_{\cos_\pm(x,y),\sin_\pm(x,y)}^*\{\Gamma_+(x,y)\}_{22}{}^1,\\
&&y_\pm(x,y):= T_{\cos_\pm(x,y),\sin_\pm(x,y)}^*\{\Gamma_+(x,y)\}_{22}{}^2.
\end{eqnarray*}
Of course $x_\pm$ and $y_\pm$ can be negative, so we must take the absolute values.
This yields the following result:
\begin{lemma}
Let $(x,y)\in\mathcal{S}$ and $(\bar x,\bar y)\in\mathcal{S}$. Then
$[\Gamma_+(x,y)]=[\Gamma_+(\bar x,\bar y)]$ if and only if at least one of the following
3 conditions (which need not be exclusive) holds:
\begin{enumerate}
\item $(\bar x,\bar y)=(x,y)$.
\item $x_-(x,y)$ and $y_-(x,y)$ are defined and $(\bar x,\bar y)=(|x_-(x,y)|,|y_-(x,y)|)$.
\item $x_+(x,y)$ and $y_+(x,y)$ are defined and $(\bar x,\bar y)=(|x_+(x,y)|,|y_+(x,y)|)$.
\end{enumerate}\end{lemma}

This shows that the map $(x,y)\rightarrow[\Gamma_+(x,y)]$
is at most a $3\rightarrow1$ map from $\mathcal{S}$ to $\mathfrak{Z}_+$. 
If $(x,y)\in\mathcal{S}$, let
$1\le n_+(x,y)\le3$ be the number of distinct points $(\bar x,\bar y)\in\mathcal{S}$
which satisfy $[\Gamma_+(x,y)]=[\Gamma_+(\bar x,\bar y)]$. 
To determine $n_+(x,y)$, we examine the 5 possibilities
of Equation~(\ref{E2.i}).
\subsection*{Case 2.1.1. Let $(x,y)=(1,0)$}
 As
Equation~(\ref{E2.g}) has no solutions,
$n_+(x,y)=1$.
\subsection*{Case 2.1.2.  Let $y^2+4(x^2-1)<0$} As
Equation~(\ref{E2.g})
has no solutions, $n_+(x,y)=1$.
\subsection*{Case 2.1.3. Let $y^2+4(x^2-1)=0$ and $0<x<1$}
Let
\begin{eqnarray*}
&&\mathcal{D}:=\{(x,y)\in\mathcal{S}:y^2+4(x^2-1)=0\},\\
&&\mathcal{R}:=\{(x,y)\in\mathcal{S}:x=1\text{ and }y>0\}
\end{eqnarray*}
be the {\it discriminant locus} and a ray in $\mathcal{S}$, respectively. We define $\mathcal{D}$
and $\mathcal{R}$ by setting, respectively:
$P_{\mathcal{D}}(t):=(t,2\sqrt{1-t^2})$ and $P_{\mathcal{R}}(t):=(1,\frac{\sqrt{1-t^2}}t)$
for $t\in(0,1)$.
If $(x,y)=P_{\mathcal{D}}(t)\in\mathcal{D}$, then there is only one solution
$(x_\pm,y_\pm)=(1,-\frac{\sqrt{1-t^2}}t)$.
Thus after changing the sign of $y_\pm$, we see  
$[\Gamma_+(P_{\mathcal{D}}(t))]=[\Gamma_+(P_{\mathcal{R}}(t))]$ and
$n=2$ on $\mathcal{D}$ and on $\mathcal{R}$.

\subsection*{Case 2.1.4. Let $(x,y)\in\mathcal{R}$} We use the discussion of Case 3 to
see $n_+(x,y)=2$ and $[\Gamma(x,y)]=[\Gamma(\bar x,\bar y)]$ for exactly one
point $(\bar x,\bar y)\in\mathcal{D}$.

\subsection*{Case 2.1.5. Suppose that $y=0$ and $x>1$} We compute
$$\begin{array}{ll}
(x_+,y_+)=(\frac1{\sqrt{2x^2-1}},-\frac{(1+2x^2)\sqrt{x^2-1}}{x\sqrt{2x^2-1}}),\\
(x_-,y_-)=(-\frac1{\sqrt{2x^2-1}},-\frac{(1+2x^2)\sqrt{x^2-1}}{x\sqrt{2x^2-1}}).
\end{array}$$
Since $x_-=-x_+$ and $y_+=y_-$, $n_+(x,y)=2$. Since $x>1$, $x_+<1$. Furthermore,
one may verify that $(x_+,y_+)$ is on the Jacobi locus $L_+$. The map $x\rightarrow(x_+,y_+)$
parametrizes the Jacobi locus.

Suppose $(x,y)$ is on the Jacobi locus. The discussion above shows $n_+(x,y)=2$
since there exists a unique point $(\bar x,0)$ with $\bar x>1$ so
$[\Gamma_+(x,y)]=[\Gamma_+(\bar x,0)]$. 

The complement in $\mathcal{S}$
of the discriminant locus $\mathcal{D}$, the ray $\mathcal{R}$, and the
 Jacobi locus  $L_+$ consists of 4 open disjoint regions pictured below in Fig.~2.3.
\smallbreak\centerline{
\includegraphics[width=4cm,keepaspectratio=true]{Fig2-2.jpeg}\qquad
\includegraphics[width=4cm,keepaspectratio=true]{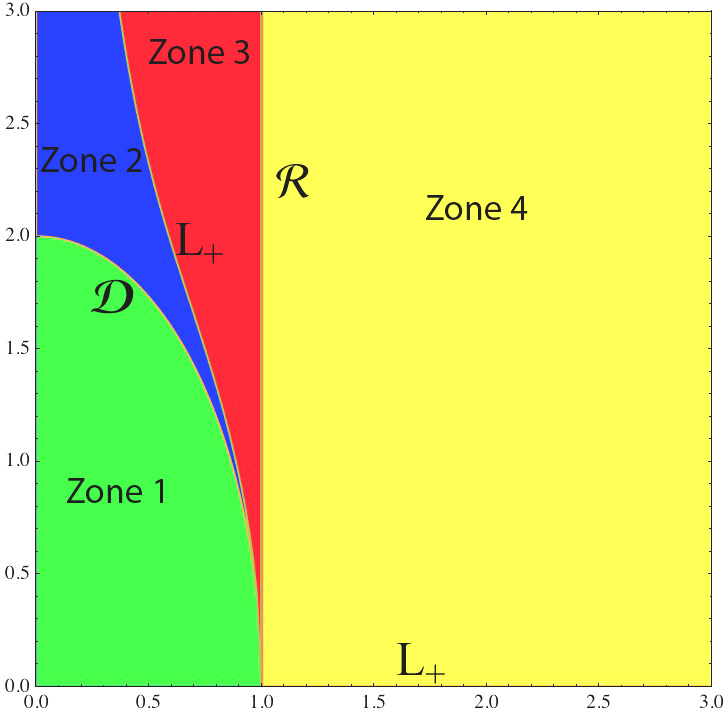}}
\smallbreak\hfill Fig.~2.2 $\Uparrow$\hfill Fig.~2.3 $\Uparrow$\hfill\vphantom{.}
\smallbreak\noindent
Zone 1 is in green, Zone 2 is in blue, Zone 3 is in red, and Zone 4 is in yellow. The boundary
between Zone 1 and Zone 2 is $\mathcal{D}$, the boundary between
Zone 2 and Zone 3 is $L_+$, and the boundary between Zone 3 and
Zone 4 is $\mathcal{R}$.  If $(x,y)\in\mathcal{S}$ and if $(\bar x,\bar y)\in\mathcal{S}$,
then we say that $(x,y)\sim(\bar x,\bar y)$ if and only if
$[\Gamma_+(x,y)]=[\Gamma_+(\bar x,\bar y)]$. 
Let $\mathcal{O}$ be the union of Zone 2, Zone 3, and Zone 4.
The analysis above shows that if
$(x,y)\in\mathcal{O}$ and if $(x,y)\sim(\bar x,\bar y)$, then $(\bar x,\bar y)\in\mathcal{O}$
as well.

We examine a specific example:
$$\begin{array}{rrr}
(x,y)&(x_+(x,y),y_+(x,y))&(x_-(x,y),y_-(x,y))\\
(0.1, 2.2),&(0.636041, -10.0394),&( 1.56535, -9.72988),\\
(0.636041, 10.0394),&(0.1, -2.2), &(1.56535, 9.72988),\\
(1.56535, 9.72988),&(0.1, -2.2),& (-0.636041, -10.0394).
\end{array}$$
We adjust the signs and define $Q^2_\pm(x,y):=(x_\pm(x,y),-y_\pm(x,y))$ on Zone 2.
The analysis performed above shows that $Q^2_\pm$ can not cross zone boundaries.
Since $Q^2_+(0.1,2.2)$ takes values in Zone 3 and $Q^2_-(0.1,2.2)$ takes values in
Zone 4, the same is true on all of Zone 2. Thus $n_+(x,y)=3$ on Zone 2. A similar analysis
pertains on Zone 3 and Zone 4. 

Let Equation~(\ref{E2.d}) define $Y_+$.
Let $\mathcal{C}$ be the closure of Zone 1 and Zone 2:
$$\mathcal{C}:=\{(x,y):0\le x\le 1\text{ and }0\le y\le Y_+(x)\}\,.$$
The following is now immediate from our discussion:

\begin{lemma}\label{L2.4} Let $[\Gamma]\in\mathfrak{Z}_+$. There is a unique
$(x,y)\in\mathcal{C}$ so $[\Gamma_+(x,y)]=[\Gamma]$.
\end{lemma}

\begin{proof}[The proof of Theorem~\ref{T1.5} if $\rho>0$]We now turn to the question of defining invariants which completely detect
$\mathfrak{Z}_+$. We change variables slightly to define
\begin{eqnarray*}
&&\Omega_+(x,y):=(\psi_3(\Gamma_+(x,y)),4\psi_3(\Gamma_+(x,y))-\Psi_3(\Gamma_+(x,y))-18)\\
&&\qquad=\left(2+x^{-2}+4x^2+y^2, -x^{-2}(x^2-1)^2(-4+4x^2+y^2)\right)\,.
\end{eqnarray*}
Let $(x,y)\equiv(\bar x,\bar y)$ if and only if
$\Omega_+(x,y)=\Omega_+(\bar x,\bar y)$.
To complete the proof of Theorem~\ref{T1.5} for $\mathfrak{Z}_+$, we must show
that if $(x,y)\in \mathcal{C}$, $(\bar x,\bar y)\in \mathcal{C}$, and $(x,y)\equiv(\bar x,\bar y)$ then $(x,y)=(\bar x,\bar y)$.
We observe that
the curves $y\rightarrow\Omega_+(x,y)$ for fixed $x$ are rays. We introduce three pictures
\smallbreak\centerline{\includegraphics[width=3.5cm,keepaspectratio=true]{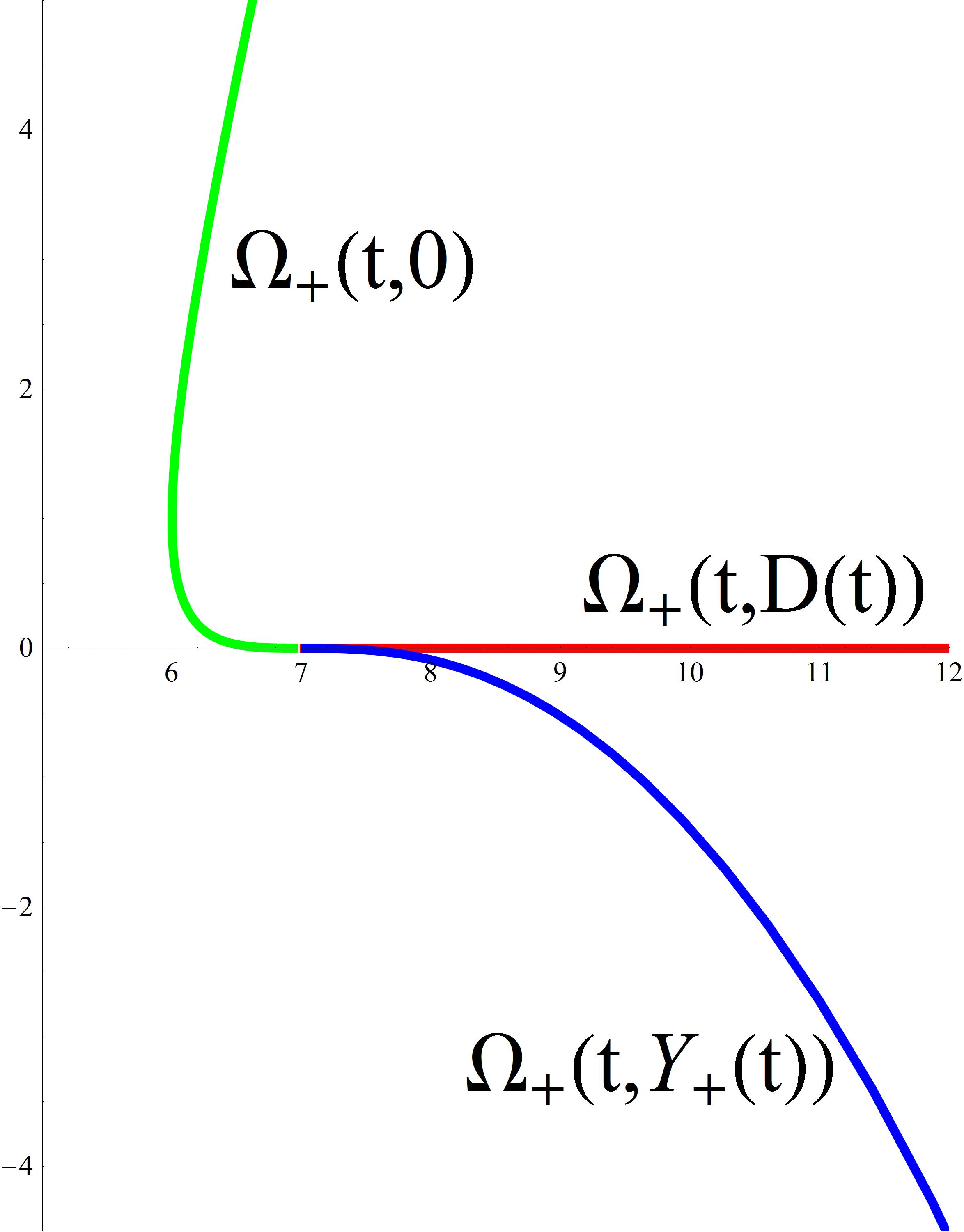}
\qquad\includegraphics[width=3.5cm,keepaspectratio=true]{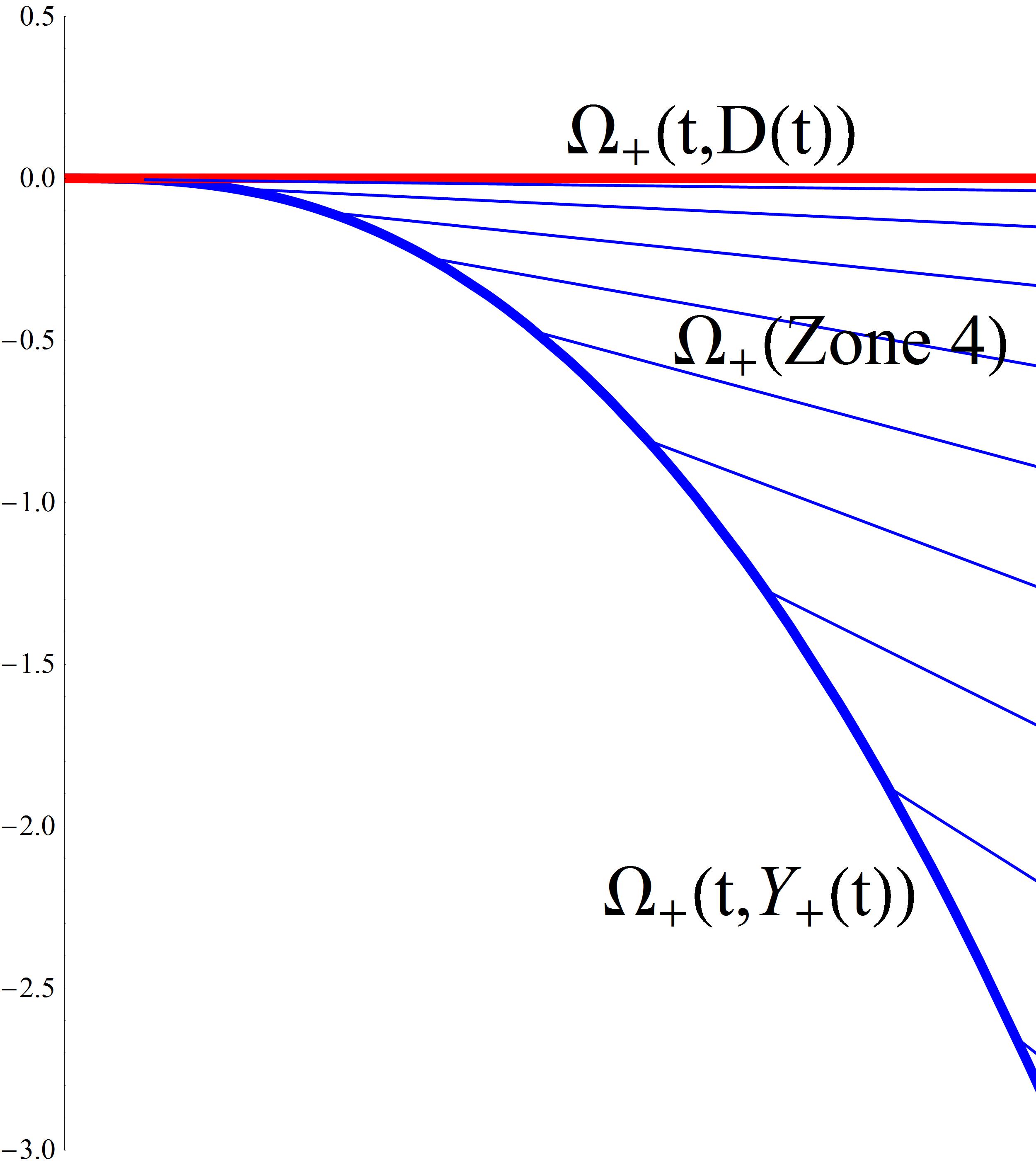}\qquad
\includegraphics[width=3.5cm,keepaspectratio=true]{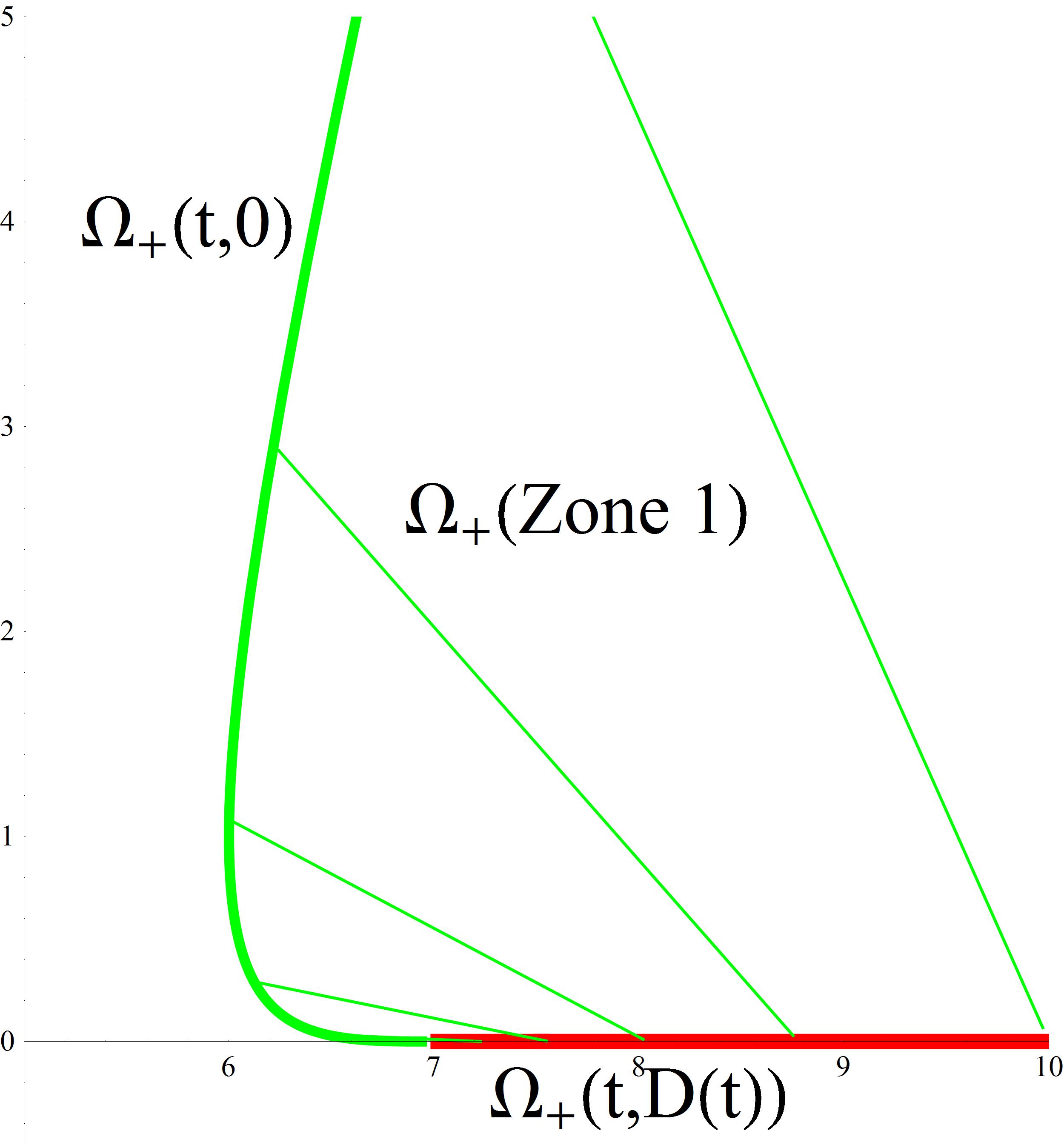}
}
\par\noindent
\hfill Fig.~2.4 $\Uparrow$\qquad\hfill Fig.~2.5 $\Uparrow$\hfill\qquad Fig.~2.6 $\Uparrow$\hfill\vphantom{.}
\smallbreak\noindent
Fig.~2.4 gives the boundary of the image of $\mathcal{C}$. The upper curve in the first quadrant is the
image of the Jacobi locus, the image of the discriminant locus
lies along the horizontal axis, and the lower curve in the fourth quadrant is the image of the $x$-axis
for $0<x<1$.
\begin{eqnarray*}
&&\Omega_+(t,0)=(4 t^2+t^{-2}+2,4t^{-2}(1-t^2)^3),\\
&&\Omega_+(t,D(t))=(t^{-2}+6,0),\\
&&\textstyle\Omega_+(t,Y_+(t))=(6+2(t^2+t^4)^{-1},(t^4+t^6)^{-1}(t^2-1)^3)\,.
\end{eqnarray*}
It is immediate that $\Omega_+(t,0)$ lies in the 1st quadrant,
 $\Omega_+(t,D(t))$ lies along the horizontal axis, and $\Omega_+(t,Y_+(t))$
 lies in the 4th quadrant.
 If $(x,y)\in C$, we may represent $(x,y)=(t,s\cdot Y_+(t))$ for some $s\in[0,1]$.
The curve $s\rightarrow\Omega_+(t,s\cdot Y_+(t))$ for fixed $t$ and $s\in[0,1]$ is a straight line which runs from
the upper curve in the first quadrant down to the corresponding point horizontal axis and then down to the 
lower curve. The singular point $(1,0)$
corresponds to the point $(7,0)$ on the horizontal axis where the upper curve intersects the lower curve.
If $y$ lies in the closure of Zone 1,
then $\Omega_+(x,y)$ lies in the 4th quadrant; if $y$ lies in the closure of Zone 2, then
$\Omega_+(x,y)$ lies in the 1st quadrant. Suppose $(x,y)\equiv(\bar x,\bar y)$ for $0< x \leq 1$
and $0\le y\le Y_+(x)$. Our analysis shows that if $x=1$ then $(x,y)=(\bar x,\bar y)=(1,0)$.
Furthermore, $(x,y)$ belongs to the closure of Zone 1 (resp. Zone 2) if and only if $(\bar x,\bar y)$
belongs to the closure of Zone 1 (resp. Zone 2). We may therefore examine the two cases seriatim.

Instead of examining the image of Zone 2, we may equivalently examine the image of Zone 4.
This is pictured in Fig.~2.5. The upper boundary is the $x$ axis; this is the curve $\Omega_+(1,y)$ for $y\ge0$.
The lower boundary is the curve $\gamma(x):=\Omega_+(x,0)$ where $x\ge1$. The  straight lines
are
$\Omega(x_n,y)$ for $y\ge1$ and suitably chosen $x_n$. 
Our task is to show these do not intersect. 
We must show the first coordinate of $\gamma$ increases monotonically, the
second coordinate of $\gamma$ decreases monotonically, the slope of  $\gamma$
decreases monotonically, and the slope of the lines leaving $\gamma$ decrease monotonically
and are always strictly greater than the slope of $\gamma$.
Let $\gamma=(\varrho_1,\varrho_2)$. We have
$\gamma(x)=(2+x^{-2}+4x^2,4x^{-2}(1-x^2)^3)$. As desired, the first coordinate
increases monotonically and the second coordinate decreases monotonically.
One computes that
\begin{eqnarray*}
&&\text{Slope of }\gamma
=\left.\frac{\partial_x\varrho_2(x,y)}{\partial_x\varrho_1(x,y)}\right|_{y=0}= -\frac{4 \left(x^2-1\right)^2}{2 x^2-1},\\
&&\text{Slope of line from }\gamma
=\left.\frac{\partial_y\varrho_2(x,y)}{\partial_y\varrho_1(x,y)}\right|_{y=0}=-\frac{(x^2-1)^2}{x^2}\,.
\end{eqnarray*}
These are monotonically decreasing for $x\geq 1$, negative, and as desired 
the slope of $\gamma$ is more negative than the slope of the line from $\gamma$. This
completes our analysis of Zone 4.

We now turn our analysis to Zone 1 and refer to Fig.~2.6. 
The lower boundary is the $x$ axis; 
this is the curve $\Omega_+(x,2\sqrt{1-x^2})$ and the upper boundary is the curve 
$\Omega_+(x,0)$. The straight lines are the lines $(x_n,y)$ 
for suitable values of $x_n$
where $0\leq y\leq 2\sqrt{1-x_n^2}$.
And if the picture is to be believed, the full rays do not intersect. 
Let $\tau(x)=2\sqrt{1-x^2}$ so that $(x,\tau(x))$ parametrizes the discriminant locus. We have
$(\varrho_1,\varrho_2)=\Omega_+(x,\tau(x))=(6+x^{-2},0)$ and
$\frac{\partial_y\varrho_1(x,y)}{\partial_y\varrho_2(x,y)}|_{y=\tau(x)}=-\frac{(-1+x^2)^2}{x^2}$.
Thus as $x$ increases from $0$ to $1$, $\varrho_1$ monotonically 
decreases from $\infty$ to $7$ 
and the slope of the straight lines
monotonically increases from $-\infty$ to $0$. Thus the picture really is as drawn
and the rays do
not intersect. This completes the proof of Theorem~\ref{T1.5} in the positive definite setting.\end{proof}
\subsection{The moduli space $\mathfrak{Z}_-$}\label{S2.2}
We suppose that $\rho$ is negative definite and recall our parametrization:
$$
\Gamma_-(x,y):=\{
\Gamma_{11}{}^1=\textstyle x-\frac1x,
\Gamma_{11}{}^2=0,\Gamma_{12}{}^1=0,\Gamma_{12}{}^2=x,
\Gamma_{22}{}^1=x,\Gamma_{22}{}^2=y\}\,.
$$
We recall some notation
established previously and set
\medbreak\qquad
$p_-(x,y):=\psi_3(\Gamma_-(x,y))=2-4 x^2-\frac{1}{x^2}-y^2$,
\smallbreak\qquad
$P_-(x,y):=\Psi_3(\Gamma_-(x,y))=4 x^4+x^2 \left(y^2-4\right)+\frac{y^2}{x^2}+2(- y^2+1)$,
\smallbreak\qquad
$J_-(x,y):=\det(\Theta_{-}^\prime)=\frac{4 \left(x^2+1\right) y \left(4 x^6+x^4 y^2-x^2 \left(y^2+3\right)+1\right)}{x^5}$.
\smallbreak\centerline{\includegraphics[width=4cm,keepaspectratio=true]{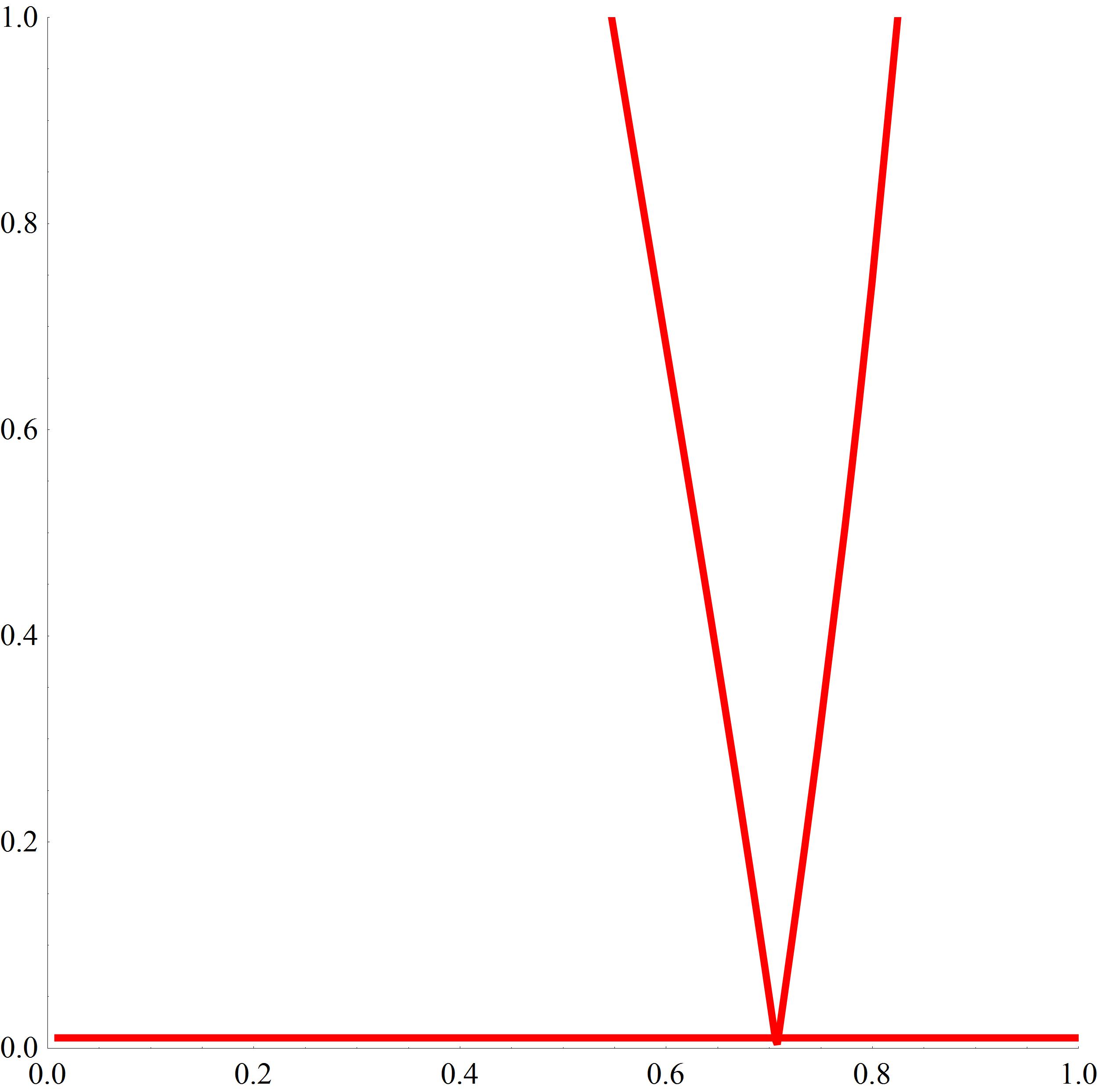}\qquad
\quad\includegraphics[width=4cm,keepaspectratio=true]{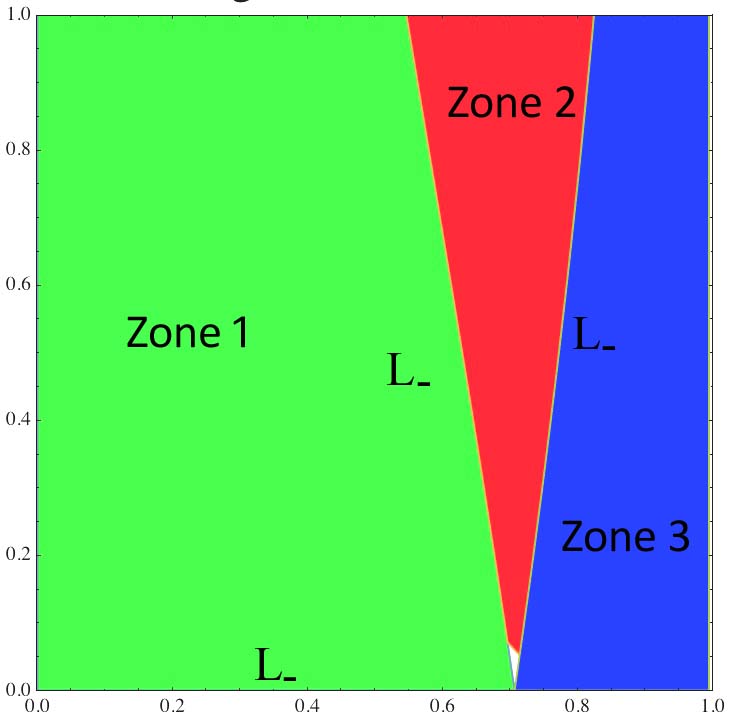}}
\par\noindent
\hfill Fig.~2.7 $\Uparrow$\hfill Fig.~2.8 $\Uparrow$\hfill \vphantom{.}
\smallbreak\noindent
The solutions of $J_-(x,y)=0$ are the Jacobi-locus. This includes the  ray $y=0$ and the graph $(x,Y_-(x))$
pictured in Figure~2.7 with a corner at $(\frac1{\sqrt2},0)$ where 
$$y=Y_-(x)=\sqrt{\frac{1 + 4 x^6 - 3 x^2}{x^2 - x^4}}\text{ for }0<x<1\,.$$
The Jacobi locus divides $\mathcal{S}$ into 3 zones. Zone 1 is in green; it lies below
the Jacobi locus for $x<\sqrt 2^{-1}$. Zone 2 is in red; it lies above the Jacobi locus.
Zone 3 is the remainder. It lies below the Jacobi locus for $\sqrt 2^{-1}\le x<1$. It also contains
the half-plane $x\ge1$.

We use Equation~(\ref{E2.f})
to define $T_{u,v}$. Setting $\{T_{u,v}^*\Gamma\}_{12}{}^1=0$ yields the following cubic equation which is analogous to Equation~(\ref{E2.g}) in the positive definite setting:
$$(1+x^2)u^2+uxy-x^2=0\,.$$
There is no discriminant locus and we obtain two solutions
$$
u_\pm=x\frac{-y\pm\sqrt{4(1+x^2)+y^2}}{2(1+x^2)}\,,
$$
which we use to define $x_\pm$ and $y_\pm$ as before. We now let $n_-(x,y)$ be the
number of distinct points $(\bar x,\bar y)$ so $[\Gamma_-(x,y)]=[\Gamma_+(\bar x,\bar y)]$.
We analyze cases as follows.
\subsection*{Case 2.2.1. We examine the singular point $(\frac1{\sqrt2},0)$}
This is where
the Jacobi locus intersects the horizontal axis in a corner. We show $n_-(\frac1{\sqrt2},0)=1$ by computing:
\begin{eqnarray*}
&&\textstyle(x_+(\frac1{\sqrt2},0),y_+(\frac1{\sqrt2},0))=(-\frac1{\sqrt2},0),\\
&&\textstyle(x_-(\frac1{\sqrt2},0),y_-(\frac1{\sqrt2},0))=(+\frac1{\sqrt2},0)\,.
\end{eqnarray*}

\subsection*{Case 2.2.2. Suppose that $(x,y)=(x,0)$ lies on the horizontal axis but $x\ne \frac1{\sqrt2}$}
We show $(x_\pm(x,0),y_\pm(x,0))$ lies on the Jacobi locus as follows.
If $x=\sqrt t$, then:
\begin{eqnarray*}
&&x_\pm^2(\sqrt t,0)=\frac1{1+2t},\quad
y_\pm^2(\sqrt t,0)=\frac{1-3t+4t^3}{t+2t^2},\\
&&4x_\pm^6+x_\pm^4y_\pm^2-x_\pm^2(y_\pm^2+3)+1=0.
\end{eqnarray*}
Consequently, $n(x,y)=2$ on the horizontal axis minus $(\frac1{\sqrt2},0)$ and $n(x,y)=2$ on
the Jacobi locus minus $(\frac1{\sqrt2},0)$; every point on the horizontal axis for $x<\frac1{\sqrt2}$
corresponds to a point of the Jacobi locus $(\bar x,\bar y)$ for $\bar x>\frac1{\sqrt2}$ while every
point on the horizontal axis for $x >\frac1{\sqrt2}$ corresponds to a point on the Jacobi locus $(\bar x,\bar y)$
for $\bar x<\frac1{\sqrt2}$; the two regions are reversed.
\subsection*{Case 2.2.3.}
The discussion above shows that no point in the interior of Zone 1, Zone 2, or Zone 3
is equivalent to a point on the Jacobi locus or on the horizontal axis in the moduli space. We compute:
$$\begin{array}{rrrrrr}
(x,y)&(x_+(x,y),y_+(x,y))&(x_-(x,y),y_-(x,y))\\
(1/\sqrt2, 100),&( -0.00999825, 0.706983),&( 1.4139,-99.9775),\\
(0.00999825, 0.706983),&( -0.707107, 100.),&( 1.4139, 99.9775),\\
(1.4139, 99.9775),&( -0.00999825, -0.706983),&( 0.707107, -100.).
\end{array}$$
We now conclude every point in Zone 1 is equivalent to a point in Zone 2 and to a point in Zone 3;  so
$n(x,y)=3$ for these points.

We now examine the image of Zone 2. It is convenient to decompose Zone 2 into two parts
where Zone 2-a lies above the Jacobi locus for $0<x\le\sqrt 2^{-1}$ and Zone~2-b lies above
the Jacobi locus for $\sqrt 2^{-1}\le x<1$. Again, we renormalize setting
\begin{eqnarray*}
&&\Omega_-(x,y):=(-2-\psi_3(\Gamma_-(x,y)),2\Psi_3(\Gamma_-(x,y))+\psi_3(\Gamma_+(x,y))\\
&&\qquad=\left(4 x^2+x^{-2}+y^2-4,x^{-2}(2 x^2-1) \left(4 x^4+x^2 \left(y^2-4\right)-2 y^2+1\right) \right)\,.
\end{eqnarray*}
\centerline{
\includegraphics[height=4cm,keepaspectratio=true]{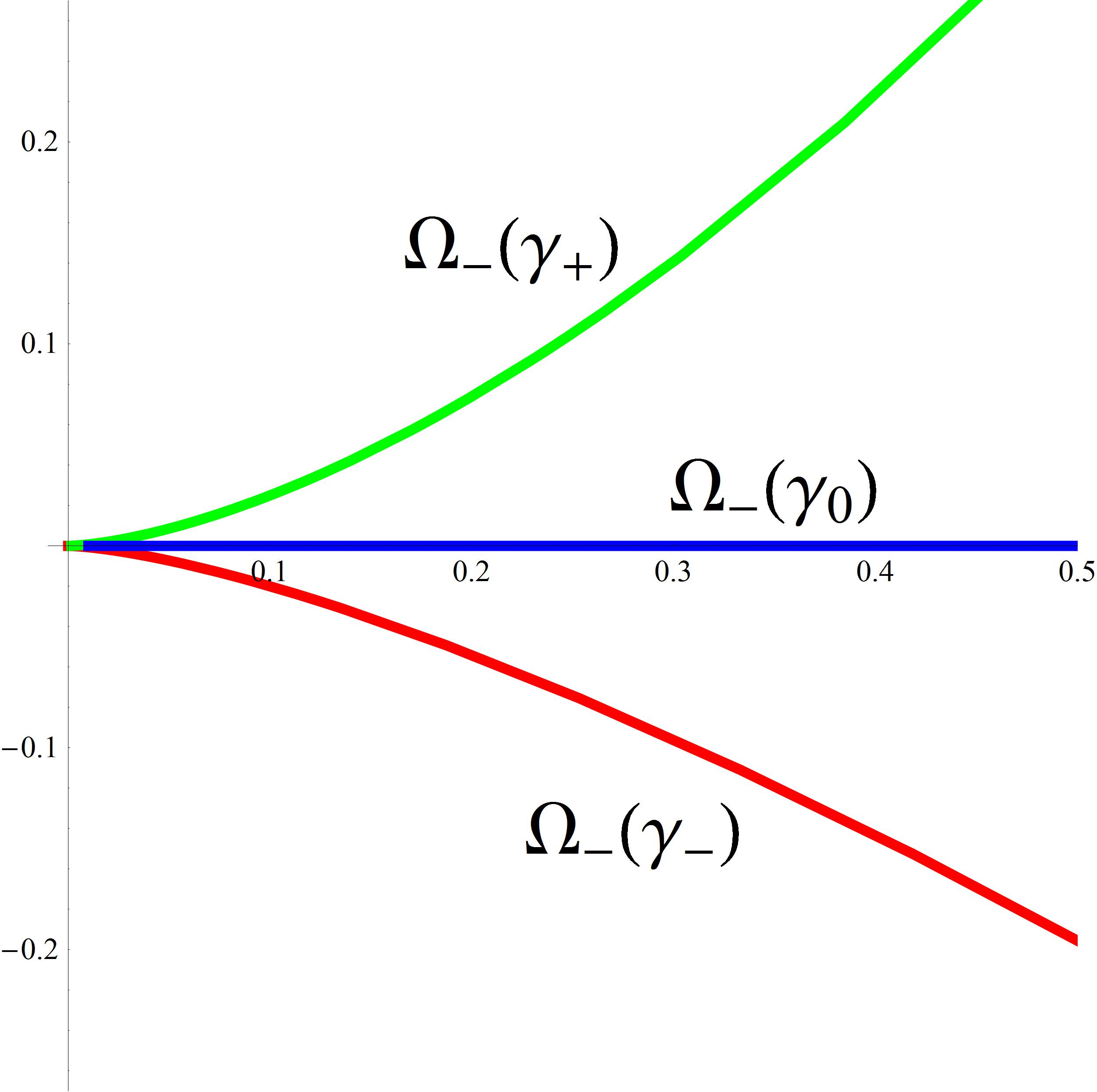}\quad
\includegraphics[height=4cm,keepaspectratio=true]{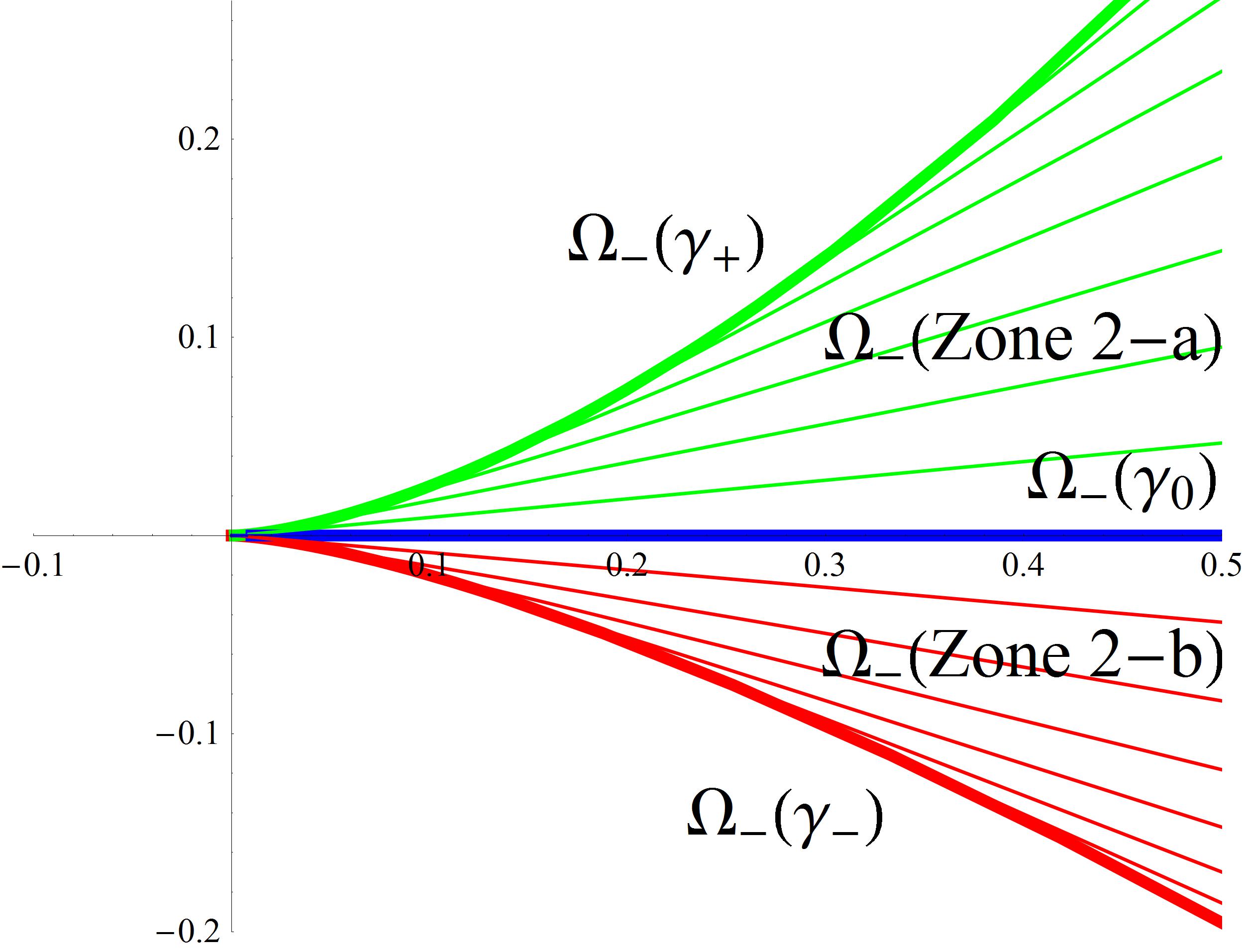}}
\par\hfill Fig.~2.9 $\Uparrow$\hfill Fig.~2.10 $\Uparrow$\hfill\vphantom{.}
\smallbreak\noindent

The $x$-axis is also part of the Jacobi locus in this instance.
Figure~2.9 shows the image of the graph of $Y_-$ of the Jacobi locus and the line $x=\sqrt{2}^{-1}$. The upper
curve $\gamma_+$ is the image of the graph of $Y_-$ of the Jacobi locus for $0<x<\sqrt{2}^{-1}$,
the curve lying along the horizontal axis $\gamma_0$ is the image of the ray $(\sqrt{2}^{-1},t)$,
and the lower curve $\gamma_-$ is the image of the graph of $Y_-$ of the Jacobi locus for $\sqrt{2}^{-1}<x<1$. 
More explicitly,
\begin{eqnarray*}
&&\gamma_+(t)=2\{t^2(1-t^2)\}^{-1}((1-2t^2)^2, \frac{(1-2t^2)^3}{t^2})\text{ for }0<t\le\sqrt2^{-1},\\
&&\gamma_0(t)=(t^2,0), \text{ for } 0\leq t,\\
&&\gamma_-(t)=2\{t^2(1-t^2)\}^{-1}((1-2t^2)^2, \frac{(1-2t^2)^3}{t^2})\text{ for }\sqrt2^{-1}\le t<1\,.
\end{eqnarray*}
Thus the picture actually is as drawn. Fig.~2.10 includes the lines
$\Omega(x_n,Y_+(x_n)+s)$ for suitably chosen values of $x_n$ and $s\in[0,\infty)$. This will
complete the proof if the picture can be believed as this forms 
a pencil of rays that do not intersect.
Since we are on the Jacobi locus, the rank of the Jacobian of $\Omega_-$ is $1$. Let $F_-(t,s):=\Omega_-(t,Y_-(t)+s)=(\varrho_1(t,s),\varrho_2(t,s))$. 
We compute
$$
\frac{\partial_t\varrho_2}{\partial_t\varrho_1}|_{s=0}=
\frac{\partial_s\varrho_2}{\partial_s\varrho_1}|_{s=0}=-5+2t^{-2}+2t^2\,.
$$
As $t$ ranges from $0$ to $\sqrt{2}^{-1}$ to $1$, 
$F_-(t,0)$ ranges from $(\infty,\infty)$ to $(0,0)$
to $(\infty,-\infty)$ and traverses the curve $\gamma=\gamma_+\cup \gamma_-$ counterclockwise. The
slope ranges from $+\infty$ to $0$ to $-\infty$ and decreases monotonically. The
lines $s\rightarrow F_-(t_0,s)$ are tangent to the curve and their slopes also
decrease monotonically. Thus the situation is as depicted in Fig.~2.10 and these
curves form a pencil of rays that do not intersect. This completes the proof of
Theorem~\ref{T1.5} in the negative definite setting.

\section{The moduli space $\mathfrak{Z}_0$}\label{S3}
Let $\Gamma\in\mathcal{Z}_0\subset\mathbb{R}^6$ be a Christoffel symbol so that the associated
Ricci tensor $\rho$ has signature $(1,1)$. After making a linear change of coordinates,
we may assume 
$$
\rho=dx^1\otimes dx^2+dx^2\otimes dx^1=\left(\begin{array}{ll}0&1\\1&0\end{array}\right)\,.
$$
This normalizes the change of coordinates to the appropriate orthogonal group
\begin{equation}\label{E3.a}
\begin{array}{l}
SO(1,1):=\{T_a\in\operatorname{GL}(2,\mathbb{R}):T_a:(x^1,x^2)\rightarrow(ax^1,a^{-1}x^2)\text{ for }a\ne0\},\\[0.05in]
O(1,1):=SO(1,1)\cup \tilde T\cdot SO(1,1)\text{ for }\tilde T:(x^1,x^2)\rightarrow(x^2,x^1)\,.
\end{array}\end{equation}

In this case we do not have a single parametrization of
	$\mathfrak{Z}_0$ like we do for $\mathfrak{Z}_\pm$; there are in fact two parametrizations which give rise
	to the two boundary curves for 	$\mathfrak{Z}_0$ that must be glued along their common border $xy = 1$.

We set $\Gamma_{11}{}^1=\Gamma_{12}{}^2+\alpha$ and
$\Gamma_{22}{}^2=\Gamma_{12}{}^1+\beta$ to obtain
\begin{eqnarray*}
\rho&=&\left(\begin{array}{ll}
\beta\Gamma_{11}{}^2+\alpha\Gamma_{12}{}^2&
\Gamma_{12}{}^1\Gamma_{12}{}^2-\Gamma_{11}{}^2\Gamma_{22}{}^1\\
\Gamma_{12}{}^1\Gamma_{12}{}^2-\Gamma_{11}{}^2\Gamma_{22}{}^1&
\beta\Gamma_{12}{}^1+\alpha\Gamma_{22}{}^1\end{array}\right)\\
&=&\left(\begin{array}{ll}
(\Gamma_{12}{}^2,\Gamma_{11}{}^2)\cdot(\alpha,\beta)&(\Gamma_{12}{}^2,\Gamma_{11}{}^2)\cdot(\Gamma_{12}{}^1,-\Gamma_{22}{}^1)\\
(\Gamma_{12}{}^2,\Gamma_{11}{}^2)\cdot(\Gamma_{12}{}^1,-\Gamma_{22}{}^1)&(\Gamma_{22}{}^1,\Gamma_{12}{}^1)\cdot(\alpha,\beta).
\end{array}\right)\,.
\end{eqnarray*}
Suppose $(\alpha,\beta)\ne(0,0)$. Then $(-\beta,\alpha)$ is a non-zero vector perpendicular to $(\alpha,\beta)$.
Setting $\rho_{11}=\rho_{22}=0$ yields $(\Gamma_{12}{}^2,\Gamma_{11}{}^2)=a(-\beta,\alpha)$
and $(\Gamma_{22}{}^1,\Gamma_{12}{}^1)=b(-\beta,\alpha)$ for some $a,b\in\mathbb{R}$.
This implies that
$$
\rho_{12}=(\Gamma_{12}{}^2,\Gamma_{11}{}^2)\cdot(\Gamma_{12}{}^1,-\Gamma_{22}{}^1)=ab(-\beta,\alpha)\cdot(\alpha,\beta)=0
$$
and $\rho=0$. Thus $(\alpha,\beta)=(0,0)$. This yields the ansatz:
$$
\Gamma_{11}{}^1=\Gamma_{12}{}^2=x,\ \ \Gamma_{12}{}^1=\Gamma_{22}{}^2=y,\ \ 
\Gamma_{11}{}^2\Gamma_{22}{}^1=xy-1\ \ \text{for some}\ \ (x,y)\in\mathbb{R}^2\,.
$$
The map $T_a$ of Equation~(\ref{E3.a}) multiplies $\Gamma_{11}{}^2$ by $a^3$ and $\Gamma_{22}{}^1$ by $a^{-3}$.
Suppose first $xy-1>0$. This implies $\Gamma_{11}{}^2\Gamma_{22}{}^1>0$. By choosing $a$ appropriately, we
may assume $\Gamma_{11}{}^2=\Gamma_{22}{}^1$ is positive and obtain the parametrization:
\begin{equation}\label{E3.b}
\begin{array}{l}
\Gamma_{0,1}(x,y):=\{
\Gamma_{11}{}^1=\Gamma_{12}{}^2=x,\quad\Gamma_{11}{}^2=\sqrt{xy-1},\\
\hphantom{\Gamma_0^1(x,y):=\{..}
\Gamma_{22}{}^1=\sqrt{xy-1},\quad\Gamma_{12}{}^1=\Gamma_{22}{}^2=y\}\,.
\end{array}\end{equation}
On the other hand, if $xy-1<0$, then $\Gamma_{11}{}^2\Gamma_{22}{}^1$ is negative and again by choosing $a$ appropriately
to define $T_a$,
we may assume $\Gamma_{11}{}^2=-\Gamma_{22}{}^1$ and $\Gamma_{11}{}^2$ is positive and obtain the parametrization:
\begin{equation}\label{E3.c}
\begin{array}{l}
\Gamma_{0,2}(x,y):=\{
\Gamma_{11}{}^1=\Gamma_{12}{}^2=x,\quad\Gamma_{11}{}^2=\sqrt{1-xy},\\
\hphantom{\Gamma_0^1(x,y):=\{..}
\Gamma_{22}{}^1=-\sqrt{1-xy},\quad\Gamma_{12}{}^1=\Gamma_{22}{}^2=y\}\,.
\end{array}\end{equation}
The case $xy=1$ is exceptional. Set
$$
\Theta_{0,i}(x,y):=(\psi_3(\Gamma_{0,i}(x,y)),\Psi_3(\Gamma_{0,i}(x,y))\text{ for }i=1,2\,.
$$
We compute:
\begin{equation}\label{E3.d}
\begin{array}{l}
\Theta_{0,1}(x,y)=(-2 + 8 x y,1-4 x y+8 x^2 y^2-4 (x^3 +y^3)\sqrt{x y-1}),\\
\Theta_{0,2}(x,y)=(-2 + 8 x y,1-4 x y+8 x^2 y^2+4 (x^3-y^3) \sqrt{1-x y}).
\end{array}
\end{equation}
We distinguish cases:
\subsection*{Case 3.1. Suppose $xy>1$}
 Adopt the notation of Equation~(\ref{E3.b}).
By interchanging $x^1$ and $x^2$ we can interchange the roles of $x$ and $y$ to assume $|x|\ge|y|$.
It is natural to make the change of variables:
\begin{equation}\label{E3.e}
(s(x,y),t(x,y)):=(xy,x/\sqrt s)\text{ and }(x(s,t),y(s,t)):=(\sqrt st,\sqrt st^{-1})\,.
\end{equation}
To ensure $xy>1$, we take $s>1$. To ensure $|x|\ge|y|$, we take $|t|\ge1$. Then: 
$$
\Theta_{0,1}(s(x,y),t(x,y))=\big(-2+8s,1-4 s+8s^2-4s^{3/2}\sqrt{s-1}(t^3+t^{-3})\big)\,.
$$
We may recover $s$ from the first coordinate and $t^3+t^{-3}$ from the second coordinate. The function $t\rightarrow t^3+t^{-3}$ is
monotonically increasing and negative for $t\in(-\infty,-1)$ and monotonically increasing and positive for $t\in(1,\infty)$.
Consequently we may recover $t$ as well and conclude $\Theta_0$ is 1-1 on this region. 
We give below in Fig.3.1 a
picture of this region. The vertical straight lines correspond to appropriate values of $s=s_n$ taken to be constant. The
upper rays are given by $t\le-1$ and the lower rays by $t\ge1$.
\smallbreak\centerline{
\includegraphics[height=3.5cm,keepaspectratio=true]{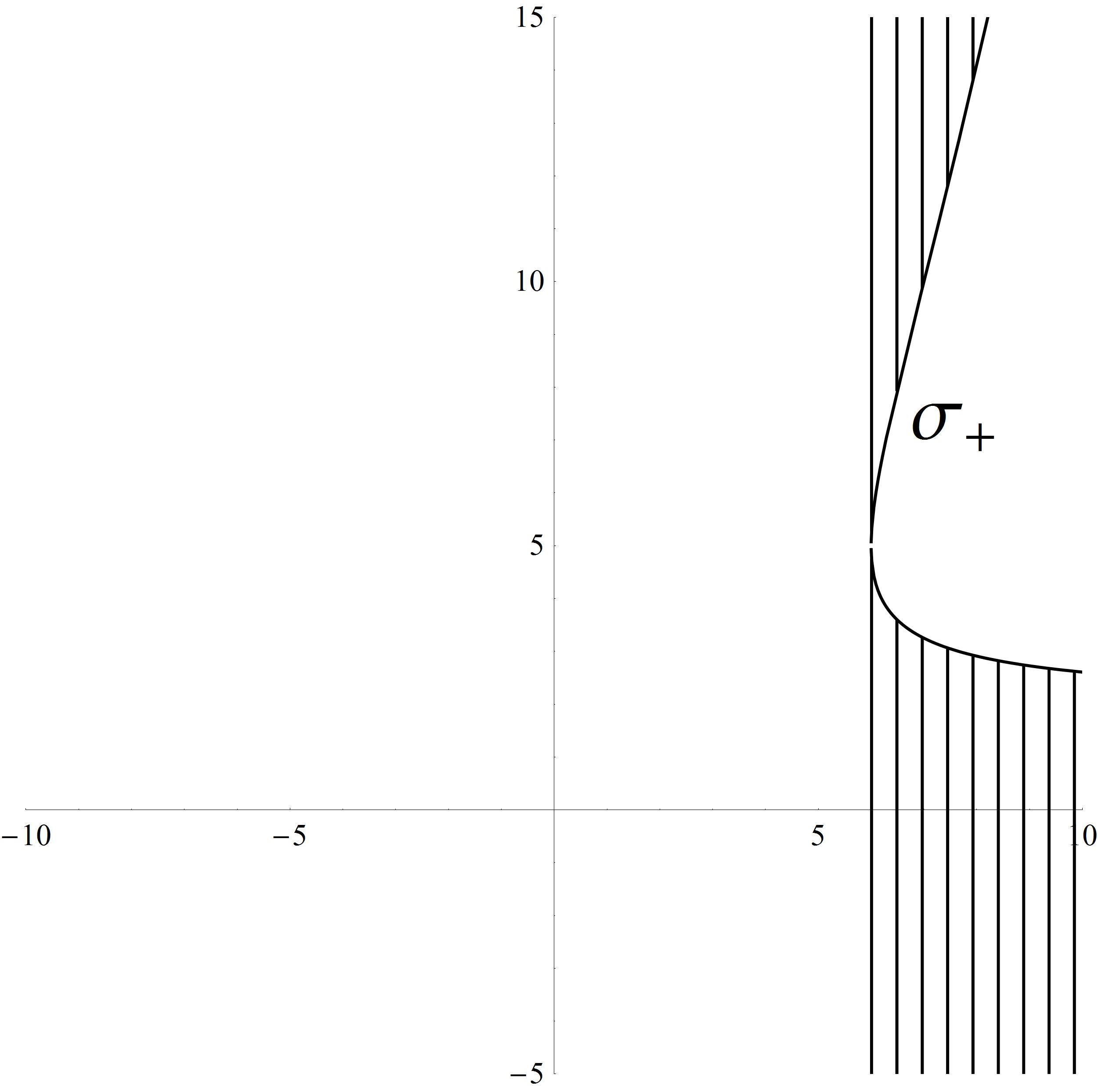}}
\smallbreak\centerline{Fig.~3.1$\Uparrow$}

\subsection*{Case 3.2. Suppose that $xy=1$}
We have $\Gamma_{11}{}^2\Gamma_{22}{}^1=0$.
If
$\Gamma_{22}{}^1=0$, we obtain the point $(6,5)$ which was missing
from the boundary curve discussed in Case 3.1 above. On the other hand, if $\Gamma_{22}{}^1$ is non-zero,
we can normalize $\Gamma_{22}{}^1=1$ and obtain $\Psi_3=5-4xy^{-2}=5-4x^3$ where $x\ne0$. This
yields the remainder of the line $\psi_3=6$. This is pictured below in Fig.~3.2.

\subsection*{Case 3.3. Suppose $0<xy<1$} 
The map $(x^1,x^2)\rightarrow(-x^2,-x^1)$ replaces $(x,y)$ by $(-y,-x)$.
Thus as $x$ and $y$ have the same sign, we may assume without loss of generality $x>0$ and $y>0$ in
using the parametrization $\Gamma_{0,2}$ of Equation~(\ref{E3.c}). 
We make the change of variables $x=\sqrt st$, $y=\sqrt s/t$ given in Equation~(\ref{E3.e}) where $0<s<1$. We take $t>0$ since
$x>0$. We use Equation~(\ref{E3.d}) to obtain
$$
\Theta_{0,2}(s(x,y),t(x,y))=(-2+8s,1-4s+8s^2+4 s^{\frac32}\sqrt{1-s}(t^3-t^{-3})\,.
$$
The parameter $s$ is determined by $\psi_3$
and $t^3-t^{-3}$ is determined by $\Psi_3$. Since $t^3-t^{-3}$ is positive and
monotonic increasing for $t\in(0,\infty)$, we may compute $t$. Consequently
$(\psi_3,\Psi_3)$ is $1-1$ in this range and completely fills up the region $-2<\psi_3<6$ and 
$-\infty<\Psi_3<\infty$. This is pictured  below in Fig. 3.3. Note that if $0 < xy < 1$ then $-2 < xy < 6$.
   
\subsection*{Case 3.4. Suppose $xy=0$.}
We set $x=0$ in $\Gamma_{0,2}(x,y)$. We obtain $\psi_3=-2$ and
$\Psi_3=1-4y^3$. So this fills up the complete line. This is pictured below in Fig. 3.4:
\smallbreak\centerline{
\includegraphics[height=3.5cm,keepaspectratio=true]{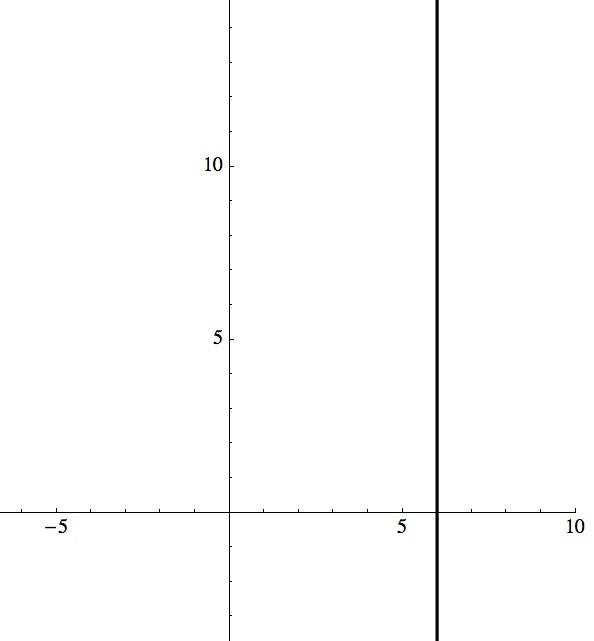}\qquad
\includegraphics[height=3.5cm,keepaspectratio=true]{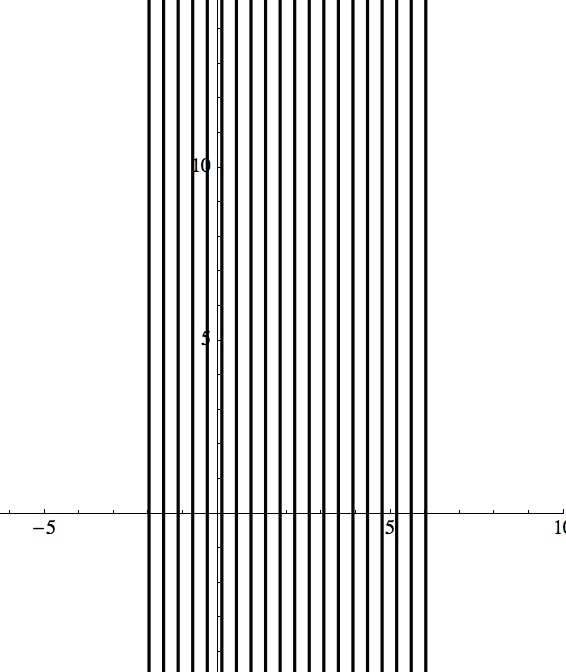}\qquad
\includegraphics[height=3.5cm,keepaspectratio=true]{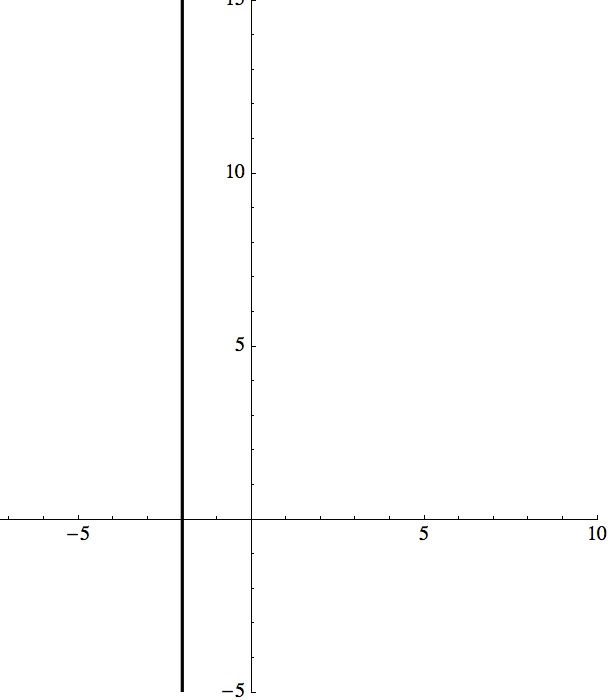}}
\smallbreak\centerline{Fig.~3.2$\Uparrow$\qquad\qquad\qquad Fig.~3.3$\Uparrow$\qquad\qquad\qquad Fig.~3.4$\Uparrow$}

\subsection*{Case 3.5. Suppose $xy<0$}
The map
$(x^1,x^2)\rightarrow(-x^2,-x^1)$ preserves the normalization. So we can assume $|x|\ge|y|$. We now make the change of variables
$$
(s(x,y),t(x,y)):=(-xy,x/\sqrt s)\text{ and }(x(s,t),y(s,t)):=(\sqrt s t,-\sqrt s t^{-1})\,.
$$
We then have $s>0$, and $|t|\ge1$. We compute
$$
\Theta_{0,2}(\Gamma_{0,2}(x(t,s),y(t,s)))=(-2-8s,1+4s+8s^2+4s^{3/2}(t^3 +t^{-3})\sqrt{1+s}\,.
$$
Thus $s$ is determined and $t^3 +t^{-3}$ is determined. This is negative and monotonically increasing for $-\infty<t\le-1$
and positive and monotonically increasing for
$1\le t<\infty$ and hence $t$ is determined. Thus $\Theta$ is 1-1. This region is pictured below in
Fig.~3.5; the full region $\Theta_0(\mathfrak{Z}_0)$ is pictured in Fig.~3.6:
\smallbreak\centerline{
\includegraphics[height=3.5cm,keepaspectratio=true]{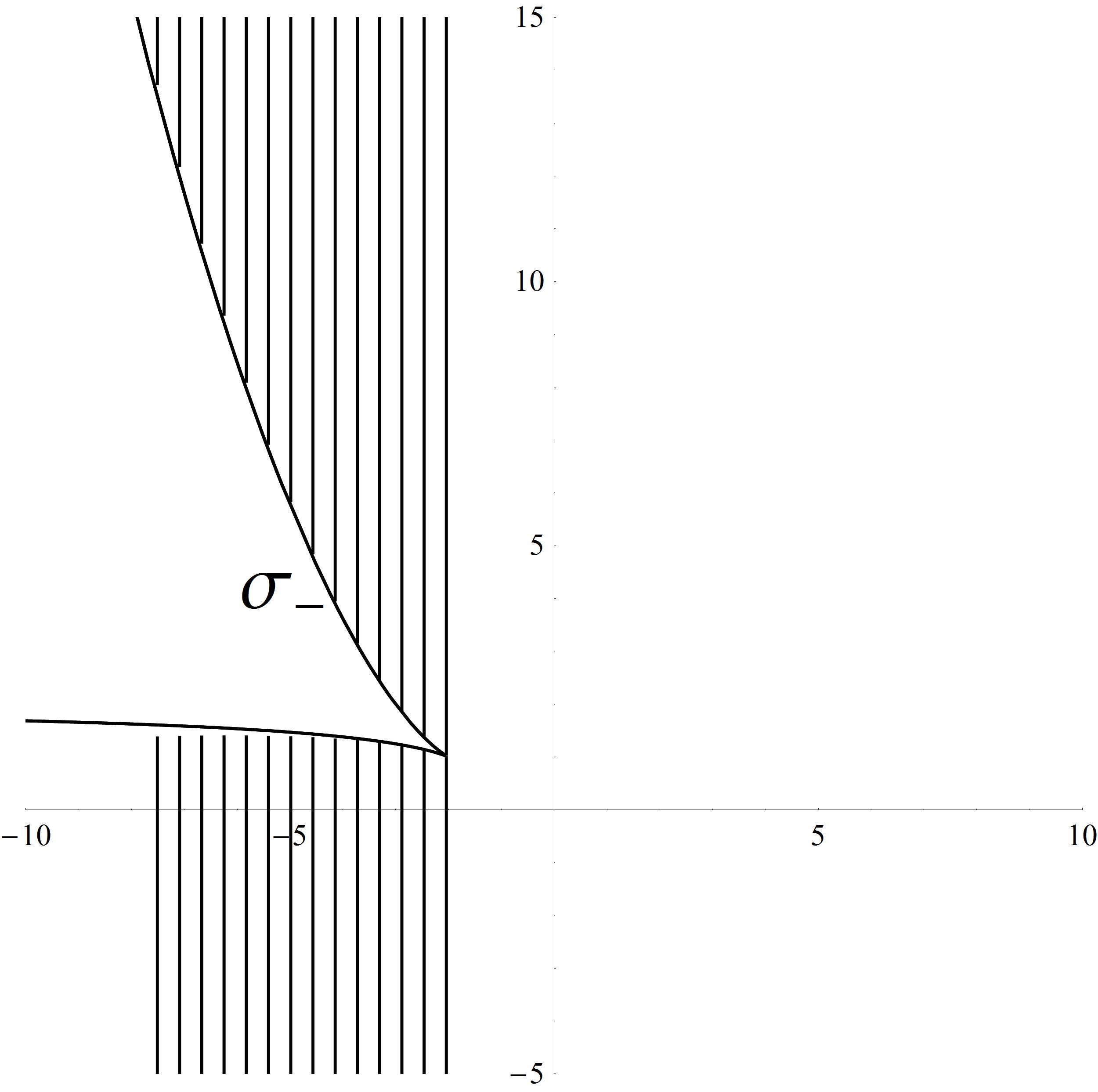}\quad
\includegraphics[height=3.5cm,keepaspectratio=true]{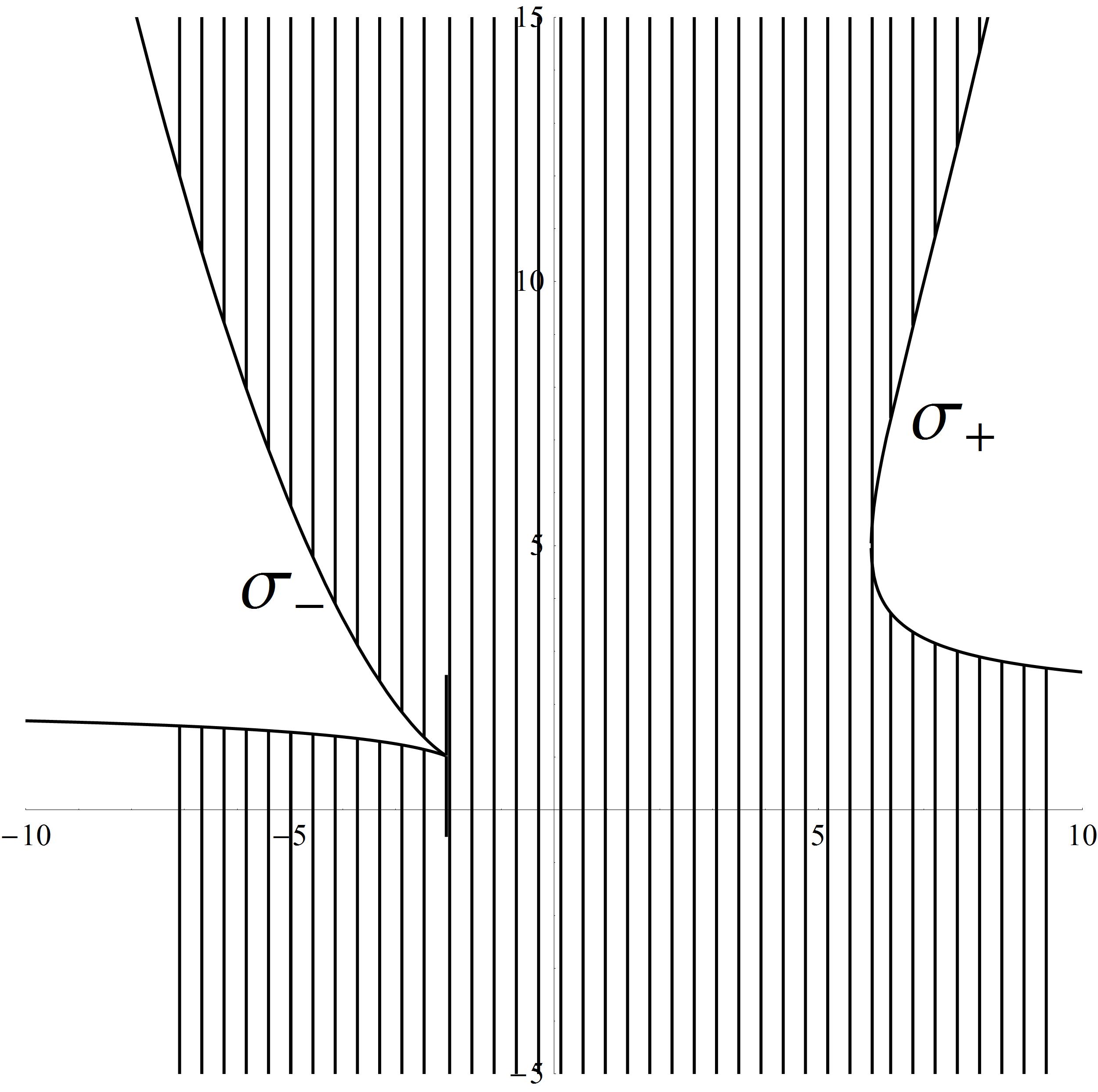}}
\smallbreak\centerline{Fig.~3.5$\Uparrow$\qquad\qquad\quad Fig.~3.6$\Uparrow$}

\section{The proof of Theorem~\ref{T1.6} and Theorem~\ref{T1.7}}\label{S4}
\subsection{The boundary curves}
In this section, we will prove Theorem~\ref{T1.6}.
The boundary of $\mathfrak{Z}_0$ consists of two pieces. We shall identify the right hand component of the boundary of $\mathfrak{Z}_0$
with the left hand boundary of $\mathfrak{Z}_+$ and the left hand component of the boundary of $\mathfrak{Z}_0$ with the right hand
boundary of $\mathfrak{Z}_-$.
\smallbreak\centerline{\includegraphics[height=3.2cm,keepaspectratio=true]{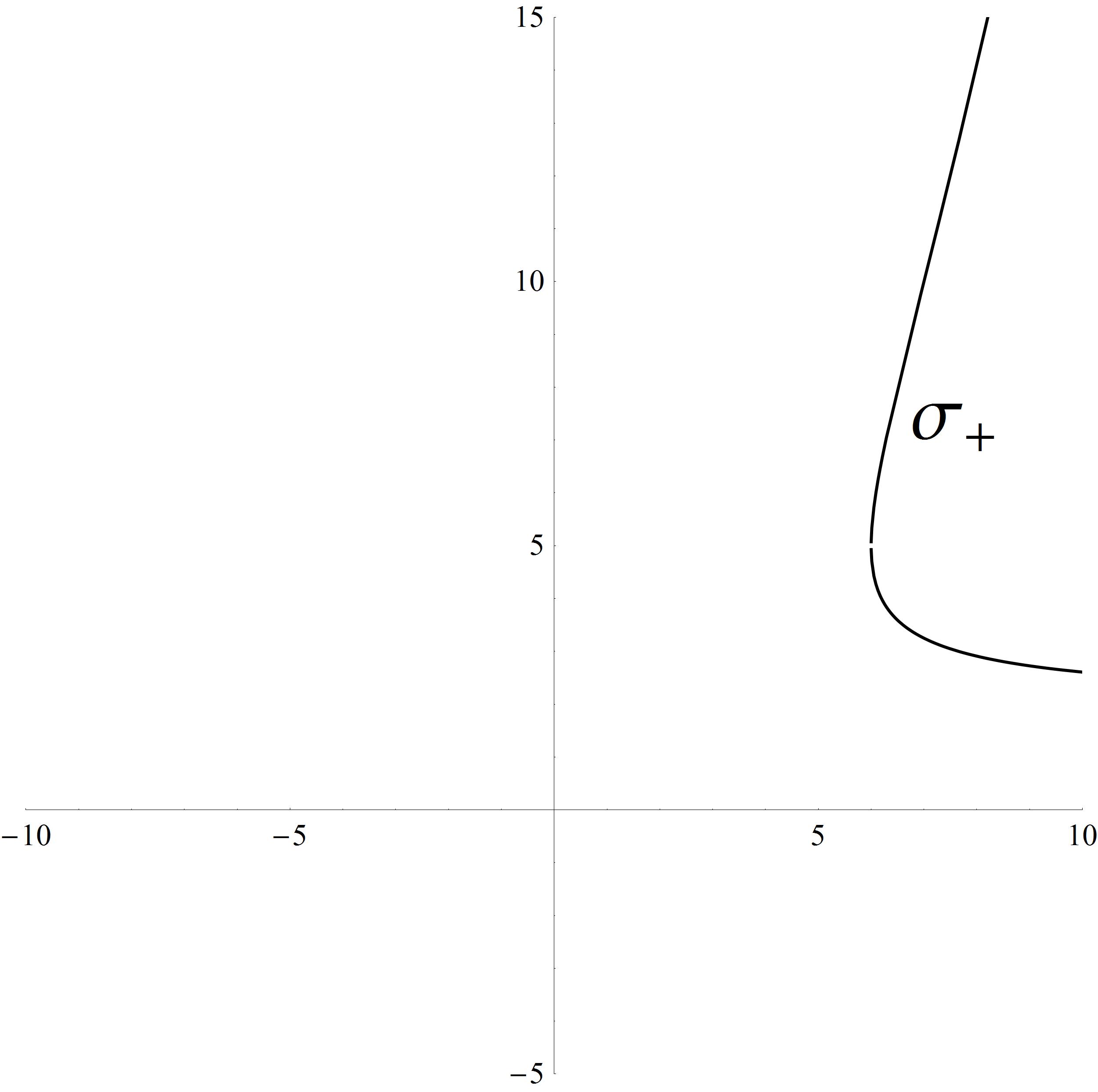}}
\smallbreak\centerline{Fig.~4.1$\Uparrow$}
\subsection*{The right boundary of $\mathfrak{Z}_0$ and the left boundary of $\mathfrak{Z}_+$}
We recall the parametrizations defined previously 
$$\begin{array}{l}
\Gamma_+(x,y):=\ \left\{\begin{array}{lll}
\Gamma_{11}{}^1= x+\frac{1}{x},&\Gamma_{11}{}^2= 0,\qquad\quad\ &\Gamma_{12}{}^1= 0,\\
\Gamma_{12}{}^2= x,&\Gamma_{22}{}^1= x,&\Gamma_{22}{}^2= y.\end{array}\right\}\\[0.15in]
\Gamma_{0,1}(x,y):=\left\{\begin{array}{lll}
\Gamma_{11}{}^1= x,\qquad&\Gamma_{11}{}^2= \sqrt{x y-1},&\Gamma_{12}{}^1= y,\\
\Gamma_{12}{}^2= x,&\Gamma_{22}{}^1= \sqrt{x y-1},&\Gamma_{22}{}^2= y.
\end{array}\right\}
\end{array}$$ 
The left boundary of $\Theta_+(\mathfrak{Z}_+)$ is the curve $\sigma_+(u):=\Theta_+(\Gamma_+(u,0))$ 
and the right boundary of $\Theta_{0,1}(\mathfrak{Z}_0)$ 
is the curve $\sigma_{0,1}(v):=\Theta_{0,1}(\Gamma_{0,1}(v,v))$. We have (see Equation~(\ref{E1.c})) that:
\begin{eqnarray*}
&&\sigma_+(u)=(2+u^{-2}+4u^2,2+4u^2+4u^4)\text{ for }0<u,\\
&&\sigma_{0,1}(v)=(-2+8v^2,1-4v^2+8v^4-8v^3\sqrt{v^2-1})\text{ for }|v|>1\,.
\end{eqnarray*}
We set $-2+8v^2=2+u^{-2}+4u^2$ to see $8v^2=4+u^{-2}+4u^2=(2 u+u^{-1})^2$.
Thus
$$\sqrt{v^2-1}=2^{-3/2}\sqrt{8v^2-8}=2^{-3/2}\sqrt{(2 u+u^{-1})^2-8}=2^{-3/2}\sqrt{(2u-u^{-1})^2}\,.$$
There are two possible choices of the square root.
\begin{enumerate}
\item We set $v=2^{-3/2} (2u+u^{-1})\ge1$,
$\sqrt{v^2-1}=2^{-3/2}(u^{-1}-2u)$, $0<u<\frac1{\sqrt2}$ to obtain
$1-4v^2+8v^4-8v^3\sqrt{v^2-1}=2+4u^2+4u^4$. This is the bottom curve marked in red in Fig. 4.1.
\item We set $v=-2^{-3/2}(2u+u^{-1})\le -1$,
$\sqrt{v^2-1}=2^{-3/2}(2u-u^{-1})$, $\frac1{\sqrt2}<u$ to obtain
$1-4v^2+8v^4-8v^3\sqrt{v^2-1}=2+4u^2+4u^4$. This is the upper curve marked in blue in Fig.~4.1 above.
\end{enumerate}

\subsection*{The left boundary of $\mathfrak{Z}_0$ and the right boundary of $\mathfrak{Z}_-$}
We recall the parametrization defined previously 
$$\begin{array}{l}
\Gamma_-(x,y):=\ \left\{\begin{array}{lll}
\Gamma_{11}{}^1= x-\frac{1}{x},&\Gamma_{11}{}^2= 0,\qquad\quad\ &\Gamma_{12}{}^1= 0,\\
\Gamma_{12}{}^2= x,&\Gamma_{22}{}^1= x,&\Gamma_{22}{}^2= y.\end{array}\right\}\\[0.15in]
\Gamma_{0,2}(x,y):=\left\{\begin{array}{lll}
\Gamma_{11}{}^1= x,\qquad&\Gamma_{11}{}^2= \sqrt{1-x y},&\Gamma_{12}{}^1= y,\\
\Gamma_{12}{}^2= x,&\Gamma_{22}{}^1=- \sqrt{1-x y},&\Gamma_{22}{}^2= y.
\end{array}\right\}
\end{array}$$
The right boundary of $\Theta_-(\mathfrak{Z}_-)$ is the curve $\sigma_-(u):= \Omega_-(\Theta-(u,0))$ and the left boundary of 
$\Theta_{0,2}(\mathfrak{Z}_0)$ is the curve $\sigma_{0,2}(v):= \Theta_{0,2}(\Gamma_{0,2}(v,-v))$. We have
\begin{eqnarray*}
&&\sigma_-(u)=(2-4 u^2-u^{-2},4 u^4-4 u^2+2),\text{ for } 0<u\\
&&\sigma_{0,2}(v)=( -2-8v^2,1+4v^2+8v^4+8v^3\sqrt{1+v^2}), \text{ for } v\in \mathbb{R}.
\end{eqnarray*}
We solve the first equation to see 
$$8v^2=(2u-u^{-1})^2\text{ and }\sqrt{v^2+1}=2^{-3/2}\sqrt{(2u+u^{-1})^2}\,.$$
Again, there are two choices of the square root.
\begin{enumerate}
\item We set $v=2^{-3/2}(2u-u^{-1})<0$, $0<u<\frac1{\sqrt 2}$, to obtain the lower curve in Fig.~4.2 below:
$1 + 4 v^2 + 8 v^4 + 8 v^3 \sqrt{1+v^2}=2 - 4 u^2 + 4 u^4$.
\item We set $v=2^{-3/2}(2u-u^{-1})>0$, $\frac1{\sqrt 2}<u$ to obtain the  upper curve in Fig.~4.2 below:
$1 + 4 v^2 + 8 v^4 + 8 v^3 \sqrt{1+v^2}=2 - 4 u^2 + 4 u^4$.
\end{enumerate}
\smallbreak\centerline{\includegraphics[height=3.2cm,keepaspectratio=true]{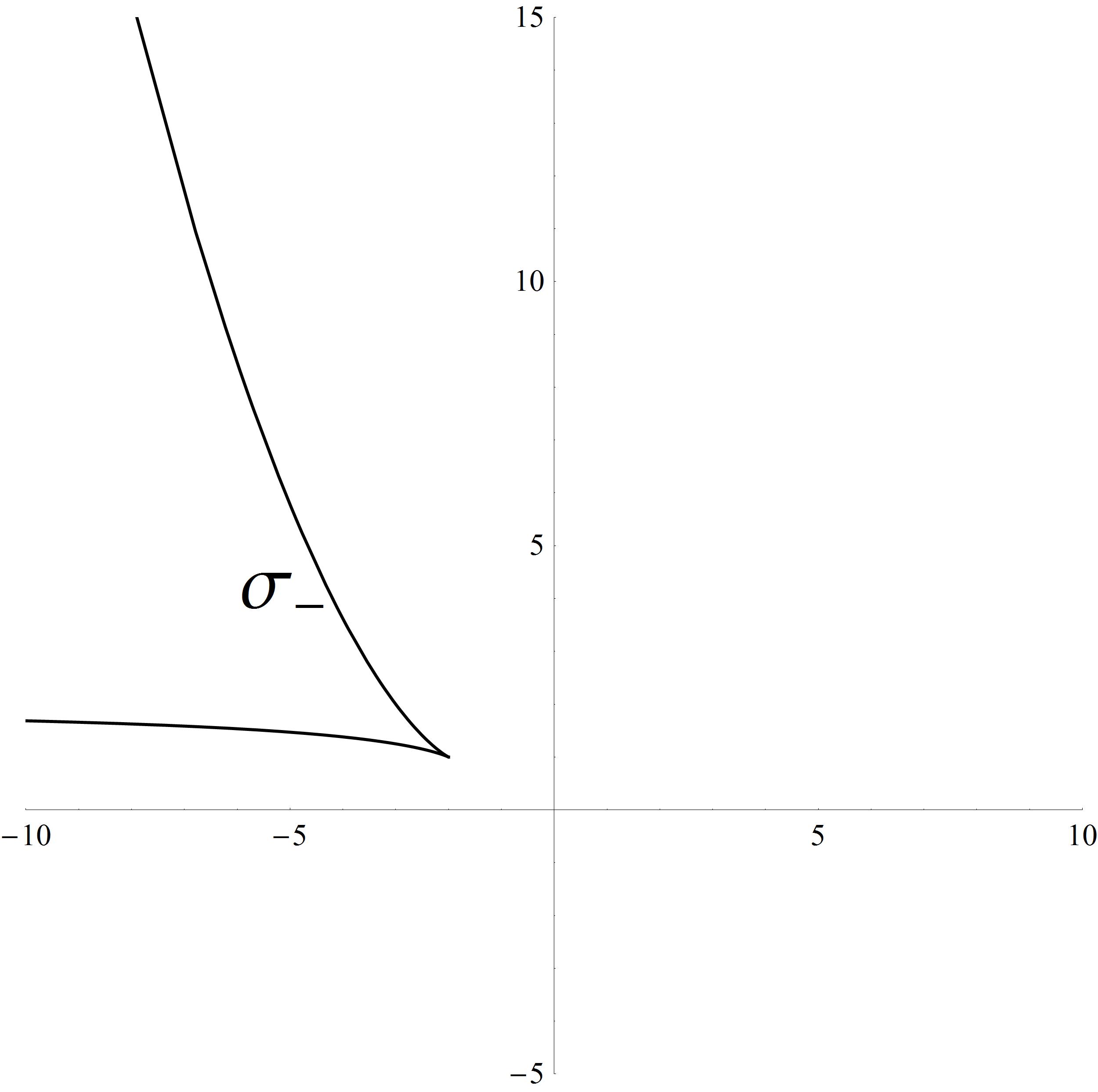}}
\smallbreak\centerline{Fig.~4.2$\Uparrow$}

\subsection{The oriented moduli space}
In this section, we shall prove Theorem~\ref{T1.7}.
Let $\Gamma\in\mathbb{R}^6$ define a Christoffel symbol with non-degenerate Ricci tensor $\rho$.
Assume $T^*\Gamma=\Gamma$ for $\operatorname{id}\ne T\in\operatorname{GL}(2,\mathbb{R})$. 

\subsection*{Suppose $\rho$ is positive definite and $\det(T)>0$}
We choose local coordinates so $\rho=\operatorname{diag}(1,1)$;
$T$ is then a rotation through an angle $\theta$. We can diagonalize $T$ over $\mathbb{C}$.
Since $\det(T)=1$, $T=\operatorname{diag}(a,a^{-1})$ for $a\in S^1$. 
We have 
$$(T^*\Gamma)_{ij}{}^k=a^\epsilon \Gamma_{ij}{}^k\text{ for }\epsilon=\pm1\pm1\pm1\in\{-3,-1,1,3\}\,.$$
If $\Gamma_{11}{}^2$ or $\Gamma_{22}{}^1$ is non-zero,
we conclude $a^{\pm3}=1$. If one of the other Christoffel symbols is non-zero, we conclude $a^{\pm1}=1$. Thus
either $T$ is the identity or $T$ is a rotation thru an angle of $\pm\frac{2\pi}3$. If $|\mathcal{G}^+(\Gamma)|\ne1$,
$T_1:=T_{\frac{ 2\pi}3}$ and $T_2:=T_{\frac{4\pi}3}$ must preserve $\Gamma$. 
We use the parametrization 
\smallbreak
$\Gamma_+(x,y):=\{
\Gamma_{11}{}^1=\textstyle x\pm\frac1x,
\Gamma_{11}{}^2=0,\Gamma_{12}{}^1=0,\Gamma_{12}{}^2=x,
\Gamma_{22}{}^1=x,\Gamma_{22}{}^2=y \}
$\smallbreak\noindent
given in Definition~\ref{D2.1}
and compute
\begin{eqnarray*}
&&0=(T_{\frac{ 2\pi}3}^*\Gamma)_{11}{}^2=(8x)^{-1}\{-\sqrt 3-2\sqrt 3x^2-3xy\} ,\\
&&0=(T_{\frac{ 4\pi}3}^*\Gamma)_{11}{}^2=(8x)^{-1}\{\sqrt 3+2\sqrt 3x^2-3xy\}\,.
\end{eqnarray*}
This implies $y=0$ and $\sqrt 3+2\sqrt 3 x^2=0$. This is not possible. Thus $|\mathcal{G}^+(\Gamma)|=1$ and
$\operatorname{GL}^{+}(2,\mathbb{R})$ acts without fixed points on $\mathcal{Z}_+$.
\subsection*{Suppose $\rho$ is positive definite and $\det(T)<0$}. 
We can choose coordinates so $\rho=\operatorname{diag}(1,1)$. Since $T^2\in\mathcal{G}^+(\Gamma)$,
we conclude $T^2=\operatorname{id}$ so $T$ is a reflection. We can normalize the choice of coordinates
so $T(x^1,x^2)=(x^1,-x^2)$ and, consequently, $(T^*\Gamma)_{12}{}^1=-\Gamma_{12}{}^1$. 
Thus $\Gamma_{12}{}^1=0$. The discussion of Section~\ref{S2}
implies $\Gamma=\Gamma_+(x,y)$ for some $(x,y)$.
A similar argument shows $\Gamma_{22}{}^2=0$ as
well. Thus $y=0$ and $\Gamma$ belongs to the boundary of $\mathfrak{C}_+$. Conversely, if $\Gamma$
belongs to the boundary of $\mathfrak{C}_+$, then $y=0$ and it follows $T^*\Gamma_+(x,0)=\Gamma_+(x,0)$.
This completes the proof of Theorem~\ref{T1.7} if $\rho$ is positive definite.
\subsection*{Suppose $\rho$ is negative definite and $\det(T)>0$} The same argument
used in the positive definite setting shows that $T$ is the identity
or is a rotation thru an angle of $\pm\frac{ 2\pi}3$. This time, we use the parametrization
\smallbreak
$\Gamma_-(x,y)=\{
\Gamma_{11}{}^1=\textstyle x-\frac1x,
\Gamma_{11}{}^2=0,\Gamma_{12}{}^1=0,\Gamma_{12}{}^2=x,
\Gamma_{22}{}^1=x,\Gamma_{22}{}^2=y\}
$\smallbreak\noindent
 and compute
\begin{eqnarray*}
&&0=(T_{\frac{ 2\pi}3}^*\Gamma)_{11}{}^2=(8x)^{-1}\{-2 \sqrt{3} x^2-3 x y+\sqrt{3}\}\\
&&0=(T_{\frac{ 4\pi}3}^*\Gamma)_{11}{}^2=(8x)^{-1}\{2 \sqrt{3} x^2-3 x y-\sqrt{3}\}\,.
\end{eqnarray*}
This implies $y=0$ and $x=\frac1{\sqrt2}$. This is, of course, the cusp point of the boundary
$$
\Gamma=\textstyle\frac1{\sqrt 2}\{\Gamma_{11}{}^1=-1,\Gamma_{11}{}^2=0,\Gamma_{12}{}^1=0,\Gamma_{12}{}^{2}=1,\Gamma_{22}{}^1= 1,\Gamma_{22}{}^2=0\}\,.
$$
And we verify explicitly $T_\theta^*\Gamma_{\text{csp}}=\Gamma_{\text{csp}}$ for $\theta=\frac{ 2\pi}3$ or 
$\theta=\frac{ 4\pi}3$.
Thus $|\mathcal{G}^+(\Gamma)|=3$ if $[\Gamma]=[\Gamma_{\text{csp}}]$ and 
$|\mathcal{G}^+(\Gamma)|=1$
otherwise.
\subsection*{Suppose $\rho$ is negative definite and $\det(T)<0$} Again, we can perform a rotation to
assume $T(x^1,x^2)=(x^1,-x^2)$. The same argument as that given in the  positive definite setting shows
$T^*\Gamma_-(x,y)=\Gamma_-(x,y)$ if and only if $y=0$, i.e. $\Gamma$ belongs to the boundary of 
$\mathfrak{C}_-$.
Thus $\mathcal{G}(\Gamma)-\mathcal{G}^+(\Gamma)$ is non-zero if and only if $y=0$;
$|\mathcal{G}(\Gamma)|=2|\mathcal{G}^+(\Gamma)|$ in this instance. This completes the proof
of Theorem~\ref{T1.7} if $\rho$ is negative definite.
\subsection*{Suppose $\rho$ is indefinite and $\det(T)>0$} We change coordinates
to ensure $\rho=dx^1\otimes dx^2+dx^2\otimes dx^1$.
In this case, $T(x^1,x^2)=(ax^1,a^{-1}x^2)$ and we
do not need to complexify. The equation $a^{\pm1}=1$ or $a^{\pm3}=1$ implies $a=1$ and $T$ is the identity.
\subsection*{Suppose $\rho$ is indefinite and $\det(T)<0$}
We have $T(x^1,x^2)=(ax^2,a^{-1}x^1)$ for $a\in\mathbb{R}$. We use the ansatz 
$$
\Gamma_{11}{}^1=\Gamma_{12}{}^2=x,\ \ \Gamma_{12}{}^1=\Gamma_{22}{}^2=y,\ \ 
\Gamma_{11}{}^2\Gamma_{22}{}^1=xy-1,
$$
for $(x,y)\in\mathbb{R}^2$ described in Section~\ref{S3}. We consider cases:
\subsection*{Case 4.1. Suppose $xy>1$}. 
We continue the discussion of Case 3.1 above and normalize 
$\Gamma_{11}{}^2=\Gamma_{22}{}^1=\sqrt{xy-1}$. Then $T^*\Gamma_{11}{}^2=a^3\Gamma_{22}{}^1$
so $a^3=1$ and $a=1$. We then have $T^*\Gamma_{11}{}^1=\Gamma_{22}{}^{2}$ so $x=y$.
Thus $\Gamma$ lies on the curve $ \sigma_{0,1}(t,t)$ which is in the right hand boundary of $\Theta_0(\mathfrak{Z}_0)$. Conversely, if $x=y$, then $T^*\Gamma=\Gamma$.
\subsection*{Case 4.2. Suppose $xy=1$.} We continue the discussion of Case 3.2 above.
We have that $\Gamma_{11}{}^2\Gamma_{22}{}^1=0$. Since $T^*$
interchanges $\Gamma_{11}{}^2$ and $\Gamma_{22}{}^1$ up to a multiple, we have
$\Gamma_{11}{}^2=\Gamma_{22}{}^1=0$. 
We can normalize $x=y=1$ to obtain the point $(6,5)$ which was missing
from the boundary curve discussed above; conversely this point is clearly invariant under the action of $T$.
\subsection*{Case 4.3. Suppose $0<xy<1$} We continue the discussion of Case 3.3 above. 
Let $\Gamma_{11}{}^2=\sqrt{1-xy}$ and $\Gamma_{22}{}^1=-\sqrt{1-xy}$. 
Since $T^*\Gamma_{11}{}^2=a^3\Gamma_{22}{}^1=-a^3\Gamma_{11}{}^2$, we have $a=-1$.
But then $T^*\Gamma_{11}{}^1=-\Gamma_{22}{}^2$ implies $x=-y$. This is not possible with $0<xy<1$.
\subsection*{Case 4.4. Suppose $0=xy$} We continue the discussion of Case 3.4 above. Since
$T$ interchanges $\Gamma_{11}{}^1$ and $\Gamma_{22}{}^2$ up to a multiple, we have $x=y=0$.
We can normalize $\Gamma_{11}{}^2=1$ and $\Gamma_{22}{}^1=-1$. This yields the cusp point
$(-2,1)$ of $\mathfrak{C}_0$ in $\mathbb{R}^2$.
\subsection*{Case 4.5. Suppose $xy<0$} We continue the discussion of Case 3.5 above. We can normalize
$\Gamma_{11}{}^2=\sqrt{1-xy}$ and $\Gamma_{22}{}^1=-\sqrt{1-xy}$. This shows $a=-1$ so $x=-y$.
This yields the boundary curve $\sigma_{0,2}$. This completes the proof of  Assertions (1), (2) and (3) in Theorem~\ref{T1.7}.

\subsection{Invariants detecting $\mathfrak{Z}_\epsilon^+$}
Let $\chi(\Gamma)$ be the invariant of Equation~(\ref{E1.d}):
$$
\chi(\Gamma):=\rho(\Gamma_{ab}{}^b\Gamma_{ij}{}^k\rho^3_{kl}\rho^{ij}dx^a\wedge dx^l,\text{dvol})\,.
$$
This is an invariant of $\mathfrak{Z}_\varepsilon^+$. We  use the parametrizations $\Gamma_\pm(x,y)$
of Definition~\ref{D2.1} and the parametrization $\Gamma_{0,1}$ and $\Gamma_{0,2}$ of
Equations~(\ref{E3.b}) and (\ref{E3.c}). We compute:
\begin{equation}\label{Ex4.a}
\begin{array}{l}
\chi_+(\Gamma_+(x,y))=yx^{-3} \left(4 x^6+x^4 y^2+x^2 \left(y^2-3\right)-1\right),\\[0.05in]
\chi_-(\Gamma_-(x,y))=yx^{-3} \left(-4 x^6-x^4 y^2+x^2 \left(y^2+3\right)-1\right),\\[0.05in]
\chi_0(\Gamma_{0,1}(x,y))=8 \sqrt{x y-1} \left(y^3-x^3\right),\\[0.05in]
\chi_0(\Gamma_{0,2}(x,y))=8 \sqrt{1-x y} \left(x^3+y^3\right)\,.
\end{array}\end{equation}
We recall Equation~(\ref{E2.c})  that the Jacobian determinant is given by
\begin{equation}\label{Ex4.b}
\begin{array}{l}
J_-(x,y):=\frac{4 \left(x^2+1\right) y \left(4 x^6+x^4 y^2-x^2 \left(y^2+3\right)+1\right)}{x^5},\\
J_+(x,y):=-\frac{4 \left(x^2-1\right) y \left(4 x^6+x^4 y^2+x^2 \left(y^2-3\right)-1\right)}{x^5}.
\end{array}\end{equation}
\subsection*{Suppose $\rho$ is positive definite} Excluding the line $x=1$, the Jacobian determinant $J_+$
given in Equation~(\ref{Ex4.b})
vanishes precisely on the $x$-axis where $y=0$ or on the Jacobi Locus $L_+$ and
these two curves are identified by $\Theta_+$. 
Thus we use Equation~(\ref{Ex4.a}) to see that $\chi_+$ vanishes precisely
on those elements of $\mathfrak{Z}_+$ which are invariant under reversing the orientation.
In particular, note that $\chi_+(x,y):=\chi(\Gamma_+(x,y))$ 
changes sign if we replace $y$ by $-y$ (or $x$ by $-x$) which
reverses the orientation as it corresponds to the coordinate change
$x^2\rightarrow-x^2$ (or $x^1\rightarrow-x^1)$.
This shows that $(\psi_3,\Psi_3,\chi)$ completely detect  $\mathfrak{Z}_+$. A picture of
this surface in $\mathbb{R}^3$ is given below. It crosses the plane where the third coordinate
(listed horizontally in stripes) vanishes; this is exactly the boundary of $\mathfrak{Z}_+$. We have doubled
this region with boundary to produce what looks like a smooth surface in $\mathbb{R}^3$.

\smallbreak\centerline{\includegraphics[height=3.5cm,keepaspectratio=true]{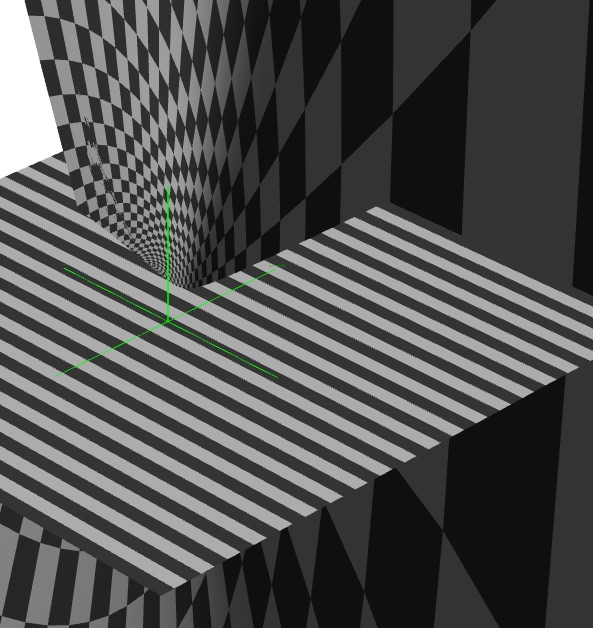}
\quad\includegraphics[height=3.5cm,keepaspectratio=true]{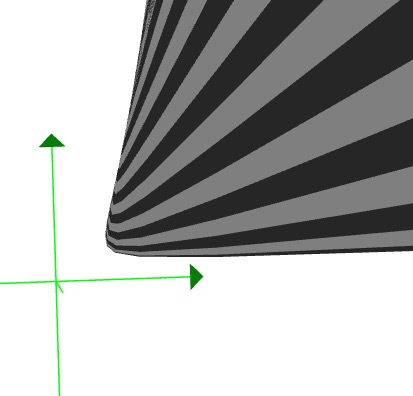}
}
\smallbreak\centerline{Fig.~4.3$\Uparrow$\qquad\qquad\qquad Fig.~4.4$\Uparrow$}

\subsection*{Suppose $\rho$ is negative definite}
The analysis is essentially the same as in the positive definite
setting. Let $\chi_-(x,y):=\chi(\Sigma_-(x,y))$. The Jacobian determinant $J_-$ vanishes exactly where 
$\chi_-$ vanishes; this is the locus of points where reversing the orientation does not change the element in
the moduli space. We present two viewpoints - one with the cusp point to the left and one to the right
giving the ``outside" and ``inside" of a surface which appears smooth in $\mathbb{R}^3$ except at the cusp
point.
\smallbreak\centerline{\includegraphics[height=3.5cm,keepaspectratio=true]{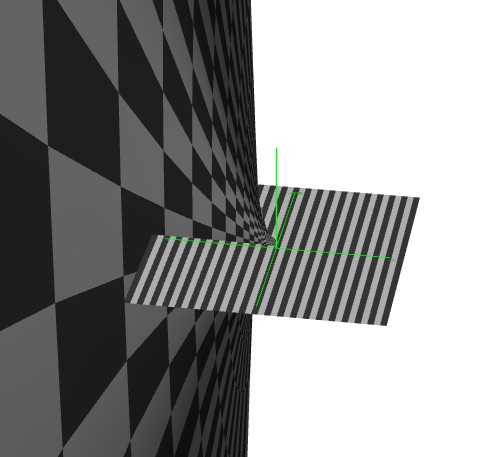}
\quad
\includegraphics[height=3.5cm,keepaspectratio=true]{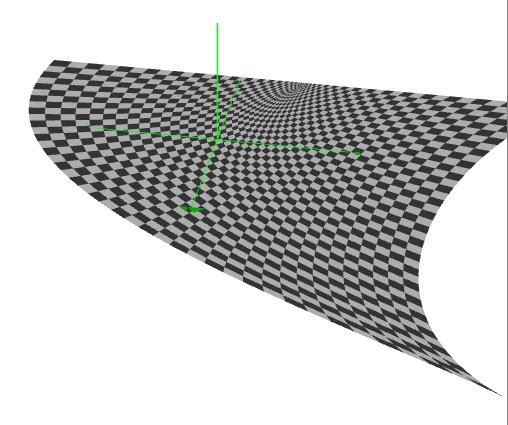}}
\smallbreak\centerline{Fig.~4.5$\Uparrow$\qquad\qquad\qquad Fig.~4.6$\Uparrow$}

\subsection*{Suppose $\rho$ is indefinite} If $xy=1$, we can change the parametrization so $x=y=1$ and
obtain a point of $\sigma_+$. Otherwise $\chi_{0,1}(x,y):=\chi(\Gamma_{0,1}(x,y))$ vanishes
precisely when $x=y$ and this is the range of $\sigma_+$ which is the right boundary of 
$\Theta_0(\mathfrak{Z}_0)$. Similarly $\chi_{0,2}(x,y):=\chi(\Gamma_{0,2}(x,y))$ 
vanishes precisely when $x=-y$ and
this is the range of $\sigma_-$ which is the left boundary of $\Theta_0(\mathfrak{Z}_0)$. Thus once again $\chi$
suffices to distinguish the relevant orientations. We omit pictures as they were not particularly useful.

\section{Type~$\mathcal{B}$ structures}\label{S5}
Let $C\in\mathcal{Z}_{23\mathcal{B}}$ define a homogeneous structure on $\mathbb{R}^+\times\mathbb{R}$
with $2\le\kappa(\mathcal{M}_C)\le3$.
Use $\pm dx^1\wedge dx^2$ to define a corresponding oriented surface $\mathcal{M}_C^\pm$. Let 
$\mathfrak{Z}_{23\mathcal{B}}^+$ be the associated moduli space of oriented surfaces with 
$2\le\kappa(\mathcal{M}_C^\pm)\le 3$. Let
$$
\begin{array}{l}
\mathfrak{G}:=\{T:
T(x^1,x^2)=(ax^1,bx^1+cx^2+d)\text{ for }a>0\text{ and }c\ne0\},\\[0.05in]
\mathfrak{I}:=\{T_{b,c}\in\mathfrak{G}:T_{b,c}(x^1,x^2)=(x^1,bx^1+cx^2)\text{ for }c\ne0\},\\[0.05in]
\mathfrak{I}^+:=\{T_{b,c}\in\mathfrak{I}\text{ for }c>0\}\,.
\end{array}$$
Our previous discussion in \cite{BGGP16} permits us to identify 
$$
\mathfrak{Z}_{23\mathcal{B}}=\mathcal{Z}_{23\mathcal{B}}/\mathfrak{I}\text{ and }
\mathfrak{Z}_{23\mathcal{B}}^+=\mathcal{Z}_{23\mathcal{B}}/\mathfrak{I}^+\,.
$$
This is a non-trivial assertion if $\kappa(\mathcal{M})=3$ as there are non-linear affine transformations.
However, they play no role in defining the affine isomorphism type.
In Lemma~\ref{L5.1}, we will define several invariant tensors on $\mathfrak{Z}_{23\mathcal{B}}$
that will play a crucial role in our subsequent discussion.  In Lemma~\ref{L5.2}, we will examine the
isotropy subgroup of the natural action of $\mathfrak{I}^+$ on $ \mathcal{Z}_{23B}$, we will
show that $\mathfrak{Z}_{23\mathcal{B}}$ is a 4-dimensional real analytic manifold, and we will
prove the natural projection $\pi^+$ from $\mathcal{Z}_{23\mathcal{B}}$ to $\mathfrak{Z}_{23\mathcal{B}}^+$ is a real-analytic $\mathfrak{J}^+$ fiber bundle.
In Lemma~\ref{L5.3}, we show that $\mathfrak{Z}_{23\mathcal{B}}^+$ is simply connected and
that the second Betti number is 2.
In Theorem~\ref{T5.4} we complete the proof of Theorem~\ref{T1.10}
by giving $\mathfrak{Z}_{23\mathcal{B}}$ a real analytic structure, by
showing that the projection from $\mathfrak{Z}_{23\mathcal{B}}^+$ to 
$\mathfrak{Z}_{23\mathcal{B}}$
is a $\mathbb{Z}_2$ branched cover where the ramification set is a real analytic sub-manifold of
co-dimension 2, and by demonstrating that $\mathfrak{Z}_{23\mathcal{B}}^+$ is simply connected and
has second Betti number equal to 1.

\begin{lemma}\label{L5.1}
The following tensors are invariantly defined on $\mathfrak{Z}_{23\mathcal{B}}$:
\begin{eqnarray*}
&&\rho_1:=\textstyle\frac1{x^1}
\{\Gamma_{12}{}^2dx^1\otimes dx^1+\Gamma_{22}{}^2dx^1\otimes dx^2
-\Gamma_{12}{}^1dx^2\otimes dx^1-\Gamma_{22}{}^1dx^2\otimes dx^2\},\nonumber\\
&&\rho_2:=\Gamma_{ij}{}^k\Gamma_{kl}{}^ldx^i\otimes dx^j,\quad
\rho_3:=\Gamma_{ik}{}^l\Gamma_{jl}{}^kdx^i\otimes dx^j,\quad
\rho_0:=\Gamma_{ij}{}^jdx^i,\\
&&\rho_4:=\Gamma_{ij}{}^kC_{ak}{}^jC_{bc}{}^idx^a\otimes dx^b\otimes dx^c\,.\nonumber
\end{eqnarray*}
\end{lemma}

\begin{proof} 
Since contracting an upper index against a lower index is an affine invariant,
$\rho_i$ for $i\ne1$ is an affine invariant and hence invariant under $\mathfrak{G}$.
Since we may express $\rho=\rho_1+\rho_2-\rho_3$, we may
conclude $\rho_1$ is invariant
under the action of $\mathfrak{G}$
although not under the action of $\operatorname{GL}(2,\mathbb{R})$.
Since $\mathfrak{Z}_{23\mathcal{B}}=\mathcal{Z}_{23\mathcal{B}}/\mathfrak{I}$, the
desired result follows.\end{proof}

Pull-back defines an action of 
$\mathfrak{I}$ on $\mathbb{R}^6$ which may be described as follows:
\medbreak\quad
$(T_{b,c})_*\partial_{x^1}=\partial_{x^1}+b \partial_{x^2},\quad(T_{b,c})_*\partial_{x^2}=c\partial_{x^2}$,
\smallbreak\quad
$(T_{b,c}^*C)_{11}{}^1=C_{11}{}^1+2bC_{12}{}^1+b^2C_{22}{}^1$,
\smallbreak\quad
$(T_{b,c }^*C)_{11}{}^2=\frac{C_{11}{}^2+b (2 C_{12}{}^2-C_{11}{}^1)+b^2 (C_{22}{}^2-2 C_{12}{}^1)-b^3 C_{22}{}^1}{c}$,
\smallbreak\quad
$(T_{b,c}^*C)_{12}{}^1=cC_{12}{}^1+bcC_{22}{}^1$,
\smallbreak\quad
$(T_{b,c}^*C)_{12}{}^2=C_{12}{}^2+bC_{22}{}^2-b(C_{12}{}^1+bC_{22}{}^1)$,
\smallbreak\quad
$(T_{b,c}^*C)_{22}{}^1=c^2C_{22}{}^1$, and $(T_{b,c}^*C)_{22}{}^2=cC_{22}{}^2-bcC_{22}{}^1$.
\medbreak\noindent
Let $\mathfrak{I}^C:=\{T_{b,c}:T_{b,c}^*C=C\}$ and $\mathfrak{I}^{+,C}:=\mathfrak{I}^C\cap\mathfrak{I}^+$ be the isotropy subgroups.

\begin{lemma}\label{L5.2}
Adopt the notation established above.
\begin{enumerate}
\item If $\dim\{\mathfrak{K}(\mathcal{M}_C)\}=4$, then $\mathfrak{I}^{+,C}$ is non-trivial.
\item If $\kappa(\mathcal{M}_C)\le3$, then $\mathfrak{I}^{+,C}$ is trivial and
the following conditions are equivalent:
\begin{enumerate}
\item $\mathfrak{I}^C$ is non-trivial.
\item $(C_{11}{}^2,C_{12}{}^1,C_{22}{}^2)=(0,0,0)$.
\item $\mathcal{M}_C$ is amphichiral, i.e. $\mathcal{M}^+_C\approx\mathcal{M}^-_C$.
\end{enumerate}

\item $\mathfrak{Z}_{23\mathcal{B}}^+$ is a 4-dimensional real analytic manifold without boundary.
The natural
projection from $\mathcal{Z}_{23\mathcal{B}}$ to $\mathfrak{Z}_{23\mathcal{B}}^+$ 
is a real-analytic principal $\mathfrak{I}^+$ fiber bundle.
\end{enumerate}\end{lemma}
 
 \begin{proof} 
If $\kappa(\mathcal{M}_C)=4$, then
$C_{12}{}^1=C_{22}{}^1=C_{22}{}^2=0$  (see \cite{BGGP16}). We have
\begin{eqnarray*}
&&(T_{b,c}^*C)_{12}{}^2=C_{12}{}^2,\quad (T_{b,c}^*C)_{11}{}^1=C_{11}{}^1,\\
&&(T_{b,c }^*C)_{11}{}^2=c^{-1}\{C_{11}{}^2+2bC_{12}{}^2-bC_{11}{}^1\}\,.
\end{eqnarray*}
We distinguish cases:
\begin{enumerate}
\item If $C_{11}{}^2=0$ and if $2C_{12}{}^2-C_{11}{}^1=0$, then $\mathfrak{I}^{+,C}=\mathfrak{I}^+$.
\item If $C_{11}{}^2\ne0$ and if $2C_{12}{}^2-C_{11}{}^1=0$, then $\mathfrak{I}^{+,C}=\{T_{b,1}\}$.
\item If $2C_{12}{}^2-C_{11}{}^1\ne0$, then $\mathfrak{I}^{+,C}$ is defined by
$b=(c-1)C_{11}{}^2\{2C_{12}{}^2-C_{11}{}^1\}^{-1}$.
\end{enumerate}

For the remainder of the proof, we suppose that
$2\le\kappa(\mathcal{M}_C)\le3$ or, equivalently, that
 $(C_{12}{}^1,C_{22}{}^1,C_{22}{}^2)\ne(0,0,0)$.
We first show that $(T_{b,c}^*C)=C$ implies $c=1$ and $b=0$.
\begin{enumerate}
\item Suppose $C_{22}{}^1\ne0$. As 
$(T_{b,c}^*C)_{22}{}^1=C_{22}{}^1$ and $c>0$, we have $c=1$. 
Setting $(T_{b,c}^*C)_{12}{}^1=C_{12}{}^1$ then implies 
$b=0$.
\item Suppose $C_{22}{}^1=0$ but $C_{12}{}^1\ne0$. As 
$(T_{b,c}^*C)_{12}{}^1=C_{12}{}^1$, we have $c=1$.
Setting $(T_{b,c}^*)C_{11}{}^1=C_{11}{}^1$ then yields $b=0$.
\item Suppose $C_{22}{}^1=0$ and $C_{12}{}^1=0$
but $C_{22}{}^2\ne0$. As
$(T_{b,c}^*C)_{22}{}^2=C_{22}{}^2$, we have that $c=1$. Setting
$(T_{b,c}^*C)_{12}{}^2=C_{12}{}^2$ then implies $b=0$.
\end{enumerate}
Suppose $T\in\mathfrak{J}^C$. Then $T^2\in\mathfrak{J}^{+,C}$ and  
$T^2=\operatorname{id}$. Hence if $T\ne\operatorname{id}$, then 
$T(x^1,x^2)=(x^1,-x^2)$. Thus $\mathfrak{J}^C$ is non-trivial if and only if $T^*C=C$, i.e.
$C_{11}{}^2=C_{12}{}^1=C_{22}{}^2=0$. This establishes the equivalence of Assertion~2-a and
Assertion~2-b. The equivalence with Assertion~2-c is then immediate.

To prove Assertion~3,  we must construct real analytic charts on $\mathfrak{Z}_{23\mathcal{B}}^+$
and show the natural projection $\pi^+$ from $\mathcal{Z}_{23\mathcal{B}}$ to
$\mathfrak{Z}_{23\mathcal{B}}^+$ defines a real analytic principal fiber bundle. Let
\begin{eqnarray*}
	&&\mathcal{O}_0^\pm:=\{C\in\mathcal{Z}_{23\mathcal{B}}:\pm\rho_0(\partial_2)>0\},\\
	&&\mathcal{O}_1^\pm:=\{C\in\mathcal{Z}_{23\mathcal{B}}:\pm\rho_1(\partial_2,\partial_2)>0\}.
\end{eqnarray*}
Since $(T_{b,c})_*(\partial_2)=c\partial_2$, these sets are open subsets of $\mathbb{R}^6$
which are invariant under the 
action of $\mathfrak{I}^+$.  If $\rho_0(C)(\partial_2)=0$ and $\rho_1(C)(\partial_2,\partial_2)=0$, then 
$C_{12}{}^1+C_{22}{}^2=0$ and $C_{22}{}^1=0$ and thus $C_{12}{}^1\ne0$ since
$C\in\mathcal{Z}_{23\mathcal{B}}$.
Let
\begin{eqnarray}
	&&\mathcal{E}:=\mathcal{Z}_{23\mathcal{B}}\cap\{\mathcal{O}_0^+\cup\mathcal{O}_0^-
	\cup\mathcal{O}_1^+\cup\mathcal{O}_1^-\}^c\nonumber\\
	&&\phantom{.a.}=\{C:C_{12}{}^1+C_{22}{}^2=C_{22}{}^1=0\}\cap\mathcal{Z}_{23\mathcal{B}}\label{5.31}\\
	&&\phantom{.a.}\subset\{C:C_{12}{}^1\ne0\}\nonumber
\end{eqnarray}
be the exceptional set. We will construct equivariant
charts covering $\mathcal{E}$ when we discuss Case 5.1.3 below; the description is a bit
complicated so we postpone a precise definition until that time.
\subsection*{\bf Case 5.1.1. Let $C\in\mathcal{O}_0^\pm$} 
We have $\rho_0(C)(\partial_{x^2})\ne0$. Let 
$$b=b(C):=\mp\rho_0(C)(\partial_{x^1})\text{ and }c=c(C):=|\rho_0 (C)(\partial_{x^2})|^{-1}\,.$$
Let $\tilde C=\tilde C(C):=T_{b,1}^*T_{0,c}^*(C)$. We then have the relations
\begin{equation}\label{Eq5.a}
	\rho_{0,\tilde C}(\partial_{x^1})=0\text{ and }\rho_{0,\tilde C}(\partial_{x^2})=\pm1\,.
\end{equation}
Furthermore, $T(C):=T_{b,0}T_{0,c}\in\mathfrak{I}^+$ is uniquely determined by the
requirement that $\tilde C:=T(C)^*(C)$ satisfies the relations of Equation~(\ref{Eq5.a}), i.e.
$\tilde C_{12}{}^1+\tilde C_{22}{}^2=\pm1$ and $\tilde C_{11}{}^1+\tilde C_{12}{}^2=0$.
Introduce local coordinates by
setting:
$$\begin{array}{lll}
\tilde C_{11}{}^1=z_0^1,&
\tilde C_{11}{}^2=z_0^2,&
\tilde C_{12}{}^1=z_0^3,\\[0.05in]
\tilde C_{12}{}^2=-z_0^1,&
\tilde C_{22}{}^1=z_0^4,&
\tilde C_{22}{}^2=\pm1-z_0^3.
\end{array}$$
This shows that $\mathcal{O}_0^\pm\rightarrow\pi^+(\mathcal{O}_0^\pm)$ 
is a $\mathfrak{I}^+$ principal bundle where $\mathfrak{U}_0^\pm:=\pi^+(\mathcal{O}_0^\pm)$
can be identified with the open subset of $\mathbb{R}^4$ where $\rho_0(\tilde C(\vec z_0))\ne0$. 
Since $\tilde C_{12}{}^1+\tilde C_{22}{}^2\ne0$,
$2\le\kappa(\tilde C(\vec z_0))\le3$. 

\subsection*{\bf Case 5.1.2. Let $C\in\mathcal{O}_1^\pm$} 
We apply the same argument as that given above to see there is a unique $T(C)\in\mathfrak{I}^+$
so that $\tilde C:=T(C)^*C$ satisfies $\rho_{1,\tilde C}(\partial_{x^2}, \partial_{x^2})=\pm1$ and
$\rho_{1,\tilde C}^s(\partial_{x^1},\partial_{x^2})=0$, i.e.
$\tilde C_{22}{}^1=\mp1$ and $\tilde C_{22}{}^2-\tilde C_{12}{}^1=0$.
We introduce local coordinates by setting:
$$\begin{array}{lll}
\tilde C_{11}{}^1=z_1^1,&
\tilde C_{11}{}^2=z_1^2,&
\tilde C_{12}{}^1=z_1^3,\\[0.05in]
\tilde C_{12}{}^2=z_1^4,&
\tilde C_{22}{}^1=\mp 1,&
\tilde C_{22}{}^2=z_1^3.
\end{array}$$
This shows that $\mathcal{O}_1^\pm\rightarrow\pi^+(\mathcal{O}_1^\pm)$ 
is a $\mathfrak{I}^+$ principal bundle where $\mathfrak{U}_1^\pm=\pi^+(\mathcal{O}_1^\pm)$
can be identified  with the open subset of $\mathbb{R}^4$ where $\rho_1(\tilde C(\vec z_1))\ne0$. Since $\tilde C_{22}{}^1\ne0$,
$2\le\kappa(\tilde C(\vec z_1))\le3$. 

\subsection*{Case 5.1.3} The final charts $\mathcal{O}_3^\pm$ are a bit more complicated to define as $\rho_3$
is purely quadratic and does not contain any linear terms.
We adopt the notation of Equation~(\ref{5.31}) and consider the exceptional set
$$\mathcal{E}:=\{C:C_{12}{}^1\ne0, C_{22}{}^2=-C_{12}{}^1, C_{22}{}^1=0\}\,.$$
We wish to fix the gauge. Let
$\tilde{\mathcal{E}}:=\{C\in\mathcal{E}:C_{12}{}^1=\pm1\text{ and }C_{11}{}^1=0\}$.
If $C\in\mathcal{E}$, let $b=b(C):=\mp\frac12 C_{11}{}^1$, let $c=c(C):=1$, and let
$\bar C=T_{b,c}^*C$. The calculations performed
in the proof of Lemma~\ref{L5.1} then show $\bar C_{11}{}^1=0$ and
$\bar C_{12}{}^1=\operatorname{sign}(C_{12}{}^1)$ or, 
equivalently, $\rho_3^{\bar C}(\partial_{x^1},\partial_{x^2})=0$
and $\rho_3^{\bar C}(\partial_{x^2},\partial_{x^2})=2{(x^1)^{-2}}$. Furthermore, $(b,c)$ are uniquely
specified by these equations. Thus
$\mathcal{E}=\tilde{\mathcal{E}}\cdot\mathfrak{I}^+$.
We use this as our ansatz; we wish to ensure $\rho_3(\partial_1,\partial_2)=0$ and
$C_{12}{}^1=\pm1$.
Let $\vec z_3:=(z_3^1,z_3^2,z_3^3,z_3^4)\in\mathbb{R}^4$, define $\tilde C^\pm(\vec z)$ by setting:
\begin{eqnarray*}
	&&\tilde C_{11}^\pm{}^2(\vec z_3):=
	z_3^1,\ \ \tilde C_{12}^\pm{}^2(\vec z_3):=z_3^2,\ \ \tilde C_{22}^\pm{}^1:=z_3^3,
	\ \ \tilde C_{22}^\pm{}^2:=z_3^4,\ \ \tilde C_{12}^\pm{}^1(\vec z_3)=\pm1,\\
	&&\tilde C_{11}^\pm{}^1(\vec z_3):=
	-\tilde C_{12}^\pm{}^2(\vec z_3)\mp\{\tilde C_{11}^\pm{}^2(\vec z_3)\tilde C_{22}^\pm{}^1(\vec z_3)
	+\tilde C_{12}^\pm{}^2(\vec z_3)\tilde C_{22}^\pm{}^2(\vec z_3)\}\,.
\end{eqnarray*}
Since $\tilde C_{12}^\pm{}^1(\vec z_3)=\pm1$, $2\le\kappa(\tilde C^\pm(\vec z_3))\le 3$ and
$\tilde C^\pm(\vec z_3)\in\mathcal{Z}_{23\mathcal{B}}$. We compute
\begin{eqnarray*}
	&&\rho_3(\partial_{x^1},\partial_{x^2})=0,\\
	&&\rho_3(\partial_{x^2},\partial_{x^2})
	=1+2C_{12}{}^2C_{22}{}^1+C_{22}{}^2C_{22}{}^2=1+2z_3^2z_3^3+z_3^4z_3^4\,.
\end{eqnarray*}
Let $\tilde{O}_3:=\{\vec z_3:1+2z_3^2z_3^3+z_3^4z_3^4>0\}$. Then
$\tilde{C}(\tilde{O}_3)$ is a neighborhood of $\tilde{\mathcal{E}}$. Suppose
$$
T_{b,c}^*\{\tilde C(\tilde O_3)\}\cap\tilde C(\tilde O_3)\ne\emptyset\,.
$$
If $b\ne0$, then  $\rho_3((T_{b,c})_*\partial_{x^1},(T_{b,c})_*\partial_{x^2})\ne0$. Consequently
$b=0$. To ensure $(T_{b,c}^*C)_{12}{}^1=\pm1$, we then need $c=1$. 
Consequently exactly as in the special case that $C\in\mathcal{E}$, $b$ and $c$ are determined.
This shows that
\begin{equation}\label{5.c}
	T_{b,c}^*\{\tilde C(\tilde O_3)\}\cap\tilde C(\tilde O_3)=\emptyset\text{ for }(b,c)\ne(0,1)\,.
\end{equation}
Clear the previous notation and set
$\Psi^\pm(\vec z_3,T_{b,c}):=T_{b,c}^*\tilde C^\pm(\vec z_3)$. The argument given above
to establish Equation~(\ref{5.c})
shows that $\Psi^\pm$ is a $\mathfrak{I}^+$ equivariant
diffeomorphism from $\tilde{O}_3\times\mathfrak{I}^+$ to
an open subset $\mathcal{O}_3^\pm$ of $\mathbb{R}^6$ which contains $\mathcal{E}$.
This provides the missing charts over which $\pi$ is
a principal bundle. 

We have constructed coordinate charts 
$\{\mathfrak{U}_0^\pm,\mathfrak{U}_1^\pm,\mathfrak{U}_3^\pm\}$
on $\mathfrak{Z}_{23\mathcal{B}}^+$ so that
$\pi^+$ admits a section and thus the bundle is trivial over these charts.
The construction is in the real analytic category
and the transition functions between coordinate charts are real analytic. 
This completes the proof of Assertion~3.
\end{proof}

\begin{remark}
\rm
A special case of the affine structures $C$  with $\kappa(\mathcal{M}_C)\leq 3$ and non-trivial $\mathfrak{J}^C$, as discussed in Lemma \ref{L5.2}-(2), is given by
$$
C_{11}^1=C_{12}^2=\pm C_{22}^1,\quad \mbox{and}\quad
C_{11}{}^2=C_{12}{}^1=C_{22}{}^2=0.
$$
These affine structures correspond to the only non-flat homogeneous affine surfaces that are metrizable and the associated pseudo-Riemannian metric is non-homogeneous \cite{KVOp}.
\end{remark}

Let
\begin{eqnarray*}
&&K:=\{C\in\mathbb{R}^6:(C_{12}{}^1,C_{22}{}^1,C_{22}{}^2)=(0,0,0)\},\\
&&P^\pm:=(C_{11}{}^1=1,C_{11}{}^2=0,C_{12}{}^1=0,C_{12}{}^2=0,
C_{22}{}^1=\pm1,C_{22}{}^2=0),\\
&&K^\pm(b,c):=\left\{\begin{array}{rrr}
C_{11}{}^1=1\pm b^2,&C_{11}{}^2=-bc^{-1}(1\pm b^2),&C_{12}{}^1=\pm bc,\\[0.05in]
C_{12}{}^2=\mp b^2,&C_{22}{}^1=\pm c^2,& C_{22}{}^2=\mp bc
\end{array}\right\}\,.
\end{eqnarray*}
Let $K^\pm:=K^\pm(\mathbb{R}\times\mathbb{R}^+)$; a direct computation shows 
$K^\pm=P^\pm\cdot\mathfrak{I}^+$.
\begin{lemma}\label{L5.3}
\ \begin{enumerate}
\item $\mathcal{Z}_{23\mathcal{B}}=\mathbb{R}^6-K-K^+-K^-$.
\item $\mathfrak{Z}_{23\mathcal{B}}^+$ is simply connected and
$\pi_2(\mathfrak{Z}_{23\mathcal{B}}^+)=\mathbb{Z}$.
\item $H_{\text{DeR}}^1(\mathfrak{Z}_{23\mathcal{B}}^+)=0$ and
$H_{\text{DeR}}^2(\mathfrak{Z}_{23\mathcal{B}}^+)=\mathbb{R}$.
\end{enumerate}\end{lemma}

\begin{proof} $C\in\mathcal{Z}_{23\mathcal{B}}$ if and only if
$(C_{12}{}^1,C_{22}{}^1,C_{22}{}^2)\ne(0,0,0)$ and $\rho(C)\ne0$.
Suppose $(C_{12}{}^1,C_{22}{}^1,C_{22}{}^2)\ne(0,0,0)$ but $\rho(C)=0$.
Because $0=\rho_{12}-\rho_{21}=(\Gamma_{12}{}^1+\Gamma_{22}{}^2)/x^1$, we may set
$\Gamma_{12}{}^1=x$ and $\Gamma_{22}{}^2=-x$.
First suppose $C_{22}{}^1\ne0$ or, equivalently, $\rho_1(\partial_{x^2},\partial_{x^2})\ne0$.
We consider the shear $T:(x^1,x^2)\rightarrow(x^1,\varepsilon x^1+x^2)$ which
sends $\partial_{x^1}$ to $\partial_{x^1}-\varepsilon \partial_{x^2}$ and $\partial_{x^2}$ to $\partial_{x^2}$.
By choosing $\varepsilon$ appropriately, we can ensure that $\rho_1(\partial_{x^2}, \partial_{x^1})=0$ or,
equivalently, that $\Gamma_{12}{}^1=0$ and set $x=0$. We rescale $x^2$ to set $C_{22}{}^1=\pm1$. 
Setting $\rho_{12}=0$ yields $C_{11}{}^2=0$. The remaining equations
become
$$C_{12}{}^2(1+C_{11}{}^1-C_{12}{}^2)=0\text{ and }(-1+C_{11}{}^1-C_{12}{}^2)=0\,.
$$
We conclude $C_{12}{}^2=0$ and $C_{11}{}^1=1$ so
$C=P^\pm$; this gives rise to the surfaces $K^\pm$.
Suppose next $C_{22}{}^1=0$. We have $\rho_{22}=-2x^2(x^1)^{-2}$ so $x=0$. This implies 
$(C_{12}{}^1,C_{22}{}^1,C_{22}{}^2)=(0,0,0)$ contrary
to our assumption. This establishes Assertion~1.

The principal bundle
$\mathfrak{I}^+\rightarrow\mathcal{Z}_{23\mathcal{B}}
\rightarrow\mathfrak{Z}_{23\mathcal{B}}^+$
gives rise to an associated long exact sequence of homotopy groups
\begin{equation}\label{E5.b}
\dots\rightarrow\pi_k(\mathfrak{I}^+)\rightarrow\pi_k(\mathcal{Z}_{23\mathcal{B}})
\rightarrow\pi_k(\mathfrak{Z}_{23\mathcal{B}}^+)\rightarrow\pi_{k-1}(\mathfrak{I}^+)\dots\,.
\end{equation}
Since $\mathfrak{I}^+$ is contractible, the $k^{\operatorname{th}}$ homotopy group
$\pi_k(\mathfrak{I}^+)=0$ for all $k$ and thus
$\pi_k(\mathcal{Z}_{23\mathcal{B}})$
is isomorphic to $\pi_k(\mathfrak{Z}_{23\mathcal{B}}^+)$ for any $k$. It therefore
suffices to prove $\mathcal{Z}_{23\mathcal{B}}$ is simply connected and
$\pi_2(\mathcal{Z}_{23\mathcal{B}})=\mathbb{Z}$.
A crucial point is that $\mathcal{Z}_{23\mathcal{B}}$ is obtained from $\mathbb{R}^6$
(which is contractible) by deleting the submanifold $K$ (which has codimension 3)
 and the two surfaces $K^\pm$ (which have codimension 4). 
 Let $P$ and $Q$ be two points in $\mathcal{Z}_{23\mathcal{B}}$. 
Find a smooth curve $\gamma$ in $\mathbb{R}^6$ from
$P$ to $Q$. Make $\gamma$ transverse to $K^\pm$ and to $K$. 
Then $\gamma$ does not intersect these submanifolds and
thus joins $P$ to $Q$ in $\mathcal{Z}_{23\mathcal{B}}$. 
Consequently $\mathcal{Z}_{23\mathcal{B}}$
is connected - i.e. $\pi_0( \mathcal{Z}_{23\mathcal{B}})=0$.
Next let $\gamma$ be a closed loop in $\mathcal{Z}_{23\mathcal{B}}$. Since $\mathbb{R}^6$
is contractible, we can construct a homotopy $H$ from $\gamma$ to the constant path in 
$\mathbb{R}^6$. Make this homotopy transverse to $K$ and to $K^\pm$. Again,
for dimensional reasons, the homotopy misses $K^\pm$ and $K$. Consequently, 
$\pi_1(\mathcal{Z}_{23\mathcal{B}})=0$.
Note that $\mathbb{R}^6-K$ is homotopy equivalent to $S^2$. 
Let $\gamma_0:S^2\rightarrow\mathbb{R}^6-K$
generate 
$$
\pi_2(\mathbb{R}^6-K)=\pi_2(S^2)=\mathbb{Z}\,.
$$
Make $\gamma_0$ transverse to $K^\pm$; again for
dimensional reasons, $\gamma_0$ misses the surfaces $K^\pm$ so we can regard 
$\gamma_0\in\pi_2(\mathcal{Z}_{23\mathcal{B}})$. Since $\gamma_0$ 
generates $\pi_2(\mathcal{R}^6-K)=\mathbb{Z}$, $\gamma_0$
generates a cyclic subgroup of $\pi_2(\mathcal{Z}_{23\mathcal{B}})$.
Let $\gamma:S^2\rightarrow\mathcal{Z}_{23\mathcal{B}}$. Construct a homotopy $H$
between $\gamma$ and $n\cdot\gamma_0$ in $\mathbb{R}^6-K$. 
Make the homotopy (which is 3-dimensional) transverse to $K^\pm$ 
(which are 2-dimensional). Thus $H$ misses $K^\pm$ and lives
in $\mathcal{Z}_{23\mathcal{B}}$. 
This shows $\pi_2(\mathcal{Z}_{23\mathcal{B}})=\mathbb{Z}$ and completes the proof
of Assertion~2.

By the Hurewicz Theorem,
$H_1(\mathfrak{Z}_{23\mathcal{B}})=0$ and 
$H_2(\mathfrak{Z}_{23\mathcal{B}})=\mathbb{Z}$.
The universal coefficient theorem shows dually
$H^1(\mathfrak{Z}_{23\mathcal{B}})=0$ and 
$H^2(\mathfrak{Z}_{23\mathcal{B}})=\mathbb{Z}$. We may use
 the de Rham theorem to identify topological cohomology with DeRham cohomology
 and verify Assertion~3.\end{proof}
 
 \begin{theorem}\label{T5.4}
Let $\mathfrak{Z}_{23\mathcal{B}}= \mathcal{Z}_{23\mathcal{B}}/\mathfrak{I}= \mathfrak{Z}_{23\mathcal{B}}^+/\{\mathfrak{I}/\mathfrak{I}^+\}$ be the moduli
space of isomorphism classes of unoriented surfaces $\mathcal{M}_C$ of Type~$\mathcal{B}$ with
$2\le\kappa(\mathcal{M}_C)\le3$. 
\begin{enumerate}
\item $\mathfrak{Z}_{23\mathcal{B}}$ is a 4-dimensional real analytic manifold.
\item The projection $\pi$
from $\mathcal{Z}_{23\mathcal{B}}$
to $\mathfrak{Z}_{23\mathcal{B}}$ is
real analytic.
\item The projection $\sigma$ from $\mathfrak{Z}_{23\mathcal{B}}^+$ to
$\mathfrak{Z}_{23\mathcal{B}}$ is a real analytic 2-sheeted branched cover which
is ramified over a real analytic manifold of dimension $2$ which is isomorphic to
$\mathbb{R}^2-0\cup\mathbb{R}^2-0$.
\item $\mathfrak{Z}_{23\mathcal{B}}$ is simply connected and 
$H^2_{\operatorname{DeR}}(\mathfrak{Z}_{23\mathcal{B}})=\mathbb{R}$.
\end{enumerate}
 \end{theorem}
 
 \begin{proof}
 Let $T(x^1,x^2)=(x^1,-x^2)$. 
 Then $[T]$ generates $\mathfrak{I}/\mathfrak{I}^+=\mathbb{Z}_2$; we could
 have taken $T_{b,-1}$ for any $b\in\mathbb{R}$, but this is a particularly
 felicitous choice. Let $\mathfrak{F}$ be the fixed point set of $T$ in 
 $\mathfrak{Z}_{23\mathcal{B}}^+$. 
 We examine the action of $T$ on the coordinate systems discussed previously
 to prove Assertion~1.

 \subsubsection*{\bf Case 5.2.1. The coordinate system $\mathcal{O}_0^\pm$}
 We have
\begin{eqnarray*}
&&\begin{array}{lll}
 (C^\pm)_{11}{}^1=z_0^{1,\pm},&
 (C^\pm)_{11}{}^2=z_0^{2,\pm},&
 (C^\pm)_{12}{}^1=z_0^{3,\pm},\\[0.05in]
 (C^\pm)_{12}{}^2=-z_0^{1,\pm},&
 (C^\pm)_{22}{}^1=z_0^{4,\pm},&
 (C^\pm)_{22}{}^2=\pm1-z_0^{3,\pm},
\end{array}\\
&&\ \ T(C^\pm(z_0^{1,\pm},z_0^{2,\pm},z_0^{3,\pm},z_0^{4,\pm}))
=C^{\mp}(z_0^{1,\mp},-z_0^{2,\mp},-z_0^{3,\mp},z_0^{4,\mp})\,.
\end{eqnarray*}
Thus, since $T^\ast\rho_0(\partial_{x^2})=-\rho_0(\partial_{x^2})$
 and $\rho_0$ is non-trivial, $[T]$ acts without fixed points and preserves
$\mathcal{O}_0^+\cup\mathcal{O}_0^-$; it is a linear
interchange of $\mathcal{O}_0^+$ and $\mathcal{O}_0^-$. 
Thus the quotient gives a coordinate
chart $\mathcal{O}_0$ on $\mathfrak{Z}_{23\mathcal{B}}$ in such a way
that  $\sigma$ is an unramified 
 real analytic double covering projection.
 \subsubsection*{\bf Case 5.2.2. The coordinate system $\mathcal{O}_3^\pm$}
We have 
\begin{eqnarray*}
&& (C^\pm)_{11}{}^1=- z_3^{2,\pm}\mp\{ z_3^{1,\pm} z_3^{3,\pm}+ z_3^{2,\pm} z_3^{4,\pm}\},\\
&&(C^\pm)_{12}{}^1=\pm1,\\
&& (C^\pm)_{11}{}^2=z_3^{1,\pm},\quad  (C^\pm)_{12}{}^2=z_3^{2,\pm},
\quad  (C^\pm)_{22}{}^1=z_3^{3,\pm},\quad  (C^\pm)_{22}{}^2=z_3^{4,\pm},\\
&&T(C^\pm(z_3^{1,\pm},z_3^{2,\pm},z_3^{3,\pm},z_3^{4,\pm}))=
(C^\mp)(-z_3^{1,\mp},z_3^{2,\mp},z_3^{3,\mp},-z_3^{4,\mp})\,.
\end{eqnarray*}
Thus $[T]$ acts without fixed points and preserves
$\mathcal{O}_3^+\cup\mathcal{O}_3^-$; it is a linear
interchange of $\mathcal{O}_3^+$ and $\mathcal{O}_3^-$ since $T^\ast (C_{12}{}^1)=-(C_{12}{}^1)$.
Thus the quotient gives a coordinate
chart $\mathcal{O}_3$ on $\mathfrak{Z}_{23\mathcal{B}}$ in such a way
that  $\sigma$ is an unramified 
 real analytic double covering projection.
\subsubsection*{\bf Case 5.2.3. The coordinate system $\mathcal{O}_1^\pm$}
Here the analysis is less trivial. We have
\begin{eqnarray*}
&&\begin{array}{lll}
 (C_\pm)_{11}{}^1=z_1^{1,\pm},&
 (C_\pm)_{11}{}^2=z_1^{2,\pm},&
 (C_\pm)_{12}{}^1=z_1^{3,\pm},\\[0.05in]
 (C_\pm)_{12}{}^2=z_1^4,&
 (C_\pm)_{22}{}^1=\pm1,&
 (C_\pm)_{22}{}^2=z_1^{3,\pm}.
\end{array}\\
&&\ \ T( C(z_1^{1,\pm},z_1^{2,\pm},z_1^{3,\pm},z_1^{4,\pm}))
=C^\pm(z_1^{1,\pm},-z_1^{2,\pm},-z_1^{3,\pm},z_1^{4,\pm})\,.
\end{eqnarray*}
Note that the exceptional points $P^\pm\in\mathfrak{F}^\pm$ and that $K$ does not meet
$\mathcal{O}_1^\pm$. In contrast to the previous two cases, $T$ preserves 
$\mathcal{O}_1^\pm$ separately 
and has non-trivial fixed point sets
$\mathfrak{F}^\pm=\{\vec z_{1,\pm}:z_1^{2,\pm}=z_1^{3,\pm}=0\}$.
We introduce complex coordinates setting
$$
w_1^{1,\pm}:=z_1^{1,\pm}+\sqrt{-1}z_1^{4,\pm}\text{ and }
w_1^{2,\pm}:=z_2^{2,\pm}+\sqrt{-1}z_1^{3,\pm}\,.
$$
Consider the map $\Phi^\pm(w_1^{1,\pm},w_1^{2,\pm}):
=(w_1^{1,\pm},(w_1^{2,\pm})^2)$. We have $\mathfrak{F}^\pm=\mathbb{C}\times\{0\}$ 
and $\Phi^\pm(w_1^{1,\pm},w_1^{2,\pm})=\Phi^\pm(\tilde w_1^{1,\pm},\tilde w_1^{2,\pm})$
 if and only if 
$\sigma(w_1^{1,\pm},w_1^{2,\pm})=\sigma(\tilde w_1^{1,\pm},\tilde w_1^{2,\pm})$.
Thus $\Phi$ extends to a 1-1
map from $\mathbb{C}^2/\mathbb{Z}_2$ to $\mathbb{C}^2$ which gives coordinates
on $\mathfrak{Z}_{23\mathcal{B}}$. This gives $\mathfrak{Z}_{23\mathcal{B}}$ the
structure of a real-analytic 4-dimensional manifold in such a way that $\sigma$ is a
real analytic double covering which is ramified along
$\mathfrak{F}^\pm=\mathbb{C}\times\{0\}$. Assertion~1 now follows.

Since $\sigma$ is surjective and $\mathfrak{Z}_{23\mathcal{B}}^+$ is connected, we
conclude $\mathfrak{Z}_{23\mathcal{B}}$ is connected. 
We give $\mathfrak{Z}_{23\mathcal{B}}$ a simplicial structure so that $[T]$ is a simplicial
map and so that $[T]$ is $1-1$ on each simplex.
This implies that  every simplicial curve $\alpha$ in $\mathfrak{Z}_{23\mathcal{B}}$ 
can be lifted to a curve
in $\tilde\alpha$ in $\mathfrak{Z}_{23\mathcal{B}}^+$. The lift will in general not
be unique even if the initial vertex is fixed
 if one of the subsequent vertices is in $\mathfrak{F}^\pm$. Choose the base point
 to be in the ramifying set. Let $\alpha$ be a closed curve in $\mathfrak{Z}_{23\mathcal{B}}$
 starting and ending in the ramifying set. 
 Lift $\alpha$ to a (possibly) not closed curve in $\mathfrak{Z}_{23\mathcal{B}}^+$. 
 Since $\sigma$ is 1-1
 on $\mathfrak{F}^{-1}$, $\tilde\alpha$ is in fact a closed curve in 
 $\mathfrak{Z}_{23\mathcal{B}}^+$. Since $\mathfrak{Z}_{23\mathcal{B}}^+$ 
 is simply connected, $\tilde\alpha$
is homotopic to the constant curve at the base point. Applying $\sigma$ then gives
a homotopy of $\alpha$ in  $\mathfrak{Z}_{23\mathcal{B}}$ to the constant
path and shows $\mathfrak{Z}_{23\mathcal{B}}$ is simply connected. This proves
Assertion~2.

Let $g$ be the generator of $\pi_2(\mathcal{Z}_{23\mathcal{B}})=\mathbb{Z}$.
Let $i$ be the inclusion of $\mathcal{Z}_{23\mathcal{B}}$ into
$\mathbb{R}^6-K=\mathbb{R}^3\times(\mathbb{R}^3-0)$ and let
$\pi$ be the projection of $\mathbb{R}^3\times(\mathbb{R}^3-0)$ onto
$\mathbb{R}^3-0$. We have a commutative diagram defined by $\pi_*\circ i_*$ in homotopy
$$\begin{array}{lllll}
\mathbb{Z}=\pi_2(\mathcal{Z}_{23\mathcal{B}})&\mapright{\approx}&
\mathbb{Z}=\pi_2(\mathbb{R}^3-0)\\
\ \ T_*\downarrow&\ \ \circ&\ \ T_*\downarrow\\
\mathbb{Z}=\pi_2(\mathcal{Z}_{23\mathcal{B}})&\mapright{\approx}&
\mathbb{Z}=\pi_2(\mathbb{R}^3-0)
\end{array}$$
We have $T(C_{12}{}^1,C_{22}{}^1,C_{22}{}^2)=(-C_{12}{}^1,C_{22}{}^1,-C_{22}{}^2)$.
Since $T$ is the composition of 2 reflections of $\mathbb{R}^3-0$, $T$ is homotopic
to the identity and thus $T_*$ is the identity on $\pi_2(\mathbb{R}^3-0)$. Consequently
$T_*$ is the identity on $\pi_2(\mathcal{Z}_{23\mathcal{B}})$. The long exact sequence
of the fibration given in Equation~(\ref{E5.b}) gives rise to a commutative diagram
$$\begin{array}{lllll}
\mathbb{Z}=\pi_2(\mathcal{Z}_{23\mathcal{B}})&\mapright{\approx}&
\mathbb{Z}=\pi_2(\mathfrak{Z}_{23\mathcal{B}})^+\\
\ \ T_*\downarrow&\ \ \circ&\ \ [T]_*\downarrow\\
\mathbb{Z}=\pi_2(\mathcal{Z}_{23\mathcal{B}})&\mapright{\approx}&
\mathbb{Z}=\pi_2(\mathfrak{Z}_{23\mathcal{B}})^+.
\end{array}$$
The Hurewicz isomorphism then shows
$[T]_*=\operatorname{id}\text{ on }H_2(\mathfrak{Z}_{23\mathcal{B}}^+)$.

Let $\Delta$ be a simplex in $\mathfrak{Z}_{23\mathcal{B}}$. There are two possible
lifts of $\Delta$ to $\mathfrak{Z}_{23\mathcal{B}}^+$; if $\Delta$ lies totally in $\mathfrak{F}$,
we take 2 copies of the canonical lift. We define $\sigma^*(\Delta)$ to be the formal sum of
these two lifts. This defines a chain map which induces a map in
homology denoted by $\sigma^*$ which satisfies
$$
\sigma_*\sigma^*=2\text{ and }\sigma^*\sigma_*=\operatorname{id}+[T]_*\,.
$$
Since $[T]_*$ is the identity, we also obtain $\sigma^*\sigma_*$ is multiplication by 2.
Consequently $\sigma_*$ is an isomorphism from 
$H_2(\mathfrak{Z}_{23\mathcal{B}}^+,\mathbb{R})$ to 
$H_2(\mathfrak{Z}_{23\mathcal{B}},\mathbb{R})$. Assertion~3
now follows via duality.\end{proof}

\begin{remark}\rm
If $\sigma$ was a covering projection, the proof of Assertion~3 that
we have given is just the usual application of transfer and induction in homology; this
is, for example, how the homology of real projective space with coefficients
in $\mathbb{R}$ is computed and exploits
the fact that we are dealing with a Mackey functor. There is
an additional subtlety here since the covering projection is in fact ramified that
needed to be dealt with. Note that we obtain no information concerning any
2-torsion in $H_2(\mathfrak{Z}_{23\mathcal{B}};\mathbb{Z})$ or, equivalently,
in $\pi_2(\mathfrak{Z}_{23\mathcal{B}};\mathbb{Z})$.
\end{remark}

\end{document}